\documentclass{amsart}
\usepackage{graphicx, tikz-cd}
\usepackage[T1]{fontenc}
\usepackage[toc, page]{appendix}
\usepackage{amssymb, amsfonts, latexsym, mathrsfs, amsmath, amsthm, bbm, amstext}
\usepackage{txfonts,bbding,wasysym,pifont,stmaryrd,tipa, bm, amsbsy}
\usepackage[all,cmtip]{xy}
\usepackage{float, mathtools, enumitem}
\usepackage[utf8]{inputenc}
\usepackage{hyperref}
\usepackage{cleveref}
\crefname{section}{§}{§§}
\Crefname{section}{§}{§§}
\usetikzlibrary{arrows,positioning,shapes.geometric,matrix}
\DeclarePairedDelimiter\floor{\lfloor}{\rfloor}

\newtheorem{theorem}{Theorem}[section]
\newtheorem{lemma}[theorem]{Lemma}
\newtheorem{proposition}[theorem]{Proposition}
\newtheorem{corollary}[theorem]{Corollary}

\newcommand*{\sheafhom}{\mathcal{H}\kern -.5pt om}
\newcommand*{\sheafext}{\mathcal{E}\kern -.5pt xt}
\newcommand*{\sheafend}{\mathcal{E}\kern -.5pt nd}

\theoremstyle{definition}
\newtheorem{definition}[theorem]{Definition}
\newtheorem{example}[theorem]{Example}

\theoremstyle{remark}
\newtheorem{remark}[theorem]{Remark}
\newtheorem{notation}[theorem]{Notation}
\numberwithin{equation}{section}

\newcommand{\Supp}{\text{\rm Supp}}

\newcommand{\Gal}{\text{\rm Gal}}
\newcommand{\Div}{\text{\rm Div}}
\newcommand{\Vect}{\text{\rm Vect}}

\newcommand{\Spec}{\text{\rm Spec}}

\DeclareMathAlphabet{\mathpzc}{OT1}{pzc}{m}{it}
 \DeclareMathOperator*{\Hom}{Hom}

\DeclareMathOperator{\sheafHom}{\mathscr{H}\text{\kern -3pt {\calligra\large om}}\,}

\begin{document}

\title{On the Covers of Orbifold Curves Preserving the Slope Stability under Pullback}
\author{Soumyadip Das}
\address{Mathematics,
Indian Institute of Technology Jammu, Jagti, NH 44, Jammu 181221, Jammu and Kashmir, India.}
\email{soumyadip.das@iitjammu.ac.in}


\begin{abstract}
We completely characterize the covers of connected orbifold curves which preserve slope stability of vector bundles under the pullback morphism. More precisely, given a cover $f \colon (Y,Q) \longrightarrow (X,P)$ of connected orbifold curves, we show that the maximal destabilizing sub-bundle of the pushforward sheaf $f_*\mathcal{O}_{(Y,Q)}$ defines the maximal \'{e}tale sub-cover of $f$. The cover $f$ is said to be genuinely ramified if $f$ does not factor through any non-trivial \'{e}tale sub-cover. Our main result states that the class of covers $f$ that preserves the stable bundles under a pullback are precisely the class of genuinely ramified covers $f$. Further, we establish equivalent conditions for the cover $f$ to be genuinely ramified, generalizing earlier works on covers of curves. We thoroughly study the slope stability conditions of bundles on an orbifold curve, their properties under the pushforward and pullback maps under covers with a stand point of Deligne-Mumford stacks, hence giving a solid foundation of the subject. As a consequence, we also answer the question of descent of stable bundles under genuinely ramified covers.
\end{abstract}

\maketitle

\section{Introduction}
Throughout this article, we work over an algebraically closed field $k$ of an arbitrary characteristic. We are interested in defining the slope stability conditions of vector bundles on an \textit{orbifold curve}, i.e. a one dimensional smooth proper Deligne-Mumford stack that is generically an integral curve. Further, we want to study the nature of the pullback and pushforward of a bundle under a finite surjective generically separable morphism, henceforth referred to as a \textit{cover} of orbifold curves.

First, consider the classical context of a cover $f \colon Y \longrightarrow X$ of smooth projective connected $k$-curves. The cover $f$ induces a homomorphism $f_* \colon \pi_1(Y) \longrightarrow \pi_1(X)$ of the \'{e}tale fundamental groups. The image of this homomorphism is a sub-group of $\pi_1(X)$ of finite index. This defines the (unique) maximal \'{e}tale sub-cover $Y' \longrightarrow X$ via which $f$ factors. One aspect of \cite{BP} is to capture this maximal \'{e}tale sub-cover using the maximal destabilizing sub-bundle of the pushforward bundle $f_*\mathcal{O}_Y$ of the structure sheaf $\mathcal{O}_Y$. The main result of loc. cit. says that the class of covers $f$ such that the pullback bundle $f^*E$ of a stable bundle $E$ remains stable is exactly the class of \textit{genuinely ramified covers} $f$, i.e. the covers $f$ which do not contain any non-trivial \'{e}tale sub-cover. In other words, these are the covers $f$ such that the induced homomorphism $f_*$ on the \'{e}tale fundamental groups is a surjection. These results are further generalized to higher dimensional normal varieties in \cite{BDP}, and studied in the context of certain formal orbifold curves in \cite{BKP} and in \cite{BKP2}.

In this paper, our objects of interest are the orbifold curves. By a folklore result, a connected orbifold curve is equivalent to a formal orbifold curve $(X,P)$ which consists of a smooth projective connected $k$-curve $X$ and a data $P$ of certain finite Galois field extensions associated finitely many closed points on $X$; see Definition~\ref{def_f_o_c} and Theorem~\ref{thm_equiv}. We should mention that there is a little gap in the existing literature that connects the covers of formal orbifold curves with the stacky notion of covers of orbifold curves (the later are affine morphisms that are necessarily representable in the sense of \cite[Section~7.2, page 155]{Olsson}); we adapt Definition~\ref{def_stacky_curve} for the covers of orbifold curves which are not necessarily representable morphisms; also see Remark~\ref{rmk_deg_cover}. In this regard, we need to identify the representable covers. We show in Lemma~\ref{lem_rep_cover} that the representable covers are precisely the ones induced by the pullback branch data $f_0^*P$; see Remark~\ref{rmk_pullback_of_branch_data} for the definition. The representable covers play a crucial role in our article. Since we work over a base field $k$ of an arbitrary characteristic, it is eminent to define and study the properties of the $P$-degree, $P$-slope and $P$-(semi/poly)stability of a bundle on an orbifold curve $(X,P)$, and then relate them to the existing theory of $\text{\rm char}(k) = 0$-case as well as of the equivariant notions; see Section~\Cref{sec_degree} and Section~\Cref{sec_stability_f_o_c}.

Using these results, we relate the maximal \'{e}tale sub-cover of a cover $f \colon (Y,Q) \longrightarrow (X,P)$ with the cover associated to the maximal destabilizing bundle of the vector bundle $f_*\mathcal{O}_{(Y,Q)}$, generalizing the results from \cite{BP}. More precisely, we have the following result.

\begin{theorem}[{Lemma~\ref{lem_algebra}, Proposition~\ref{prop_maximal_etal_subcover}, Remark~\ref{rmk_factorization_discussion}}]\label{thm_intro_1}
Let $f \colon (Y,Q) \longrightarrow (X,P)$ be a cover of connected orbifold curves. This induces a cover $f_0 \colon Y \longrightarrow X$ of smooth projective connected curves. The maximal destabilizing sub-bundle $V$ of $f_*\mathcal{O}_{(Y,Q)}$ is a $P$-semistable vector bundle of $P$-degree $0$, and the natural inclusion $\mathcal{O}_{(X,P)} \hookrightarrow V$ of vector bundles on $(X,P)$ equip $V$ with a structure of $\mathcal{O}_{(X,P)}$-algebras. Further, there is a maximal sub-cover $g_0 \colon \hat{X} \longrightarrow X$ of $f_0$ such that the induced cover $g \colon (\hat{X},g_0^*P) \longrightarrow (X,P)$ is the maximal \'{e}tale sub-cover of $f$, and $(\hat{X},g_0^*P) = \underline{\text{\rm Spec}}(V)$. The cover $f$ factors as a composition
$$ (Y,Q) \overset{\hat{g}}\longrightarrow (\hat{X}, g_0^*P) \overset{g}\longrightarrow (X,P),$$
and the homomorphism $\hat{g}_* \colon \pi_1(Y,Q) \rightarrow \pi_1(\hat{X},g_0^*P)$ of the \'{e}tale fundamental groups induced by $\hat{g}$ is a surjection.
\end{theorem}

As in the case of curves, we say that a non-trivial cover of connected orbifold curves is \textit{genuinely ramified} if it satisfies one of the following equivalent conditions; see \cite[Lemma~2.4, Proposition~2.5, Lemma~3.1]{BP} for smooth curves and \cite[Theorem~2.4]{BDP} for normal varieties.

\begin{proposition}[{Proposition~\ref{prop_gen_ram_equivalences}}]\label{prop_intro_1}
Let $f \colon (Y,Q) \longrightarrow (X,P)$ be a cover of connected orbifold curves. The following are equivalent.
\begin{enumerate}[leftmargin=*]
\item The maximal destabilizing sub-bundle $V$ of $f_* \mathcal{O}_{(Y,Q)}$ coincides with $\mathcal{O}_{(X,P)}$.
\item The cover $f$ does not factor through any non-trivial \'{e}tale sub-cover.
\item The homomorphism between \'{e}tale fundamental groups $f_* \colon \pi_1(Y,Q) \longrightarrow \pi_1(X,P)$ induced by $f$ (cf. \cite[Proposition~2.26]{KP}) is a surjection.
\item For any \'{e}tale cover $(Z,R) \longrightarrow (X,P)$ of connected orbifold curves, the fiber product stacky curve $(Y,Q) \times_{(X,P)} (Z,R)$ is connected.
\item The fiber product Deligne-Mumford stack $(Y,Q) \times_{(X,P)} (Y,Q)$ is connected.
\item $\text{\rm dim} \, H^0( (Y,Q) , f^* f_* \mathcal{O}_{(Y,Q)}) = 1$.
\end{enumerate}
Finally, the above conditions imply that the cover $f_0 \colon Y \longrightarrow X$ induced on the Coarse moduli curves is genuinely ramified.
\end{proposition}

Next, we provide a necessary and sufficient condition for a stable bundle to be a pullback bundle under a genuinely ramified cover.

\begin{theorem}[{\Cref{thm_descent}}]
Let $f \colon (Y,Q) \longrightarrow (X,P)$ be a genuinely cover of connected orbifold curves. This induces a cover $f_0 \colon Y \longrightarrow X$ of the underlying Coarse moduli curves. Let $E \in \Vect(Y,Q)$ be a $Q$-stable bundle. Then there exists a bundle $F \in \Vect(X,P)$ such that $E \cong f^*F$ if and only if $f_*E$ contains a $P$-stable sub-bundle $W$ satisfying $\mu_P(W) = \frac{\mu_Q(E)}{\text{\rm deg}(f_0)}$.
\end{theorem}

Our final result gives the complete characterization of covers $f \colon (Y,Q) \longrightarrow (X,P)$ of connected orbifold curves preserving slope stability conditions. Namely, we show that a cover of orbifold curves preserves the slope stability if and only if it is a genuinely ramified cover.

\begin{theorem}[{Theorem~\ref{thm_main}, Theorem~\ref{thm_converse}}]\label{thm_intro}
Let $f \colon (Y,Q) \longrightarrow (X,P)$ be a non-trivial cover of connected orbifold curves. The cover $f$ is genuinely ramified if and only if the pullback $f^*E$ of any $P$-stable bundle $E \in \Vect(X,P)$ is $Q$-stable.
\end{theorem}

The idea behind the proof of the forward direction is as follows. We show that for a representable genuinely ramified Galois cover $f$, the bundle $f^* \left( f_* \mathcal{O}_{(Y,f_0^*P)}/\mathcal{O}_{(X,P)} \right)$ admits a filtration by sub-bundles such that each successive quotient is a line bundle of negative $Q$-degree; see Proposition~\ref{prop_negative}. This can be used to show that the pullback of a $P$-stable bundle is $Q$-polystable and simple, and hence $f_0^*P$-stable (see Theorem~\ref{thm_Galois_main}). For the general case, we need to consider the Galois closure $\bar{f}_0 \colon \bar{Y} \longrightarrow X$ of $f_0$ and a certain geometric branch data $\bar{Q}$ on $\bar{Y}$. The induced cover $\bar{f} \colon (\bar{Y}, \bar{Q}) \longrightarrow (X,P)$ may not be genuinely ramified. Our approach is to use intricate techniques for the genuinely ramified part and the maximal \'{e}tale part of the cover $\bar{f}$ to establish the result; see Theorem~\ref{thm_main}.

For the converse, when $f$ is a non-trivial \'{e}tale cover, we explicitly construct a $P$-stable bundle $E$ such that $f^*E$ is not $f_0^*P$-stable. For this, the main step is to show the following: for any non-trivial finite group $G$, and a $G$-Galois cover $Z \longrightarrow X$ of smooth projective connected $k$-curves, there is a $G$-stable bundle on $Z$ which is not $H$-stable for any proper subgroup $H$ of $G$, together with the induced $H$-action. This also produces an alternate proof of \cite[Proposition~5.1]{BP}.

We mention that the slope stability of bundles on orbifold curve has been studied in \cite{BKP} using an equivariant set up and equivariant slope, and in \cite{BKP2} using parabolic bundles when $\text{\rm char}(k)=0$; see Remark~\ref{rmk_condition} for a comparison with the above result. We refrain from using the notion of parabolic bundles since an orbifold bundle is non uniquely determined by a parabolic structure in positive characteristic; this can be seen for example by noting that there are non-isomorphic line bundles on an Artin-Schreier stack producing isomorphic parabolic structure on the underlying line bundle on $\mathbb{P}^1$. This suggests that using orbifold curves and orbifold bundles in the stacky sense is a more natural and canonical way for studying the slope stability conditions.

The present article also acts as the foundation for our future work on the construction of the moduli space of the semistable orbifold bundles of a given rank and determinant bundle. We have also given a detailed account of the definition and properties of degree and slope of an orbifold bundles together with a study on their behavior under covers (see Lemma~\ref{lem_pullback_div_line_bundles}, Lemma~\ref{lem_push_div_line_rep}, Lemma~\ref{lem_push_div_line_nonrep}, Proposition~\ref{prop_properties_stacky}), for which we have not found any suitable source in the literature.

The structure of the paper is as follows. After introducing the notation and the conventions in Section~\Cref{sec_Not}, in Section~\Cref{sec_stacky_curves}, we recall the necessary definition for the objects of our interest. Section~\Cref{sec_degree} and Section~\Cref{sec_slope_stability} are devoted to the formalizing the definition and properties of degree and slope stability for vector bundles on an orbifold curve. Section~\Cref{sec_pushforward} studies the maximal \'{e}tale sub-covers of covers of orbifold curves. The equivalent conditions in Proposition~\ref{prop_intro_1} are proved in Section~\Cref{sec_equivalence_gen_ram}. In Section~\Cref{sec_main-direction}, we prove the characterization of the genuinely ramified maps as the stability preserving morphisms. \Cref{sec_equiv} is added to show a categorical equivalence between the orbifold curves and the formal orbifold curves, which further give an equivalence between the categories of bundles on them.

\section*{Acknowledgement}
I would like to thank Snehajit Misra for valuable discussions. I am indebted to Indranil Biswas, Manish Kumar, Souradeep Majumder and A. J. Parameswaran. I would also like to thank an anonymous referee for minutely reading an earlier version of the article and pointing our the subtle errors.

This work was completed while the author was an NBHM Post-doctoral Fellow.

\section{Preliminaries, Notation and Conventions}\label{sec_Not}
\subsection{Generalities}
Throughout this article, we work over an algebraically closed field $k$ of an arbitrary characteristic. In this paper, the curves we consider are reduced $k$-curves. For any $k$-scheme $W$, and a closed point $w \in W$, we denote the completion of the local ring $\mathcal{O}_{W,w}$ of $W$ at $w$ by $\widehat{\mathcal{O}}_{W,w}$. When this complete local ring is a domain, we set $K_{W,w}$ as the quotient field $\text{\rm QF}(\widehat{\mathcal{O}}_{W,w})$.

A \textit{cover} $f \colon Y \longrightarrow X$ of $k$-curves refers to a finite surjective morphism $f$ that is generically separable. For a finite group $G$, a $G$-\textit{Galois cover} $Y \longrightarrow X$ is a cover together with a $G$-action on $Y$ such that $G$ acts simply transitively on each generic geometric fiber. Any cover (in particular, a Galois cover) is \'{e}tale away from a set $B$ of finitely many points on the base curve $X$, which may be empty. The set $B$ is called the \textit{branch locus} of $f$.

For a $G$-Galois cover $f \colon Y \longrightarrow X$ of smooth projective connected $k$-curves, the group $G$ acts transitively on the fiber $f^{-1}(x) \subset Y$ for each closed point $x \in X$; the stabilizer groups at the points in $f^{-1}(x)$ are conjugate to each other in $G$. We define the \textit{inertia group} at a closed point $x \in X$ to be a stabilizer group $\text{\rm Stab}_G(y)$ for some $y \in f^{-1}(x)$. This is well defined as a subgroup of $G$ up to conjugations, hence well defined as an abstract group. In particular, the cover $f$ is \'{e}tale above $x \in X$ if and only if the inertia group at $x$ is the trivial group. When the order of the inertia group at $x$ is invertible in $k$, we say that the cover $f$ is \textit{tamely ramified} over the point $x$.

\subsection{Deligne-Mumford Stacks}
For the definition and properties of a Deligne-Mumford stack (a DM stack), we refer to \cite{Olsson}, \cite{DM} and \cite[Appendix A]{V}. Our interest is on the smooth proper one-dimensional DM stacks that are generically schematic.

In whatever follows, we adapt the following conventions.
\begin{itemize}[leftmargin=*]
\item We only consider DM stacks that are \emph{separated and of finite type} over $k$.
\item A \textit{representable morphism} $\mathfrak{Y} \longrightarrow \mathfrak{X}$ of DM stacks in this article is a morphism representable by a scheme, i.e. for any scheme $Z$ and a morphism $Z \longrightarrow \mathfrak{X}$ of stacks, the fiber product $\mathfrak{Y} \times_{\mathfrak{X}} Z$ is a scheme.
\end{itemize}

For any DM stack $\mathfrak{X}$, the diagonal morphism $\Delta_{\mathfrak{X}} \colon \mathfrak{X} \longrightarrow \mathfrak{X} \times_{\Spec(k)} \mathfrak{X}$ is a representable unramified morphism (see \cite[Proposition~7.15]{V}, \cite[Theorem~8.3.3]{Olsson}). The separatedness assumption on $\mathfrak{X}$ means that for any morphism $Y \longrightarrow \mathfrak{X} \times_{\Spec(k)} \mathfrak{X}$ of stacks where $Y$ is a scheme, the base change morphism $\mathfrak{X} \times_{\mathfrak{X} \times_{\Spec(k)} \mathfrak{X}} Y \longrightarrow Y$ of schemes is a proper (equivalently, finite) morphism. Further, by the definition of a DM stack, there exists an \'{e}tale surjective morphism $Z \longrightarrow \mathfrak{X}$ from a scheme $Z$ (the morphism $Z \longrightarrow \mathfrak{X}$ is called an \textit{atlas} of $\mathfrak{X}$). We say that $\mathfrak{X}$ is \textit{smooth} if there exists an atlas $Z \longrightarrow \mathfrak{X}$ where $Z$ is a smooth scheme (equivalently, for every atlas $Z' \longrightarrow \mathfrak{X}$, \, $Z'$ is a smooth scheme; \cite[Section~4, pg. 100]{DM}). We adapt \cite[Section~4, following Definition~4.10]{DM} for the definition of the properties of morphisms of DM stacks.

Every DM stack admits a Coarse moduli space by an algebraic space; see \cite[Definition~11.1.1 and Theorem~1.11.2]{Olsson}. For our purpose, we use the following definition of a Coarse moduli space.

\begin{definition}\label{def_KM}
A $k$-scheme $X$ is a \textit{Coarse moduli space} for a DM stack $\mathfrak{X}$ if there is a morphism $\pi \colon \mathfrak{X} \longrightarrow X$ satisfying the following conditions.
\begin{enumerate}
\item The morphism $\pi$ is initial among all morphisms from $\mathfrak{X}$ to $k$-schemes.\label{KM_1}
\item $\pi$ induces a bijective correspondence between the set of $k$-points of $X$ and the set $|\mathfrak{X}(k)|$ of isomorphism classes of $k$-points of $\mathfrak{X}$.\label{KM_2}
\item $X$ is separated and of finite type over $k$.\label{KM_3}
\item $\pi$ is a proper morphism of stacks, and $\pi_* \mathcal{O}_{\mathfrak{X}} = \mathcal{O}_X$.\label{KM_4}
\item (\cite[Theorem~11.3.6]{Olsson}) For any morphism $h \colon X' \longrightarrow X$ of $k$-schemes, the Coarse moduli scheme of the fiber product DM stack $X' \times_X \mathfrak{X}$ is universally homeomorphic to $X'$, and it is an isomorphic if $h$ is a flat morphism.\label{KM_5}
\end{enumerate}
\end{definition}

The notion of connectedness and irreducibility are well established for a DM stack $\mathfrak{X}$; see \cite[Section~4]{DM}. In particular, $\mathfrak{X}$ is a disjoint union of its connected components, and each connected component is a union of its irreducible components (\cite[Proposition~4.13, Proposition~4.15]{DM}) in a unique way. It can be seen from \cite[Proposition~4.14]{DM} that the DM stack $\mathfrak{X}$ is connected if and only if its Coarse moduli space $X$ is connected.

Let $\mathfrak{X}$ be a DM stack admitting a Coarse moduli scheme $X$. For any geometric point $x \in \mathfrak{X}(k)$, the fiber product $\mathfrak{X} \times_{\Delta_{\mathfrak{X}}, \mathfrak{X} \times_k \mathfrak{X}, (x,x)} \Spec(k) = \underline{\text{\rm Isom}}(x,x)$ is a constant group scheme over $\text{\rm Spec}(k)$ associated to a finite group $G_x$, called the \textit{stabilizer group} at $x$; as $x$ varies in its isomorphism class in $|\mathfrak{X}(k)| \cong X(k)$, the stabilizer groups are conjugate to each other. Thus, upto a canonical isomorphism, the stabilizer group $G_x$ at a point $x \in |\mathfrak{X}(k)|\cong X(k)$ is well defined as an abstract group. It follows that $G_x$ is the trivial group if and only if the image of $x$ in $\mathfrak{X}$ lies in a sub-scheme of $\mathfrak{X}$. A point $x \in \mathfrak{X}(k)$ is called a \textit{stacky point} if the stabilizer group $G_x$ is non-trivial.

\begin{example}\label{eg_quotient_stack}
One important example of a DM stack is a quotient stack (see \cite[Example~4.8]{DM} or \cite[Example~8.1.12]{Olsson}). Let $G$ be a finite group acting on a quasi-projective $k$-variety $Y$ such that the $G$-actions on the generic points are faithful. The quotient stack $[Y/G]$ is a DM stack admitting the $k$-variety $X \coloneqq Y/G$ as its Coarse moduli space, and the canonical morphism $Y \longrightarrow [Y/G]$ is an atlas. In particular, $[Y/G]$ is a smooth (respectively, proper) DM stack if and only if $Y$ is smooth (respectively, proper). Moreover, $Y \longrightarrow [Y/G]$ is a principal homogeneous space over $[Y/G]$, and the stack $[Y/G]$ is representable if and only if the map $Y \longrightarrow X$ of varieties is an \'{e}tale morphism. The isomorphism classes of the $k$-points of $[Y/G]$ and the closed points of $X$ are both canonically identified with the $G$-orbits of the closed points of $Y$.
\end{example}

\subsection{Stacky and Orbifold Curves}\label{sec_stacky_curves}
We are interested in the properties of morphisms between smooth proper connected DM stacks which are generically smooth $k$-curves. To study such morphisms, one encounters DM stacks that are not necessarily smooth. This leads us to make the following definitions.

\begin{definition}[{Stacky Curve}]\label{def_stacky_curve}
A connected DM stack $\mathfrak{X}$ over $k$ is said to be a \textit{stacky curve} if it satisfies the following properties.
\begin{enumerate}
\item $\mathfrak{X}$ admits a smooth irreducible $k$-curve $X$ as its Coarse moduli space (cf. Definition~\ref{def_KM}).
\item Every irreducible component of $\mathfrak{X}$ is one-dimensional and is generically isomorphic to $X$.
\end{enumerate}

A morphism $f \colon \mathfrak{Y} \longrightarrow \mathfrak{X}$ of stacky curves induce a morphism $f_0 \colon Y \longrightarrow X$ of their Coarse moduli curves; namely, the composite morphism $\mathfrak{Y} \overset{f}\longrightarrow \mathfrak{X} \longrightarrow X$ factors uniquely as the composition $\mathfrak{Y} \longrightarrow Y \overset{f_0} \longrightarrow X$ by Definition~\ref{def_KM}. We say that $f$ is a \textit{cover} of stacky curves if $f$ is a finite surjective morphism, and $f_0$ is a cover of curves. \hfill\qed
\end{definition}

Let $\mathfrak{X}$ be a stacky curve. For any $x \in \mathfrak{X}(k)$, the stabilizer group $G_x$ is trivial if and only if the image of $x$ lies in the maximal open sub-curve of $\mathfrak{X}$. Since $\mathfrak{X}$ is generically schematic, there are only finitely many stacky points in $\mathfrak{X}$.

\begin{definition}[{Orbifold Curve}]\label{def_orbi_curve}
A stacky curve $\mathfrak{X}$ is said to be an \textit{orbifold curve} if it is smooth (i.e. one, and hence every atlas of $\mathfrak{X}$ is a smooth $k$-curve).
\end{definition}

When $Y \longrightarrow X$ is a $G$-Galois cover of smooth $k$-curves for some finite group $G$ and $X$ is connected, the quotient stack $[Y/G]$ in Example~\ref{eg_quotient_stack} is an orbifold curve. It should be noted that stacky curves and orbifold curves are considered the same in some literature (e.g. in \cite{VZB}); but for our context, we distinguish them: \emph{an orbifold curve is a smooth stacky curve}. The following is a typical case in our work which needs this distinction.

\begin{example}\label{eq_non-smooth}
A stacky curve in our definition need not be smooth although its Coarse moduli curve is smooth. Take the example of two distinct lines in $\mathbb{A}^2$ intersecting at a point and $\mathbb{Z}/2$ acting on this union of lines via interchanging points. The corresponding quotient stacky curve is not smooth as it admits an atlas from the union of the above lines which is not smooth, but the Coarse moduli curve $\mathbb{A}^1$ is smooth.
\end{example}

We make the following remarks in regard to the above definitions.

\begin{remark}\label{rmk_deg_cover}
A cover of a stacky curve in Definition~\ref{def_stacky_curve} need not be a representable morphism: for any stacky curve $\mathfrak{X}$, the Coarse moduli map $\iota \colon \mathfrak{X} \longrightarrow X$ is a cover in the sense of Definition~\ref{def_stacky_curve} that is representable if and only if $\iota$ defines an $k$-isomorphism $\mathfrak{X} \cong X$.

Let $f \colon \mathfrak{Y} \longrightarrow \mathfrak{X}$ be a cover of stacky curves. In particular, $f$ is a finite morphism which may not be representable. It is defined analogous to the definition of a proper morphism in \cite[Definition~4.11]{DM}. So a cover is dominated by a cover $Z \longrightarrow \mathfrak{X}$ from a $k$-curve $Z$, and $Z \longrightarrow \mathfrak{Y}$ is again a cover. Since any stacky curve has a representable diagonal by our hypothesis, any morphism from a $k$-scheme to a stacky curve is representable.
\end{remark}

\begin{remark}\label{rmk_cover_curves}
We point out that a cover $f \colon \mathfrak{Y} \longrightarrow X$ of stacky curves with $X$ a $k$-curve is representable if and only if $\mathfrak{Y}$ is a $k$-curve. This is because $f$ factors as a composition
$$\mathfrak{Y} \longrightarrow Y \overset{f_0}\longrightarrow X$$
of covers where $Y$ is the Coarse moduli curve of $\mathfrak{Y}$, and the cover $\mathfrak{Y} \longrightarrow Y$ is representable if and only if $\mathfrak{Y} \cong Y$. If $f$ is also an \'{e}tale cover, then $f$ must be representable, and $\mathfrak{Y}$ is a $k$-curve. 
\end{remark}

A stacky curve $\mathfrak{X}$ is \'{e}tale locally a quotient stack (\cite[Theorem~11.3.1]{Olsson}). When $\mathfrak{X}$ is an orbifold curve, this local structure produces a finite data of certain Galois field extensions associated to finitely many closed points of the Coarse moduli curve $X$; the curve $X$ together with this finite data is called a formal orbifold curve (introduced in \cite{P} when $k = \mathbb{C}$; generalized over algebraically closed fields of arbitrary characteristic in \cite{KP}). By Theorem~\ref{thm_equiv}, the assignment of a formal orbifold curve to an orbifold curve is an equivalence of categories.

\begin{definition}[{\cite[Definition~2.1]{KP}}]\label{def_f_o_c}
A \textit{formal orbifold curve} is a pair $(X,P)$ where $X$ is a smooth $k$-curve and $P$ is a \textit{branch data}, i.e. a function that to every closed point $x \in X$ associates a finite Galois extension $P(x)$ of $K_{X,x} = \text{\rm QF}(\widehat{\mathcal{O}}_{X,x})$ (in some fixed separable algebraic closure of $K_{X,x}$) such that the set
$$\text{\rm BL}(P) \coloneqq \left\{ x \in X \hspace{.2cm} | \hspace{.2cm} P(x) \text{ is a nontrivial extension of } K_{X,x}  \right\},$$
called the \textit{branch locus} of $P$, is a finite set of closed points of $X$. A formal orbifold curve $(X,P)$ is said to be \textit{connected} (respectively, \textit{projective}) if the $k$-curve $X$ is connected (respectively, projective).
\end{definition}

Thus, a formal orbifold curve is a smooth curve $X$ together with the data of a finite set $B = \Supp(P)$ (which may be empty) of closed points in $X$, and for each $x \in B$, a finite Galois extension $P(x)/K_{X,x}$. For two branch data $P$ and $Q$ on a smooth projective curve $X$, we write $Q \geq P$ if $P(x) \subset Q(x)$ as extensions of $K_{X,x}$ for every closed point $x \in X$. Recall that when $e_x \coloneqq |\Gal\left( P(x)/K_{X,x} \right)|$ is invertible in $k$ for a closed point $x \in X$, we have $P(x) = K_{X,x}(a^{1/e_x})$, where $a$ is a uniformizer of $\widehat{\mathcal{O}}_{X,x}$, and hence the Galois extension $P(x)/K_{X,x}$ is uniquely determined by the degree $e_x$. Thus if $|\Gal\left( P(x)/K_{X,x} \right)|$ is invertible in $k$ for each closed point $x \in X$, the orbifold curve $(X,P)$ is completely determined by $X$ together with finitely many points, and a positive integer (that is invertible in $k$) attached to each of these points. We recall the following definitions from \cite{KP}.

\begin{definition}[{\cite[Definition~2.6]{KP}}]\label{def_foc_maps}
Let $X$ and $Y$ be two smooth $k$-curves with branch data $P$ and $Q$ on them, respectively.
\begin{enumerate}
\item A \textit{cover} $f \colon (Y,Q) \longrightarrow (X,P)$ of formal orbifold curves is a cover $f_0 \colon Y \longrightarrow X$ such that $P(f_0(y)) \subset Q(y)$ as extensions of $K_{X,f_0(y)} = \text{\rm QF}(\widehat{\mathcal{O}}_{X,f_0(y)})$ for each closed point $y \in Y$.
\item A cover $f \colon (Y,Q) \longrightarrow (X,P)$ of formal orbifold curves is said to be \textit{\'{e}tale} if $Q(y) = P(f_0(y))$ as extensions of $K_{X,f_0(y)}$ for each closed point $y \in Y$.
\item A cover $f_0 \colon Y \longrightarrow X$ of smooth curves is said to be \textit{essentially \'{e}tale} if $K_{Y,y} \subset P(f_0(x))$ as extensions of $K_{X,f_0(y)}$ for each closed point $y \in Y$.
\end{enumerate}
\end{definition}

\emph{In view of the equivalence of the categories in Theorem~\ref{thm_equiv}, we drop the adjective `formal', and simply write: $(X,P)$ is an orbifold curve. Whether we are using the local data on $(X,P)$ or the associate stacky curve, it will be clear from the context.}

\emph{In what follows, we will work over proper orbifold curves. Further, given a cover $f$ of orbifold curves, we always denote the induced cover of the Coarse moduli curves by $f_0$.}

We need the following useful observations from \cite{KP}.

\begin{remark}[{Pullback of a branch data}]\label{rmk_pullback_of_branch_data}
Given a cover $f_0 \colon Y \longrightarrow X$ of smooth projective curves, and a branch data $P$ on $X$, we can define the pullback branch data $f_0^*P$ on $Y$ as in \cite[Definition~2.5]{KP}: for any closed point $y \in Y$, the field $f_0^*P(y)$ is the compositum $P(f_0(y)) \cdot K_{Y,y}$ as an Galois extension of $K_{X,x}$. The cover $f_0$ induces a cover $f \colon (Y, f_0^*P) \longrightarrow (X,P)$ of orbifold curves in a natural way.

By \cite[Lemma~2.12]{KP}, there is a cover $f \colon (Y,Q) \longrightarrow (X,P)$ of orbifold curves if and only if $Q \geq f_0^*P$. Moreover, the induced cover $f \colon (Y, f_0^*P) \longrightarrow (X,P)$ of orbifold curves is \'{e}tale if and only if $f_0$ is an essentially \'{e}tale cover of $(X,P)$. So in general, a cover $f' \colon (Y,Q) \longrightarrow (X,P)$ of orbifold curves factors uniquely as a composition
$$f' \colon (Y,Q) \overset{j}\longrightarrow (Y, f_0^*P) \overset{f}\longrightarrow (X,P)$$
of covers of orbifold curves. \hfill{\qed}
\end{remark}

The orbifold curves which are quotient stacks deserve a distinction.

\begin{definition}[{\cite[Definition~2.28, Remark~2.7]{KP}}]\label{def_geometric}
A connected orbifold curve $(X,P)$ is said to be \textit{geometric} if there exists a Galois cover $Z \longrightarrow X$ of smooth projective connected $k$-curves for a finite group $G$, and $(X,P) = [Z/G]$. In this case, $P$ is called a \textit{geometric branch} data on $X$.
\end{definition}

Let us note some useful results in this regard.

\begin{remark}\label{rmk_geometric_results}
Let $(X,P)$ be a connected orbifold curve. The following hold.
\begin{enumerate}[leftmargin = *]
\item By the proof of \cite[Proposition~2.37]{KP}, there is a geometric branch data $Q$ on $X$ satisfying $Q \geq P$.\label{geo:1}
\item By \cite[Proposition~2.30]{KP}, there is a maximal geometric branch data $Q$ such that $P \geq Q$. It follows that any essentially \'{e}tale cover $Y \longrightarrow X$ of $(X,P)$ is also an essentially \'{e}tale cover of $(X,Q)$.\label{geo:2}
\end{enumerate}
\end{remark}

We also have an explicit description of an atlas of $(X,P)$ (see \cite[Proposition~2.30]{KP} for a higher dimensional analogue).

\begin{lemma}\label{lem_atlas}
Let $(X,P)$ be a connected orbifold curve. Suppose that $\text{\rm BL}(P) = \{x_1, \ldots, x_r\}$, $r \geq 1$. There is a Zariski open covering $\{U_i\}_{0 \leq i \leq r}$ of $X$ such that the following hold.
\begin{enumerate}[leftmargin = *]
\item $U_0 = X - \text{\rm BL}(P)$ is the maximal open sub-curve of $(X,P)$,
\item $x_i \in U_i$, and $x_j \not\in U_i$ for any $i \neq j$,
\item for each $1 \leq i \leq r$, there is a $G_i \coloneqq \text{\rm Gal}\left(P(x_i)/K_{X,x_i}\right)$-Galois cover $f_i \colon V_i \longrightarrow U_i$ of smooth irreducible affine $k$-curves, \'{e}tale away from $x_i$, such that the following hold.
\begin{enumerate}
\item The cover $f_i$ is totally ramified over $x_i$, i.e. $f_i^{-1}(x_i) = \{v_i\}$, and $K_{V_i,v_i} \cong P(x_i)$ as $G_i$-Galois extensions of the local field $K_{U_i,x_i}$.
\item $U_i \times_X (X,P) \cong (U_i, P|_{U_i}) \cong [V_i/G_i]$ where $P|_{U_i}$ denote the restriction of the branch data $P$ on $U_i$.
\end{enumerate} 
\end{enumerate}
This produces a natural atlas $\sqcup_{0 \leq i \leq r} V_i \longrightarrow (X,P)$ with $V_0 = U_0$.
\end{lemma}

\begin{proof}
One can use the formal patching techniques of Harbater as in \cite[Theorem~4.3]{Das} to obtain the result. We provide an alternate proof.

Let $1 \leq i \leq r$. There is a natural number $d_0$ such that for any $d \geq d_0$, the complete linear system $\mathbb{P}^N = |dx_i|$ is very ample, and $N \geq 3$. Fix a $d_i \geq d_0$ that is coprime to $|G_i|$. Let $H_i$ be the hyperplane in the linear system $|d_ix_i|$ defined by $d_ix_i = 0$. Then $H_i \cap X = \{x_i\}$. Since we work over an infinite field $k$, we can choose a $H'_i \cong \mathbb{P}^1 \subset H_i$ such that $H'_i \cap X = \emptyset$. Then the projection map $p_i \colon |d_i x_i| - H'_i \longrightarrow \mathbb{P}^1$ restricts to a proper separable morphism, and hence a cover $g_i \colon X \longrightarrow \mathbb{P}^1$ of degree $d_i$ such that $g_i^{-1}(\infty) = \{x_i\}$. The field extension $K_{X,x_i}/K_{\mathbb{P}^1,\infty}$ is an extension of degree $d_i$.

Identifying the local field $K_{X,x_i}$ with $k((t^{-1})) \cong K_{\mathbb{P}^1, \infty}$ by the choice of a local parameter at $x_i$, we view the extension $P(x_i)/K_{X,x_i}$ as $P(x_i)/k((t^{-1}))$. Now by \cite[Main Theorem~1.4.1]{Katz}, there is a $G_i$-Galois cover $Z_i \longrightarrow \mathbb{P}^1$ that is \'{e}tale away from $\{0, \infty\}$, tamely ramified over $0$, and there is a unique point $z_i \in Z_i$ over $\infty$ such that $K_{Z,z_i} \cong P(x_i)$ as extensions of $k((t^{-1})) \cong K_{\mathbb{P}^1,\infty}$. Since $d_i$ was chosen to be coprime to $|G_i|$, the field extension $K_{X,x_i}/K_{\mathbb{P}^1,\infty}$ obtained from the cover $g_i$ is linearly disjoint to the extension $K_{Z,z_i}/K_{\mathbb{P}^1,\infty}$. Let $W_i$ be the connected component in the normalization of $X \times_{g_i,\mathbb{P}^1} Z_i$ containing the point $(x_i,z_i)$. Then the projection $f_i \colon W_i \longrightarrow X$ is a $G_i$-Galois cover, and there is a unique point $w_i \in f_i^{-1}(x_i)$ such that $K_{W_i,w_i} \cong P(x_i)$ as extensions over $K_{X,x_i}$.

For each $1 \leq i \leq r$, let $U_i$ be the open sub-curve of $X$ obtained by removing the branched points of $f_i$ other than $x_i$, as well as removing $\text{\rm BL}(P) - \{x_i\}$. The covers $f_i \colon V_i \coloneqq h_i^{-1}(U_i) \longrightarrow U_i$ have the stated properties.
\end{proof}

We conclude this section with the following remarks on the \'{e}tale fundamental groups and vector bundles on orbifold curves.

\begin{remark}[{The \'{e}tale fundamental group}]\label{rmk_etale_fundamental_group}
For a connected orbifold curve $(X,P)$, the \'{e}tale fundamental group is studied in \cite[Section~2.2]{KP} in the sense of formal orbifolds and in \cite[Section~4]{Noohi} in the sense of DM stacks. The category $\text{\rm \'{E}t}_{(X,P)}$ of \'{e}tale covers of $(X,P)$ is a Galois category (see \cite[\href{https://stacks.math.columbia.edu/tag/0BMY}{Definition 0BMY}]{SP} for the definition of a Galois category). Also note that since a cover $(Y,Q) \longrightarrow (X,P)$ is \'{e}tale if and only if the cover $f_0 \colon Y \longrightarrow X$ is an essential \'{e}tale cover of $(X,P)$, and $Q = f_0^*P$ (see Lemma~\ref{lem_rep_etale_cover} and Remark~\ref{rmk_pullback_of_branch_data}), the connected objects $(Y,Q) \in \text{\rm \'{E}t}_{(X,P)}$ are precisely the essentially \'{e}tale covers $f_0 \colon Y \longrightarrow X$ with $Y$ connected. The \'{e}tale fundamental group $\pi_1(X,P)$ is defined as the colimit
$$\pi_1(X,P) \coloneqq \underset{i \in I}{\varprojlim} \, \text{\rm Aut}(Y_i/X)$$
where the indexing set is take over all the essentially \'{e}tale connected covers $Y_i \longrightarrow X$ of $(X,P)$. The above profinite group can also be defined only using Galois essentially \'{e}tale connected covers using \cite[Remark~2.7]{KP}. We refer to \cite[Section~2.2]{KP} for the results concerning the \'{e}tale fundamental groups and their homomorphisms induced by covers.
\end{remark}

\begin{remark}[{Vector bundles}]\label{rmk_bundles}
We follow \cite[Definition 7.18]{V} for the notion of coherent sheaves and vector bundles. A vector bundle (respectively, a quasi-coherent or a coherent sheaf) on a proper stacky curve $\mathfrak{X}$ is the data given by a vector bundle (respectively, a quasi-coherent or a coherent sheaf) on each atlas that satisfy certain co-cycle conditions. The structure sheaf $\mathcal{O}_{\mathfrak{X}}$ is the quasi-coherent sheaf defined by associating the structure sheaf $\mathcal{O}_Z$ for every atlas $Z$ of $\mathfrak{X}$. It should be noted that a quasi-coherent sheaf in the above sense is actually a quasi-coherent sheaf of $\mathcal{O}_{\mathfrak{X}}$-modules as in \cite[Definition~9.1.14, Proposition~9.1.15]{Olsson}. For any morphism $f \colon \mathfrak{Y} \longrightarrow \mathfrak{X}$ of stacky curves, we have the functors
$$\Vect(\mathfrak{Y}) \overset{f^*} \longrightarrow \Vect(\mathfrak{X}) \hspace{.5cm} \text{and} \hspace{.5cm} \Vect(\mathfrak{X}) \overset{f^*} \longrightarrow \Vect(\mathfrak{Y})$$
of the categories of vector bundles (\cite[Section~9.2.5. pg. 198 and Section~9.3.]{Olsson}; defined up to a canonical natural isomorphism for the choice of charts). Since the category of vector bundles on an orbifold curve in the sense of stacks coincides with the category of vector bundles (see \cite[Definition~4.4, Definition~4.5]{KM}) defined on the corresponding formal orbifold curve by Theorem~\ref{thm_equiv_bundles}, our identification of an orbifold curve with its corresponding formal orbifold curve is consistent with the above notions. It should be mentioned that a pullback is defined in \cite[Definition~5.16]{KM} only for \'{e}tale covers in the sense of formal orbifold curves, and a pushforward is defined in \cite[Definition~5.26]{KM}.
\end{remark}

\section{Pullback and Pushforward Maps}
We retain the notation and conventions from the previous section. In particular, given any cover $f \colon \mathfrak{Y} \longrightarrow \mathfrak{X}$ of stacky curves, we always denote the induced cover (cf. Definition~\ref{def_stacky_curve}) on the respective Coarse moduli curves by $f_0$.
\subsection{Representable Covers}
In the following, we address the question of representability of a cover of orbifold curves (see Remark~\ref{rmk_cover_curves}).

Given any cover $f \colon (Y,Q) \longrightarrow (X,P)$ of orbifold curves, we noted that $Q \geq f_0^*P$, and hence $f$ factors as a composition of the cover $(Y,Q) \longrightarrow (Y,f_0^*P)$ induced by $\text{\rm id}_Y$, followed by the induce cover $(Y,f_0^*P) \longrightarrow (X,P)$. We show that $f$ is a representable morphism (in the sense of a morphism of stacks) if and only if $Q = f_0^*P$; in particular, this shows that every cover has a maximal representable sub-cover. We also show that any \'{e}tale cover is representable. This will later enable us to identify the maximal \'{e}tale sub-cover only by considering representable covers. Further, we show that the pushforward functor $f_*$ under a representable cover $f$ is exact -- this will allow us to understand the pushforward of bundles more explicitly.

\begin{lemma}\label{lem_rep_etale_cover}
Let $f \colon \mathfrak{Y} \longrightarrow (X,P)$ be an \'{e}tale cover of connected stacky curves where $(X,P)$ is an orbifold curve. Then the induced cover $f_0 \colon Y \longrightarrow X$ of the Coarse moduli curves is an essentially \'{e}tale cover of $(X,P)$, and $\mathfrak{Y} = (Y, f_0^*P)$; see Definition~\ref{def_foc_maps} and Remark~\ref{rmk_pullback_of_branch_data}.
\end{lemma}

\begin{proof}
In view of Remark~\ref{rmk_cover_curves}, we may assume that $P$ is a non-trivial branch data. Consider an atlas $V \longrightarrow (X,P)$ and the corresponding fiber product DM stack $\mathfrak{Y} \times_{(X,P)} V$. Then the projection morphism $\mathfrak{Y} \times_{(X,P)} V \longrightarrow V$ is an \'{e}tale cover of $V$. So $\mathfrak{Y} \times_{(X,P)} V$ is a smooth $k$-curve. As $\mathfrak{Y} \times_{(X,P)} V \longrightarrow \mathfrak{Y}$ is an atlas, $\mathfrak{Y}$ is an orbifold curve. Thus $\mathfrak{Y} = (Y,Q)$ for some branch data $Q$ on $Y$. Since $f$ is an \'{e}tale cover, by Remark~\ref{rmk_pullback_of_branch_data}, $f_0 \colon Y \longrightarrow X$ is an essentially \'{e}tale cover of $(X,P)$, and $Q = f_0^*P$.
\end{proof}

\begin{lemma}\label{lem_rep_cover}
Let $f \colon (Y,Q) \longrightarrow (X,P)$ be a cover of connected orbifold curves. Then $f$ is a representable morphism of DM stacks if and only if $Q = f_0^*P$. In particular, every \'{e}tale cover is representable.
\end{lemma}

\begin{proof}
Using Remark~\ref{rmk_cover_curves}, we may suppose that $P$ is a non-trivial branch data $P$. Consider an atlas $g \colon V \longrightarrow (X,P)$. Since the projection morphism $p_1 \colon (Y,Q) \times_{(X,P)} V \longrightarrow (Y,Q)$ is an \'{e}tale cover of stacky curves, Lemma~\ref{lem_rep_etale_cover} says that $(Y,Q) \times_{(X,P)} V$ is an orbifold curve of the form $(Z, p_{1,0}^*Q)$ where $Z$ is the Coarse moduli curve, and $p_{1,0} \colon Z \longrightarrow Y$ is an essentially \'{e}tale cover of $(Y,Q)$ (the proof of the said lemma is applicable also when the orbifold curves are not necessary connected). By the universal property of the normalization, the curve $Z$ is the normalization of $Y \times_X V$. We can summarize the above observation in the following Cartesian square:
\begin{equation}\label{eq_Cart_atlas}
\begin{tikzcd}
(Z, p_{1,0}^*Q) = (Y,Q) \times_{(X,P)} V \arrow[r, "p_2"] \arrow[d, "p_1"] \arrow[dr, "\square", phantom] & V \arrow[d, "g"] \\
(Y,Q) \arrow[r, "f"] & (X,P)
\end{tikzcd}
\end{equation}

Now since $g \colon V \longrightarrow (X,P)$ is an atlas, $f$ is representable if and only if $p_2$ is representable. But the later statement is equivalent to $(Y,Q) \times_{(X,P)} V$ being a $k$-curve by Remark~\ref{rmk_cover_curves}. We conclude that $f$ is representable if and only if $(Y,Q) \times_{(X,P)} V$ is the normalization of the fiber product curve $Y\times_X V$.

By our above discussion, the representability of the cover $f$ is equivalent to the following condition: for any closed points $x \in X$, points $y \in Y, \, v \in V$ lying over $x$, and any irreducible component $W$ of $Z$ containing the point $w = (y,z)$, we have
\begin{equation}\label{eq_ex1}
p_{1,0}^*Q(w) = K_{W,w} = \text{\rm QF}(\widehat{\mathcal{O}}_{W,w}).
\end{equation}
As the field $K_{W,w} = \text{\rm QF}(\widehat{\mathcal{O}}_{W,w})$ is the compositum $K_{Y,y} \cdot K_{V,v} \cong K_{Y,y} \cdot P(x)$, the equality~\eqref{eq_ex1} is equivalent to the following:
\[
p_{1,0}^*Q(w) = Q(y) = K_{Y,y} \cdot P(x) = f_0^*P(y).
\]
So we conclude that $f$ is representable if and only if $Q = f_0^*P$.

Finally, if $f$ is an \'{e}tale cover, by Lemma~\ref{lem_rep_etale_cover}, $Q = f_0^*P$, and hence $f$ is a representable morphism.
\end{proof}

The following observations will be useful later.

\begin{remark}\label{rmk_natural_atlas_on_pullback}
\begin{enumerate}[leftmargin=*]
    \item If $f_0 \colon Y \longrightarrow X$ is a cover of smooth projective connected $k$-curves and $P$ is a branch data on $X$, then the fiber product $Y \times_X (X,P)$ is a reduced stacky curve whose normalization is $(Y,f_0^*P)$. To see this, consider an atlas $V \longrightarrow (X,P)$. Then $Y \times_X V \longrightarrow Y \times_X (X,P)$ is an atlas where the fiber product $Y \times_X V$ is a reduced $k$-curve. Now use the universal property of the normalization.\label{rmk:1}
    \item If $(X,P)$ is a geometric orbifold curve (see Definition~\ref{def_geometric}), then $(Y,f_0^*P)$ is again a geometric orbifold curve. Indeed, if $(X,P)$ is a quotient stack $[Z/G]$ for a finite group $G$ and a smooth projective connected $k$-curve $Z$, then $(Y,f_0^*P) \cong [W/G]$ where $W$ is any irreducible component of the fiber product $[Y \times_X Z]$.\label{rmk:2}
    \item Let $f \colon (Y, f_0^*P) \longrightarrow (X,P)$ be a representable cover. Taking the natural atlas $V = \sqcup_{0 \leq i \leq r} V_i \longrightarrow (X,P)$ from Lemma~\ref{lem_atlas}, as in the proof of the above lemma, the fiber product stack $(Y,f_0^*P) \times_{(X,P)} V$ is a smooth curve that is isomorphic to the disjoint union of the respective normalization of $Y \times_X V_i$. Then it follows that for any closed point $y \in Y$, the stabilizer group $H_y$ at $y$ is a sub-group of the stabilizer group $G_{f_0(y)}$ at $f_0(y)$; further, there are $|G_{f_0(y)}|/|H_y|$ many points in the curve $(Y,f_0^*P) \times_{(X,P)} V$ lying over $y$.\label{rmk:3}
\end{enumerate}
\end{remark}

\begin{lemma}\label{lem_push_exactness_rep_cover}
Let $f \colon (Y,f_0^*P) \longrightarrow (X,P)$ be a (representable) the cover. The following hold.
\begin{enumerate}
\item The pushforward functor $f_*$ is an exact functor of coherent sheaves of modules.\label{exact:1}
\item If $f_0$ is also $\Gamma$-Galois for some finite group $\Gamma$, and $E \in \text{\rm Vect}(X,P)$, then every $\Gamma$-equivariant sub-bundle of $f^*E$ is of the form $f^*F$ for a unique sub-bundle $F \subset E$.\label{ds}
\end{enumerate}
\end{lemma}

\begin{proof}
By \cite[Theorem~11.6.1]{Olsson}, for any coherent sheaf $F$ of $\mathcal{O}_{(Y, f_0^*P)}$-modules, $\mathrm{R}^if_* F$ is a coherent sheaf of $\mathcal{O}_{(X,P)}$-modules for all $i \geq 0$. Consider any atlas $g \colon V \longrightarrow (X,P)$. As in the proof of Lemma~\ref{lem_rep_cover}, the fiber product $(Y,f_0^*P) \times_{(X,P)} V$ is a smooth $k$-curve that is isomorphic to the normalization of $Y \times_X V$, the projection morphism $p_1 \colon (Y,f_0^*P) \times_{(X,P)} V \longrightarrow (Y,f_0^*P)$ is an atlas, and the Cartesian square~\eqref{eq_Cart_atlas} with $Q = f_0^*P$ becomes:
\begin{equation}\label{eq_geom_sq}
\begin{tikzcd}
(Y,f_0^*P) \times_{(X,P)} V \arrow[r, "p_2"] \arrow[d, "p_1"] \arrow[dr, "\square", phantom] & V \arrow[d, "g"] \\
(Y,f_0^*P) \arrow[r, "f"] & (X,P)
\end{tikzcd}
\end{equation}
Since $p_2$ is a cover of $k$-curves, the functor $p_{2,*}$ on the coherent sheaves of modules is exact by \cite[\href{https://stacks.math.columbia.edu/tag/03QP}{Proposition 03QP}]{SP}. By Proposition~\ref{prop_proj_bc}~\eqref{bc}, for any coherent sheaf $F$ of $\mathcal{O}_{(Y,f_0^*P)}$-modules, and $i \geq 0$, we have
\[
g^* \mathrm{R}^if_*F \cong \mathrm{R}^ip_{2,*} p_1^* F.
\]
In particular, $\mathrm{R}^i f_* F = 0$ for all $i \geq 1$. So  $f_*$ is an exact functor, proving~\eqref{exact:1}.

Now suppose that $f_0$ is a $\Gamma$-Galois cover. Let $E' \subset f^*E$ be a $\Gamma$-equivariant sub-bundle where $f^*E$ has the natural $\Gamma$-equivariant structure coming from the cover $f$. Let $g \colon V \longrightarrow (X,P)$ be any atlas. As in the first part, the fiber product $(Y,f_0^*P) \times_{(X,P)} V$ is isomorphic to the normalization of $Y \times_X V$, and the projection morphism $p_2$ is also a $\Gamma$-Galois cover of smooth $k$-curves. The bundle $p_2^*g^*E \cong p_1^*f^*E$ is $\Gamma$-equivariant with its $\Gamma$-equivariant sub-bundle $p_1^*E' \subset p_1^*f^*E$. By the descent of sub-bundles under a Galois flat morphism of schemes, there is a unique (up to a canonical isomorphism) sub-bundle $F' \subset g^*E$ such that $p_2^*F' \cong p_1^*E'$. Since the above holds for every choice of atlas $g$ of $(X,P)$ in a compatible way, and the pullback under $f^*$ is defined up to canonical isomorphisms, we obtain a sub-bundle $F \subset E$ such that $f^*F \cong E'$; hence, the statement~\eqref{ds}.
\end{proof}

\subsection{Divisor and Degree}\label{sec_degree}
We retain the previous notion and conventions. We will only consider proper connected orbifold curves, unless otherwise specified.

Let us review the theory of Weil divisors on orbifold curves; see \cite[Section~5.4]{VZB} for an exposition. After defining the rank and the degree of a vector bundle on an orbifold curve, we study their properties via the pullback and the pushforward under a cover. Such notions are not well studied when $\text{\rm char}(k) > 0$.

Let $(X,P)$ be an orbifold curve together with its Coarse moduli morphism $\iota \colon (X,P) \longrightarrow X$. By Definition~\ref{def_KM}, the set $|(X,P)(k)|$ of isomorphism classes of points in $(X,P)(k)$ are in a bijective correspondence with the set $X(k)$ of closed points in $X$. For each $x \in X(k)$, there is a unique strictly full subcategory $\mathfrak{Z}_x \subset (X,P)$ such that $\mathfrak{Z}_x$ is a $0$-dimensional closed irreducible reduced sub-stack of $(X,P)$, $|\mathfrak{Z}_x(k)|$ is singleton which maps to $x$ via $|\mathfrak{Z}_x(k)| \longrightarrow |(X,P)(k)| \cong X(k)$ (see \cite[\href{https://stacks.math.columbia.edu/tag/06ML}{Section 06ML}]{SP} and \cite[\href{https://stacks.math.columbia.edu/tag/06RD}{Lemma 06RD}]{SP}). This sub-stack $\mathfrak{Z}_x$ is called the \textit{residual gerbe} at $x$. More precisely, the residual gerbe of every closed point in $X - \text{\rm BL}(P)$ is itself, and for a stacky point $x \in \text{\rm BL}(P)$, we can see that $\mathfrak{Z}_x$ is isomorphic to the closed classifying sub-stack $BG_x$ of $(X,P)$ where $G_x$ is the stabilizer group $G_x = \text{\rm Gal}\left( P(x)/K_{X,x} \right)$ at $x$, and as usual, $K_{X,x_i} = \text{\rm QF} (\widehat{\mathcal{O}}_{X,x})$. So each closed point $x \in X$ bijectively corresponds to the residual gerbe $\mathfrak{Z}_x$.

\emph{With an abuse of notation, we also denote the residual gerbe $\mathfrak{Z}_X$ at a closed point $x \in X$ also by $x$.}

As residual gerbes should be treated as fractional points, \cite[Remark~5.2.3]{VZB} defines the $P$-degree of the point $x \in X(k)$ to be $\frac{1}{|G_x|} = \frac{1}{[P(x) \colon K_{X,x}]} \in \mathbb{Q}$. Clearly, this is integral if and only if $x$ is not a stacky point.

A \textit{(Weil) divisor} on $(X,P)$ is a finite formal sum of the residual gerbes of $(X,P)$. In other words, a divisor is an element of the free abelian group $\Div(X,P)$, generated by the residual gerbes of $(X,P)$. As an abstract group, we have $\text{\rm Div}(X,P) \cong \text{\rm Div}(X)$. The $P$-\textit{degree} $\text{\rm deg}_P(D)$ of a divisor $D = \sum\limits_x n_x x \in \Div(X,P)$ is defined $\mathbb{Z}$-linearly as
$$\text{\rm deg}_P(D) \coloneqq \sum\limits_x \frac{n_x}{[P(x) \colon K_{X,x}]}.$$
The \textit{support} of a divisor $D = \sum n_x x$ is the $0$-dimensional closed sub-stack of $(X,P)$ given by the union of gerbes $x$ such that $n_x \neq 0$. We say that a divisor $D = \sum n_x x \in \text{\rm Div}(X,P)$ is \textit{effective} if $n_x \geq 0$ for all $x$.

For any cover $f \colon (Y,Q) \longrightarrow (X,P)$ of connected orbifold curves, we define the pullback homomorphism
\begin{equation}\label{def_pull_divisor}
f^* \colon \Div(X,P) \longrightarrow \Div(Y,Q)
\end{equation}
as a $\mathbb{Z}$-linear map, given by
$$f^*(x) \coloneqq \sum_{f(y) = x} [Q(y) \colon P(x)] \, y$$
on a residual gerbe $x$. Since $(X,P)$ is generically isomorphic to $X$, any principal divisor on $(X,P)$ is of the form $\iota^* \text{\rm div}(\phi)$ for some non-zero rational function $\phi \in k(X)^*$ (see \cite[Lemma~4.1]{Kobin}). This defines the notion of a \textit{linear equivalence} for divisors on $(X,P)$ (\cite[Definition~5.4.2]{VZB}), and one can associate a line bundle $\mathcal{O}_{(X,P)}(D) \in \text{\rm Pic}(X,P)$ to a divisor $D$ on $(X,P)$ as follows. An effective divisor $D$ defines a closed sub-stack of $(X,P)$, which we also denote by $D$. Let $\mathcal{O}_{(X,P)}(-D)$ be its ideal sheaf. This corresponds to the exact sequence
\begin{equation}\label{ideal_sheaf}
0 \longrightarrow \mathcal{O}_{(X,P)}(-D) \longrightarrow \mathcal{O}_{(X,P)} \longrightarrow \mathcal{O}_D \longrightarrow 0
\end{equation}
of coherent sheaves of $\mathcal{O}_{(X,P)}$-modules, where $\mathcal{O}_D$ is the pushforward of the structure sheaf on $D$ via the closed immersion $D \hookrightarrow (X,P)$. We set
$$\mathcal{O}_{(X,P)}(D) \coloneqq \sheafhom_{\text{\rm Vect}(X,P)}(\mathcal{O}_{(X,P)}(-D), \mathcal{O}_{(X,P)}) \cong \mathcal{O}_{(X,P)}(-D)^{-1} \in \text{\rm Pic}(X,P).$$
As any divisor $D$ on $(X,P)$ is uniquely written as $D = D_1 - D_2$ for effective divisors $D_1, \, D_2$ with disjoint supports,
\[
\mathcal{O}_{(X,P)}(D) \coloneqq \mathcal{O}_{(X,P)}(D_1) \otimes \mathcal{O}_{(X,P)}(-D_2) \in \text{\rm Pic}(X,P).
\]

The following result is of importance to us.

\begin{lemma}[{\cite[Lemma~5.4.5]{VZB}}]\label{lem4.2}
Every line bundle $L \in \text{\rm Pic}(X,P)$ is of the form $\mathcal{O}_{(X,P)}(D)$ for some divisor $D\in \text{\rm Div}(X,P)$. Moreover, $\mathcal{O}_{(X,P)}(D) \cong \mathcal{O}_{(X,P)}(D')$ if and only if $D$ and $D'$ are linearly equivalent.
\end{lemma}

We define the \textit{degree of a line bundle} as follows.

\begin{definition}\label{def_deg_lb_DM}
Let $L$ be a line bundle on an orbifold curve $(X,P)$. Then $L \cong \mathcal{O}_{(X,P)}(D)$ for a divisor $D$ on $(X,P)$, unique up to a linear equivalence (Lemma~\ref{lem4.2}). Define the $P$-degree of $L$ to be
$$\text{\rm deg}_P(L) \coloneqq \text{\rm deg}_P(D).$$
\end{definition}

Now we define the degree and rank of a bundle. Let $E \in \text{\rm Vect}(X,P)$. As before, $\iota \colon (X,P) \longrightarrow X$ stands for the Coarse moduli morphism.

The \textit{rank} $\text{\rm rk}(E)$ of $E$ is defined to be the rank of the bundle $\iota_* E$. Since the rank is a generic property, and $(X,P)$ is generically isomorphic to $X$, our notion is well defined. More precisely, for any atlas $u \colon U \longrightarrow (X,P)$, we have
$$\text{\rm rk}(u^*E) = \text{\rm rk}((\iota \circ u)^* \iota_*E) = \text{\rm rk}(\iota_*E) = \text{\rm rk}(E).$$

To $E$, we naturally associated the \textit{determinant line bundle} by $\text{\rm det}(E) \coloneqq \wedge^{\text{\rm rk}(E)} E \in \text{\rm Pic}(X,P)$. The $P$-\textit{degree} of $E$ is defined as
$$\text{\rm deg}_P(E) \coloneqq \text{\rm deg}_P\left( \text{\rm det} \left(E\right) \right).$$
In the following, we summarize the properties of divisors and vector bundles under a cover.

\begin{lemma}\label{lem_pullback_div_line_bundles}
Let $f \colon (Y,Q) \longrightarrow (X,P)$ be a cover of connected orbifold curves. As usual, write $f_0 \colon Y \longrightarrow X$ for the cover induced on the Coarse moduli curves. Then $f$ induces a homomorphism
$$f^* \colon \Div(X,P) \longrightarrow \Div(Y,Q)$$
that takes principal divisors to principal divisors. This defines a homomorphism
$$f^* \colon \text{\rm Pic}(X,P) \longrightarrow \text{\rm Pic}(Y,Q)$$
which coincides with the usual pullback of a line bundle as a coherent sheaf on a DM stack (see \cite[Section~9.3., page 203]{Olsson}). Under this map, for any $L \in \text{\rm Pic}(X,P)$, we have
$$\text{\rm deg}_Q f^*L = \text{\rm deg}(f_0) \cdot \text{\rm deg}_P L.$$
Further, let $E$ be a bundle on $(X,P)$ of rank $n$. The pullback coherent sheaf $f^*E$ is a bundle of rank $n$, and $\text{\rm det}(f^*E) \cong f^* \text{\rm det}(E)$. In particular,
\begin{eqnarray*}
\text{\rm deg}_Q (f^*E) = \text{\rm deg}(f_0) \cdot \text{\rm deg}_P(E), \\
\mu_Q(f^* E) = \text{\rm deg}(f_0) \cdot \mu_P(E).
\end{eqnarray*}
\end{lemma}

\begin{proof}
The notion of the pullback $f^*$ in~\eqref{def_pull_divisor} for a divisor is compatible with principal divisors, i.e. for any non-zero rational function $\phi \in k(X)^*$, we have $f^* \iota^* \text{\rm div}(\phi) = j^* f_0^* \text{\rm div}(\phi)$ where $\iota \colon (X,P) \longrightarrow X$ and $j \colon (Y,Q) \longrightarrow Y$ are the respective Coarse moduli morphisms, and $f_0^* \text{\rm div}(\phi)$ is the principal divisor associated to the function $\phi$, viewed as an element in the extension field $k(Y) \supset k(X)$. This induces a homomorphism
$$(f^*)^{\text{\rm ind}} \colon \text{\rm Pic}(X,P) \longrightarrow \text{\rm Pic}(Y,Q).$$
By the definition and previously discussed properties, to check that $(f^*)^{\text{\rm ind}} = f^*$, it is enough to show that
\[
f^* \left( \mathcal{O}_{(X,P)}(D) \right) \cong \mathcal{O}_{(Y,Q)}(f^*D)
\]
for every effective divisor $D$. We have the exact sequence~\eqref{ideal_sheaf}
$$0 \longrightarrow \mathcal{O}_{(X,P)}(-D) \longrightarrow \mathcal{O}_{(X,P)} \longrightarrow \mathcal{O}_D \longrightarrow 0$$
of coherent sheaves of $\mathcal{O}_{(X,P)}$-modules. Twisting by the line bundle $\mathcal{O}_{(X,P)}(D)$, we obtain the short exact sequence
$$0 \longrightarrow \mathcal{O}_{(X,P)} \longrightarrow \mathcal{O}_{(X,P)}(D) \longrightarrow \mathcal{O}_D \longrightarrow 0$$
of coherent sheaves of $\mathcal{O}_{(X,P)}$-modules; here we use that fact that the functor $- \otimes \mathcal{O}_{(X,P)}(D)$ is exact, and that $\mathcal{O}_D \otimes \mathcal{O}_{(X,P)}(D) \cong \mathcal{O}_D$. Taking a pullback under the exact functor $f^*$, we obtain a short exact sequence
$$0 \longrightarrow \mathcal{O}_{(Y,Q)} \longrightarrow f^* \left( \mathcal{O}_{(X,P)}(D) \right) \longrightarrow f^*\mathcal{O}_D \longrightarrow 0$$
of coherent sheaves of $\mathcal{O}_{(Y,Q)}$-modules. We may replace $(X,P)$ and $(Y,Q)$ by affine sub-stacks to see that $f^*\mathcal{O}_D$ is the coherent sheaf $\mathcal{O}_{f^*D}$ on $(Y,Q)$ which is the structure sheaf on the closed sub-stack of $(Y,Q)$ defined by the effective divisor $f^*D$. Thus
$$f^* \left( \mathcal{O}_{(X,P)}(D) \right) \cong \text{\rm det}\left( \mathcal{O}_{f^*D} \right) \cong \mathcal{O}_{(Y,Q)}(f^*D).$$
So $(f^*)^{\text{\rm ind}}$ and $f^*$ coincide for $\mathcal{O}_{(X,P)}(D)$, and hence for any line bundle $L \in \text{\rm Pic}(X,P)$. Using the property of $(f^*)^{\text{\rm ind}}$, we also see that for any $L \in \text{\rm Pic}(X,P)$, and $f^* L$ has $Q$-degree $\text{\rm deg}(f_0) \cdot \text{\rm deg}_P(L)$.

The statements for the bundle is immediate since $f^*$ is an exact functor which commutes with the formation of the determinant line bundle, and the property of rank with respect to pullback under a cover of curves.
\end{proof}

Any cover $f \colon (Y,Q) \longrightarrow (X,P)$ uniquely factors as a composition of the induced cover $(Y,Q) \longrightarrow (Y,f_0^*P)$ (as $Q \geq f_0^*P$) followed by a representable cover $(Y, f_0^*P) \longrightarrow (X,P)$; see Lemma~\ref{lem_rep_cover}. The pushforward functor $f_*$ behaves differently under these two covers, and needs a separate treatment.

First, consider the pushforward of a divisor under a representable cover $f \colon (Y,f_0^*P) \longrightarrow (X,P)$ of orbifold curves. Define a homomorphism
\begin{equation}\label{def_push_div}
f_* \colon \text{\rm Div}(Y,f_0^*P) \longrightarrow \text{\rm Div}(X,P),
\end{equation}
by sending a divisor $\sum n_y y \in \text{\rm Div}(Y, f_0^*P)$ to $\sum n_y \, \frac{[P(f_0(y)) \colon K_{X,f_0(y)}]}{[f_0^*P(y) \colon K_{Y,y}]} \, f(y) \in \text{\rm Div}(X,P)$.

\begin{remark}\label{rmk_def_div_pushforward}
Note that this is the usual definition if $P$ is the trivial branch data. To justify the above definition for a non-trivial branch data $P$, consider the natural atlas $V \longrightarrow (X,P)$ from Lemma~\ref{lem_atlas}. Let $y$ be a residual gerbe in $(Y,f_0^*P)$, and $x = f(y)$. We noted in Remark~\ref{rmk_natural_atlas_on_pullback} that the stabilizer group $H_y$ at $y$ is a subgroup of the stabilizer group $G_x = \text{\rm Gal}\left( P(x)/K_{X,x} \right)$ at $x$, that the projection morphism $q_1 \colon (Y,f_0^*P) \times_{(X,P)} V \longrightarrow (Y,f_0^*P)$ is an atlas, and $q_1^*y$ is a divisor of degree $|G_x|/|H_y|$ in the smooth curve $(Y,f_0^*P) \times_{(X,P)} V$. Under the cover $q_2 \colon (Y,f_0^*P) \times_{(X,P)} V \longrightarrow V$ of smooth curves, the divisor $q_1^*y$ is mapped to the divisor $q_{2,*}q_1^*y$ of the same degree; to see this, we can restrict to the connected components of the respective atlases, and after a smooth projective completion, apply \cite[Exc. IV.2.6]{Ha}. Thus the closed sub-stack $f_*y$ has support $\{x\}$, and corresponds to the equivariant divisor $q_{1,*}q_2^*y$ of degree $|G_x|/|H_y|$. Hence, $f_*y$ is the divisor $\frac{|G_x|}{|H_y|} \, x = \frac{[P(x) \colon K_{X,x}]}{[f_0^*P(y) \colon K_{Y,y}]} \, x$.
\end{remark}

As every line bundle on $(Y,f_0^*P)$ is of the form $\mathcal{O}_{(Y,f_0^*P)}(D)$ for some $D \in \text{\rm Div}(Y,f_0^*P)$, it is important to understand the relation of the coherent sheaf $f_*(\mathcal{O}_{(Y,f_0^*P)}(D))$ with $\mathcal{O}_{(X,P)}(f_*D)$. Since the functor $f_*$ is an exact functor by Lemma~\ref{lem_push_exactness_rep_cover}, we immediately have the following result.

\begin{lemma}\label{lem_fushforward_rep_cover_bundles}
Let $f_0 \colon Y \longrightarrow X$ be a cover of smooth projective connected $k$-curves. Suppose that $P$ is a branch data on $X$. Consider the induced representable cover $f \colon (Y, f_0^*P) \longrightarrow (X,P)$. For any $F \in \Vect(Y,f_0^*P)$, the pushforward coherent sheaf $f_*F$ is a bundle on $(X,P)$ of rank
\[
\text{\rm rk}(f_*F) = \text{\rm deg}(f_0) \cdot \text{\rm rk}(F).
\]
\end{lemma}

\begin{proof}
Consider any atlas $g \colon V \longrightarrow (X,P)$. As in the proof of Lemma~\ref{lem_push_exactness_rep_cover}, consider the cartesian square~\ref{eq_geom_sq} to compute the pushforward sheaf $f_*F$. As the projection map $p_2 \colon (Y,f_0^*P) \times_{(X,P)} V \longrightarrow V$ is a cover of smooth $k$-curves, $p_{2,*}$ is an exact functor, and $p_{2,*}p_1^*F$ is a bundle on $V$. As this is compatible for different choices of the atlas, we conclude that $f_*F$ is a bundle on $(X,P)$. The statement about the ranks follow since we have
\[
\text{\rm rk}(f_*F) = \text{\rm rk}(p_{2,*}p_1^*F) = \text{\rm deg}(p_2) \cdot \text{\rm rk}(p_1^*F) = \text{\rm deg}(f) \cdot \text{\rm rk}(F) = \text{\rm deg}(f_0) \cdot \text{\rm rk}(F).
\]
\end{proof}

To state our next result, we follow up on the Grothendieck duality for a cover of orbifold curves.

\begin{remark}[{Consequences of the Grothendieck Duality}]\label{rmk_Groth}
Recall from \cite[Theorem~1.16]{Nironi} that for any proper morphism $f$ of quasi compact algebraic $k$-stacks with affine diagonals, the derived functor $\mathrm{R}f_*$ of bounded below derived categories admits a right adjoint $f^!$. Let $f_0, \, P$, and $f \colon (Y,f_0^*P) \longrightarrow (X,P)$ be as in the hypothesis of Lemma~\ref{lem_fushforward_rep_cover_bundles}. Then $f_*$ is an exact functor. As $f$ is also flat, for every coherent sheaf $F$ of $\mathcal{O}_{(X,P)}$-modules, we have $f^! F \cong f^*F \otimes f^! \mathcal{O}_{(X,P)}$ (\cite[Proposition~1.20]{Nironi}). Moreover, setting $\eta_{(Y,f_0^*P)}$ and $\eta_{(X,P)}$ as the structure morphisms of $(Y,f_0^*P)$ and $(X,P)$, respectively, the respective dualizing complexes (\cite[Definition~1.22, Theorem~2.22 -- Smooth Serre Duality]{Nironi}) are given by
\begin{equation}\label{omegas}
\eta_{(Y,f_0^*P)}^! \mathcal{O}_k \cong f^! \eta_{(X,P)}^! \mathcal{O}_k \cong \omega_{(Y,f_0^*P)}[1] \, \text{ and } \, \eta_{(X,P)}^! \mathcal{O}_k \cong \omega_{(X,P)}[1]
\end{equation}
where $\omega_{(Y,f_0^*P)}$ and $\omega_{(X,P)}$ are the respective canonical line bundles (see \cite[Proposition~7.1]{Kobin} for the definition). In particular, $f^! \mathcal{O}_{(X,P)}$ is the line bundle $\mathcal{O}_{(Y,f_0^*P)}(R)$ where $R$ is the ramification divisor for the cover $f$, given by
\begin{equation}\label{ramification_divisor}
R \coloneqq \sum_y \text{\rm dr}(f_0^*P(y)/P(f_0(y))) \, y \in \text{\rm Div}(Y,f_0^*P)
\end{equation}
where $\text{\rm dr}(f_0^*P(y)/P(f_0(y)))$ is the discriminant of the finite extension $f_0^*P(y)/P(f_0(y))$; cf. \cite[see Section~2.1, cf. Definition~2.18]{KP}. Let $E \in \text{\rm Vect}_{(Y, f_0^*P)}$ and $F \in \text{\rm Vect}_{(X,P)}$. Then $f_*E$ is also a bundle by Lemma~\ref{lem_fushforward_rep_cover_bundles}. So the functors $\sheafhom_{(Y,f_0^*P)}(E, -) \cong E^\vee \otimes_{(\mathcal{O}_{(Y,f_0^*P)})} -$ and $\sheafhom_{(X,P)}(f_*E, -) \cong (f_*E)^\vee \otimes_{\mathcal{O}_{(X,P)}} -$ are exact functors. From the sheaf version of the Grothendieck duality in \cite[Corollary~2.10]{Nironi}, we deduce that the natural map
\begin{equation}\label{eq_Grothendieck_isom}
f_* \, \sheafhom_{(Y, f_0^*P)}(E , f^! F) \longrightarrow \sheafhom_{(Y, f_0^*P)}(f_* E , f_* f^! F) \overset{\text{\rm tr}_f}\longrightarrow \sheafhom_{(X , P)}(f_* E , F)
\end{equation}
is an isomorphism; here $\text{\rm tr}_f$ is the trace morphism induced by the natural transformation $f_*f^! \Rightarrow \text{\rm id}$.
\end{remark}

\begin{lemma}\label{lem_push_div_line_rep}
Under the hypothesis of Lemma~\ref{lem_fushforward_rep_cover_bundles}, the following hold.
\begin{enumerate}[leftmargin=*]
\item For any $L \in \text{\rm Pic}(Y, f_0^*P)$, the pushforward coherent sheaf $f_*L$ is a bundle on $(X,P)$ of rank $n$. For any $D \in \text{\rm Div}(Y, f_0^*P)$, we have
$$\text{\rm det}(f_* \mathcal{O}_{(Y,f_0^*P)}(D)) \cong \text{\rm det}(f_* \mathcal{O}_{(Y,f_0^*P)}) \otimes \mathcal{O}_{(X,P)}(f_*D).$$\label{fin:1}
\item The pushforward homomorphism in~\eqref{def_push_div} preserves linear equivalence. The composite homomorphism $f_* \circ f^* \colon \text{\rm Div}(X,P) \longrightarrow \text{\rm Div}(X,P)$ is the multiplication by $n$ map.\label{fin:2}
\item $\text{\rm det}(f_* \omega_{(Y, f_0^*P)}) \cong \left(\text{\rm det}(f_* \mathcal{O}_{(Y, f_0^*P)})\right)^{-1} \otimes \omega_{(X,P)}^{\otimes n}.$\label{fin:3}
\item Let $R$ be the ramification divisor on $(Y, f_0^*P)$ for the cover $f$ given by~\eqref{ramification_divisor}, and set $B \coloneqq f_*R \in \text{\rm Div}(X,P)$. Then
$$\left(\text{\rm det}(f_* \mathcal{O}_{(Y, f_0^*P)})\right)^{\otimes 2} \cong \mathcal{O}_{(X,P)}(-B).$$\label{fin:4}
\end{enumerate}
\end{lemma}

\begin{proof}
For any line bundle $L$ on $(Y, f_0^*P)$, the pushforward coherent sheaf $f_*L$ is a bundle on $(X,P)$ of rank $n$ by Lemma~\ref{lem_fushforward_rep_cover_bundles}. For the rest of \eqref{fin:1}, we proceed as in the case of curves (see \cite[Exc. IV.2.6, page 306]{Ha}). As usual, it is enough to prove the statement only for an effective divisor $D$ on $(Y,f_0^*P)$. We also denote the corresponding $0$-dimensional closed sub-stack of $(Y,f_0^*P)$ by $D$. We have the short exact sequence~\eqref{ideal_sheaf}
$$0 \longrightarrow \mathcal{O}_{(Y,f_0^*P)}(-D) \longrightarrow \mathcal{O}_{(Y,f_0^*P)} \longrightarrow \mathcal{O}_D \longrightarrow 0$$
of coherent sheaves of $(Y,f_0^*P)$-modules. First twisting by the line bundle $\mathcal{O}_{(Y,f_0^*P)}(D)$, and then taking the pushforward under the exact functor $f_*$, we obtain the short exact sequence
$$0 \longrightarrow f_*\mathcal{O}_{(Y,f_0^*P)} \longrightarrow f_* \left( \mathcal{O}_{(Y,f_0^*P)}(D) \right) \longrightarrow f_*\mathcal{O}_D \longrightarrow 0$$
of coherent sheaves of $\mathcal{O}_{(Y,f_0^*P)}$-modules. Considering the respective determinants, we obtain
$$\text{\rm det}\left( f_*\left(\mathcal{O}_{(Y,f_0^*P)}(D)\right) \right) \cong \text{\rm det}\left( f_*\mathcal{O}_{(Y,f_0^*P)} \right) \otimes \text{\rm det}\left( f_*\mathcal{O}_D \right).$$
By Remark~\ref{rmk_def_div_pushforward}, if $D = \sum n_i y_i$, the closed sub-stack $D$ is mapped to the closed sub-stack defined by $\sum n_i f_*(y_i) = \sum n_i \, \frac{[P(f_0(y_i)) \colon K_{X,f_0(y_i)}]}{[f_0^*P(y_i) \colon K_{Y,y_i}]} \, f(y_i)$ of $(X,P)$. Thus $f_* \mathcal{O}_D \cong \oplus_{i} \mathcal{O}_{n_i f_*(y_i)}$, and hence $\text{\rm det}\left( f_* \mathcal{O}_D \right) \cong \mathcal{O}_{(X,P)}(f_*D)$. We obtain the reuired statement:
$$\text{\rm det}\left( f_*\mathcal{O}_{(Y,f_0^*P)}(D) \right) \cong \text{\rm det}\left( f_*\mathcal{O}_{(Y,f_0^*P)} \right) \otimes \mathcal{O}_{(X,P)}(f_*D).$$

Now we prove \eqref{fin:2}. Let $\iota$ and $j$ be the Coarse moduli morphism for $(X,P)$ and $(Y,f_0^*P)$, respectively. We note that the two homomorphisms
$$f_* \circ j^*, \, \iota^* \circ f_{0,*} \colon \text{\rm Div}(Y) \longrightarrow \text{\rm Div}(X,P)$$
coincide: for a divisor $\sum_y n_y \, y$,
\begin{eqnarray*}
f_* \circ j^* (\sum_y n_y \, y) = \sum_y n_y [f_0^*P(y) \colon K_{Y,y}] \frac{[P(f_0(y)) \colon K_{X,f_0(y)}]}{[f_0^*P(y) \colon K_{Y,y}]} \, f(y) \\
= \sum_y n_y \, [P(f_0(y)) \colon K_{X,f_0(y)}] f(y) = \sum_y n_y \, \iota^* f_0(y).
\end{eqnarray*}
Since a principal divisor on $(Y,f_0^*P)$ is of the form $j^* \text{\rm div}(\phi)$ for some $\phi \in k(Y)^*$, we have
$$f_* j^* \text{\rm div}(\phi) = \iota^* \text{\rm div}(N(\phi))$$ where $N(\phi) \in k(X)^*$ is the norm of $\phi$ under the cover $f_0$. Thus $f_*$ preserves linear equivalence. Further, for a residual gerbe $x$ on $(X,P)$,
$$f_*f^*x = f_* \left( \sum_{f(y)=x} [f_0^*P(y) \colon P(x)]  \, y \right) = \sum_{f(y)=x} [f_0^*P(y) \colon P(x)] \, \frac{[P(x) \colon K_{X,x}]}{[f_0^*P(y) \colon K_{Y,y}]} \, x = n \, x.$$
Since $f^*$ and $f_*$ are defined linearly, $f_*f^*$ is the multiplication by $n$ map.

As noted in Remark~\ref{rmk_Groth}, we have $\omega_{(Y,f_0^*P)} \cong f^! \omega_{(X,P)}$. By the Grothendieck duality~\eqref{eq_Grothendieck_isom}, we have an isomorphism
$$f_* \omega_{(Y,f_0^*P)} \cong f_* \sheafhom_{(Y,f_0^*P)}(\mathcal{O}_{(Y,f_0^*P)}, f^!\omega_{(X,P)}) \cong \sheafhom_{(X,P)}(f_*\mathcal{O}_{(Y,f_0^*P)}, \omega_{(X,P)}).$$
Since $f_*\mathcal{O}_{(Y,f_0^*P)}$ is a vector bundle of rank $n$, we have
$$\sheafhom_{(X,P)}(f_*\mathcal{O}_{(Y,f_0^*P)}, \omega_{(X,P)}) \cong \left( f_*\mathcal{O}_{(Y,f_0^*P)} \right)^\vee \otimes \omega_{(X,P)}.$$
Considering the determinants, we obtain \eqref{fin:3}.

We have $\omega_{(Y,f_0^*P)} \cong \mathcal{O}_{(Y,f_0^*P)}(K_{(Y,f_0^*P)})$, and $\omega_{(X,P)} \cong \mathcal{O}_{(X,P)}(K_{(X,P)})$ where $K_{(Y,f_0^*P)}$ and $K_{(X,P)}$ are the canonical divisors in \cite[Proposition~7.1]{Kobin}. By Remark~\ref{rmk_Groth}\eqref{omegas}, the divisors $K_{(Y,f_0^*P)}$ and $f^*K_{(X,P)} + R$ are linearly equivalent. By \eqref{fin:2}, we see that $f_*K_{(Y,f_0^*P)}$ and $f_*(f^*K_{(X,P)}) + f_*R$ are linearly equivalent, and $f_*f^*$ is the multiplication by $n$ map. So we have
\begin{equation}\label{l:1}
\mathcal{O}_{(X,P)}(-B) \cong \omega_{(X,P)}^{\otimes n} \otimes \mathcal{O}_{(X,P)}(f_* K_{(Y,f_0^*P)})^{-1}.
\end{equation}
We have
$$\mathcal{O}_{(X,P)}(f_* K_{(Y,f_0^*P)})^{-1} \cong \text{\rm det}(f_* \mathcal{O}_{(Y,f_0^*P)}) \otimes \left( \text{\rm det}(f_* \omega_{(Y,f_0^*P)}) \right)^{-1} \cong \left( \text{\rm det}(f_* \mathcal{O}_{(Y,f_0^*P)}) \right)^{\otimes 2} \otimes \omega_{(X,P)}^{\otimes -n}$$
where the first isomorphism is by \eqref{fin:1}, and the second isomorphism is by \eqref{fin:3}. Setting this in Equation~\ref{l:1}, we conclude~\eqref{fin:4}.
\end{proof}

Finally, we consider the pushforward of a bundle under a non-representable cover. We have the following result.

\begin{lemma}\label{lem_push_div_line_nonrep}
Let $X$ be a smooth projective connected $k$-curves, $Q \geq P$ be two branch data on $X$. Let $j \colon (X,Q) \longrightarrow (X,P)$ be the cover induced by $\text{\rm id}_X$. Let $D = \sum n_x x \in \text{\rm Div}(X,Q)$. Consider the divisor
$$\floor{D} \coloneqq \sum \floor{\frac{n_x}{[Q(x)\colon P(x)]}} x \in \text{\rm Div}(X,P)$$
where $\floor{\frac{n_x}{[Q(x)\colon P(x)]}}$ denote the integral part of $\frac{n_x}{[Q(x)\colon P(x)]}$. Then $j_* \mathcal{O}_{(X,Q)}(D)$ is the line bundle given by $\mathcal{O}_{(X,P)}(\floor{D})$. In particular, $j_* \mathcal{O}_{(X,Q)} \cong \mathcal{O}_{(X,P)}$. Further, for any bundle $E$ on $(X,Q)$ of rank $n$, the pushforward coherent sheaf $j_*E$ is a bundle on $(X,P)$ of rank $n$.
\end{lemma}

\begin{proof}
To prove the first two statements, considering an atlas of $(X,P)$, and since the divisors have finite support contained in an affine sub-stack, we are reduced to the case: $P$ is the trivial branch data, and $j \colon (X,Q) \longrightarrow X$ is the Coarse moduli morphism. Then the result is \cite[Lemma~4.10]{Kobin}.

For the last statement, again considering atlases, we see that the pushforward of a bundle is again a bundle. The rank remains the same since it is defined generically.
\end{proof}

\section{Slope Stability}\label{sec_slope_stability}
\subsection{Slope Stability for Orbifold Curves}\label{sec_stability_f_o_c}
The purpose of this section is to define and study the slope stability conditions. Unless otherwise specified, we work with the following notation.

\begin{notation}\label{not_geometric_notation}
Let $(X,P)$ be a connected proper orbifold curve, and $\iota \colon (X,P) \longrightarrow X$ be the Coarse moduli morphism. Let $Q \geq P$ be a geometric branch data on $X$ such that $(X,Q) = [Z/G]$ for a $G$-Galois cover $g_0 \colon Z \longrightarrow X$ of smooth projective connected $k$-curves (cf. Remark~\ref{rmk_geometric_results}). The cover $g_0$ factors as the composition of the covers
\begin{equation}\label{eq_g0_factors}
g_0 \colon Z \overset{u}\longrightarrow [Z/G] = (X,Q) \overset{j}\longrightarrow (X,P) \overset{\iota}\longrightarrow X
\end{equation}
where the canonical morphism $u$ is an atlas, $\iota \circ j \colon (X,Q) \longrightarrow X$ is the Coarse moduli morphism. Set $g \coloneqq j \circ u \colon  Z \longrightarrow (X,P)$.
\end{notation}

Let $E \in \text{\rm Vect}(X,P)$. The rank $\text{\rm rk}(E) = \text{\rm rk}(\iota_*E)$ and the $P$-degree $\text{\rm deg}_P(E)$ of $E$ are defined in Section~\ref{sec_degree}. The $P$-\textit{slope} $\mu_P(E)$ of $E$ is defined as
\begin{equation}\label{def_P-slope}
\mu_P(E) \coloneqq \frac{\text{deg}_P(E)}{\text{\rm rk} (E)}.
\end{equation}

We recall that there is also the notion of an equivariant slope (that we define in Definition~\ref{def_equivariant_slope}) for $E$ as follows. The functor $u^*$ defines an equivalence of categories (\cite[Definition~7.18]{V})
\begin{equation}\label{eq_equivalence_orbifold_equivariant_bundle}
u^* \colon \text{\rm Vect}(X,Q) \overset{\sim}\longrightarrow \text{\rm Vect}^G(Z)
\end{equation}
between the bundles on $(X,Q)$ and the $G$-equivariant bundles on $Z$, with a quasi-inverse defined by the equivariant pushforward $u^G_*$ (to see that this defines a quasi-inverse, one can work over charts and use the Galois \'{e}tale descent for schemes). We also have an embedding of categories
\begin{equation}\label{eq_inclusion_for_vect_on_stacky_curves}
j^* \colon \text{\rm Vect}(X,P) \hookrightarrow \text{\rm Vect}(X,Q).
\end{equation}
To see this, note that the functor $j^*$ is injective on the objects, and for any two bundles $E, \, F \in \text{\rm Vect}(X,P)$, by \cite[Proposition~9.3.6, pg. 205]{Olsson} and the projection formula~Proposition~\ref{prop_proj_bc}, we have the following.
\begin{equation*}
\begin{array}{rcl}
\Hom_{\text{\rm Vect}(X,Q)}\left( j^* E, j^* F \right) & = & \Hom_{\text{\rm Vect}(X,P)}\left( E, j_* j^* F \right) \\
 & = & \Hom_{\text{\rm Vect}(X,P)}\left( E, F \otimes_{\mathcal{O}_{(X,P)}} j_*\mathcal{O}_{(X,Q)} \right) \\
 & = & \Hom_{\text{\rm Vect}(X,P)}\left( E, F \right)
\end{array}
\end{equation*}
The last equality follows since $j_* \mathcal{O}_{(X,Q)} \cong \mathcal{O}_{(X,P)}$ by Lemma~\ref{lem_push_div_line_nonrep}.

\begin{definition}\label{def_equivariant_slope}
Under the above notation, we have the embedding (as the composition of the functors~\eqref{eq_equivalence_orbifold_equivariant_bundle} and \eqref{eq_inclusion_for_vect_on_stacky_curves})
\[
g^* = (j \circ u)^* \colon \text{\rm Vect}(X,P) \hookrightarrow \text{\rm Vect}^G(Z).
\]
For any $E \in \Vect(X,P)$, define the \textit{equivariant degree} and the \textit{equivariant slope} of $E$ as follows.
$$\text{\rm deg}^{\text{\rm eq}}_{(X,P)}(E) \coloneqq \frac{1}{|G|} \text{\rm deg}( g^* E), \, \hspace{.5em} \text{and } \, \hspace{.5em} \mu^{\text{\rm eq}}_{(X,P)}(E) \coloneqq \frac{1}{|G|} \mu( g^* E).$$ 
\end{definition}

\begin{remark}
To see that the rational numbers $\text{\rm deg}^{\text{\rm eq}}_{(X,P)}$ and $\mu^{\text{\rm eq}}_{(X,P)}$ do not depend on the choice of the cover $g_0 \colon Z \longrightarrow X$, it is enough to consider the case $(X,P) = (X,Q) = [Z/G]$. Suppose that $[Z/G] = [Z'/G']$, and let $\mathcal{E}'$ be the $G'$-equivariant bundle on $Z'$ corresponding to $E$. Then both $g^* E$ and $\mathcal{E}'$ pullback to the same equivariant bundle on $Z \times_{(X,P)} Z'$ of degree $|G'| \text{\rm deg}(g^* E) = |G| \text{\rm deg}(\mathcal{E}')$.
\end{remark}

\begin{proposition}\label{prop_slopes_coincide}
For any $E \in \Vect(X,P)$, we have
\[
\text{\rm deg}^{\text{\rm eq}}_{(X,P)}(E) = \text{\rm deg}_P(E) \hspace{.5em} \, \text{\rm and} \, \hspace{.5em} \mu^{\text{\rm eq}}_{(X,P)}(E) = \mu_P(E).
\]
\end{proposition}

\begin{proof}
This is immediate by applying Lemma~\ref{lem_pullback_div_line_bundles} to the representable cover $g \colon Z \longrightarrow (X,P)$.
\end{proof}

Using the above notion of slope, we can define the notion of $P$-(semi/poly)stability.

\begin{definition}\label{def_stability_condions_stacky_curve}
A bundle $E$ on $(X,P)$ is called $P$-(semi)stable if for any sub-bundle $0 \neq F \subset E$ in $\Vect(X,P)$, we have
$$\mu_{P}(F) \hspace{.2cm} ( \leq ) \hspace{.2cm} \mu_{P}(E).$$
As in the case of curves, the notation $(\leq)$ means that $E$ is $P$-semistable if we have $\leq$, and it is $P$-stable if we have the strict inequality $<$.

A $P$-semistable bundle $E$ is called $P$-\textit{polystable} if $E = \oplus E_i$, a finite sum, where for each $i$, the bundle $E_i$ is $P$-stable satisfying $\mu_P(E_i) = \mu_P(E)$.
\end{definition}

We keep the following notion and important properties of equivariant slope stability conditions as a separate remark for our later use.

\begin{remark}\label{rmk_G-stability_vs_stability}
Let $\mathcal{E}$ be a $G$-equivariant bundle on $Z$. Recall that $\mathcal{E}$ on $Z$ is $G$-(semi)stable if for any $G$-equivariant sub-bundle $\mathcal{F} \subset \mathcal{E}$, we have $\mu(\mathcal{F}) \, ( \leq ) \, \mu(\mathcal{E})$; a $G$-polystable bundle on $Z$ is a $G$-semistable bundle that is a finite sum of $G$-stable bundles having the same slope.

The bundle $\mathcal{E}$ is $G$-semistable if and only if $\mathcal{E}$ is a semistable in the usual sense (for example, see \cite[Lemma~2.7]{B}; this follow from the uniqueness of the Harder-Narasimhan filtration). From the definition, it is clear that if $\mathcal{E}$ is stable in the usual sense, it is also $G$-stable. Whereas, a $G$-stable bundle need not be stable in the usual sense -- suppose that there is an irreducible $k[G]$-module $V$ of dimension $\geq 2$, and consider the $G$-equivariant trivial bundle $\mathcal{O}_Z \otimes_k V$ equipped with the diagonal $G$-action. This is a $G$-stable bundle that is not stable. When $G$ is a non-abelian finite group, such a $V$ always exists.

We further note that $\mathcal{E}$ is $G$-polystable if and only if it is polystable in the usual sense. By the uniqueness of the the socle of a semistable bundle, it follows that a $G$-polystable bundle $\mathcal{E}$ on $Z$ is polystable in the usual sense. Conversely, if $\mathcal{E}$ is polystable, it can be written uniquely as
\[
\mathcal{E} = \oplus_{1 \leq i \leq l} \, \mathcal{E}_i \otimes_k \mathrm{Hom}(\mathcal{E}_i, \mathcal{E}),
\]
where $\mathcal{E}_i$ are mutually non-isomorphic stable bundles on $Z$, each of slope $\mu(\mathcal{E})$. In particular, each $k$-vector space $\mathrm{Hom}(\mathcal{E}_i, \mathcal{E})$ is a $k[G]$-module. Then it follows that any $\Gamma$-stable sub-bundle $\mathcal{F}$ of $\mathcal{E}$ with $\mu(\mathcal{F}) = \mu(\mathcal{E})$ is of the form
\[
\mathcal{F} = \oplus_{i \in I} \, \mathcal{E}_i \otimes_k \mathrm{Hom}(\mathcal{E}_i, \mathcal{F})
\]
for some subset $I \subset \{1, \ldots, l\}$, and where $\mathrm{Hom}(\mathcal{E}_i, \mathcal{F})$ is a $G$-invariant sub-module of $\mathrm{Hom}(\mathcal{E}_i, \mathcal{E})$. Thus, every $G$-stable sub-bundle of $\mathcal{E}$ having the same slope as $\mathcal{E}$ is a $G$-invariant direct summand. This shows that every polystable $G$-equivariant bundle on $Z$ is $G$-polystable. \qed
\end{remark}

We have the following observation on the properties of the stability conditions under a cover $(X,Q) \xlongrightarrow{j} (X,P)$ induced by two branch data $P$ and $Q$ on $X$, satisfying $Q \geq P$.

\begin{remark}\label{rmk_iota_preserve}
Let $E \in \Vect(X,P)$. Using Proposition~\ref{prop_slopes_coincide}, we conclude that
\begin{itemize}[leftmargin=*]
    \item $E$ is $P$-(semi)stable if and only if the $G$-equivariant bundle $g^* E$ on $Z$ is $G$-(semi)stable;
    \item $E$ is $P$-polystable if and only if $g^* E$ is $G$-polystable.
\end{itemize}
In particular, this implies that the embedding $j^*$ in~\eqref{eq_inclusion_for_vect_on_stacky_curves} preserves slope stability conditions; namely, for any $E \in \Vect(X,P)$, the pullback bundle $j^* E$ is $Q$-(semi)stable (respectively, $Q$-polystable) if and only if $E$ is $P$-(semi)stable (respectively, $P$-polystable). In fact, this holds for any two branch data $Q \geq P$ (where $Q$ is not necessarily geometric) as we can further choose a geometric branch data $Q' \geq Q$ by Remark~\ref{rmk_geometric_results}\eqref{geo:1}. \qed
\end{remark}

Let $E \in \text{\rm Vect}(X,P)$. The $G$-equivariant bundle $g^*E$ (where $g = j \circ u$ as in Notation~\ref{not_geometric_notation}) has rank $\text{\rm rk}(E)$ and degree $|G| \text{\rm deg}_P(E)$ by Lemma~\ref{lem_pullback_div_line_bundles}. By the Riemann-Roch Theorem over $Z$, we have
$$\mu(g^*E) - (g_Z -1) = \frac{\chi(g^*E)}{\text{\rm rk}(E)} \leq \frac{h^0(Z, g^*E)}{\text{\rm rk}(E)} \leq h^0(Z, g^*E)$$
where $g_Z$ is the genus of $Z$, and $\chi$ denote the Euler characteristic. Then
$$\mu_P(E) = \frac{\mu(g^*E)}{|G|} \leq \frac{h^0(Z, g^*E) + (g_Z -1)}{|G|}.$$
In particular, the $P$-slope $\mu_P(F)$ for any sub-bundle $F$ of $E$ is bounded above, and there exists a maximal possible slope. A sub-bundle $E_1$ of $E$ of the maximal $P$-slope with the maximal possible rank satisfies the following property: for any sub-bundle $F \subseteq E$, we have $\mu_P(E_1) \geq \mu_P(F)$; when $\mu_P(E_1) = \mu_P(F)$, we have $F \subset E_1$ (the argument is the same as the proof of \cite[Lemma~1.3.5, Theorem~1.6.7]{HL}). Such a sub-bundle $E_1$ of $E$ is $P$-semistable and unique, called the \textit{maximal destabilizing sub-bundle} of $E$. In view of our previous discussion, Remark~\ref{rmk_iota_preserve}, and the analogous arguments as in the classical case of curves (cf. \cite[Theorem~1.6.7]{HL}) produces the following useful properties.

\begin{proposition}\label{prop_properties_stacky}
Suppose that Notation~\ref{not_geometric_notation} hold. For any $E \in \text{\rm Vect}(X,P)$ and $E' \in \text{\rm Vect}(X,Q)$, we have the following properties.
\begin{enumerate}[leftmargin=*]
\item The maximal destabilizing sub-bundle $\text{\rm HN}(E)_1$ exists, and it has the following property: for any sub-bundle $F \subseteq E$, we have $\mu_P(\text{\rm HN}(E)_1) \geq \mu_P(F)$; when $\mu_P(\text{\rm HN}(E)_1) = \mu_P(F)$, we have $F \subset \text{\rm HN}(E)_1$. The sub-bundle $\text{\rm HN}(E)_1$ of $E$ is $P$-semistable and unique. Define $\mu_{P, \text{\rm max}}(E) \coloneqq \mu_P(\text{\rm HN}(E)_1)$.\label{H:1}
\item (Harder-Narasimhan filtration) There is a unique filtration
$$0 = \text{\rm HN}(E)_0 \subset \text{\rm HN}(E)_1 \subset \cdots \subset \text{\rm HN}(E)_l = E$$
such that $\text{\rm HN}(E)_i/\text{\rm HN}(E)_{i-1}$ are $P$-semistable and their $P$-slopes satisfy
$$\mu_{P,\text{\rm max}}(E) \coloneqq \mu_P(\text{\rm HN}(E)_1) > \cdots > \mu_P(E/\text{\rm HN}(E)_{l-1}).$$\label{H:2}
\item The rational number $\mu_{P, \text{\rm max}}^{\rm eq}(E) \coloneqq \frac{1}{|G|} \mu_{\text{\rm max}}(g^*E)$ is independent of the choice of the cover $g_0 \colon Z \longrightarrow X$.\label{H:3}
\item Under the equivalence $u^* \colon \text{\rm Vect}(X,Q) \overset{\sim}\longrightarrow \text{\rm Vect}^G(Z)$, the Harder-Narasimhan filtration for $E'$ uniquely corresponds to the Harder-Narasimhan filtration for $u^*E'$. In particular, $u^* \text{\rm HN}(E')_1 \cong \text{\rm HN}(u^*E')_1$, and $\mu_{Q, \text{\rm max}}^{\rm eq}(E') = \mu_{Q, \text{\rm max}}(E')$.\label{H:4}
\item If $\text{\rm Hom}_{\Vect(X,P)}(E , F)$ is non-trivial for some $F \in \Vect(X,P)$, we have $\mu_{P,\text{\rm max}}(E) \leq \mu_{P, \text{\rm max}}(F)$.\label{H:5}
\item If $L$ is a line bundle on $(X,P)$, the tensor product $E \otimes L$ is $P$-(semi)stable if and only if $E$ is $P$-(semi)stable.\label{H:6}
\item (Jordan-H\"{o}lder filtration) If $E$ is $P$-semistable, there exists a filtration
$$0 = E^{(0)} \subset E^{(1)} \subset \cdots \subset E^{(l-1)} \subset E^{(l)}$$
such that each $E^{(i)}/E^{(i-1)}$ is $P$-stable, having the same $P$-slope $\mu_P(E)$.\label{H:7}
\item (Socle) If $E$ is $P$-semistable, there is a maximal unique polystable sub-bundle $\mathcal{S}(E)$ of $E$ with $\mu_P(\mathcal{S}(E)) = \mu_P(E)$. We call $\mathcal{S}(E)$ the socle of $E$.\label{H:S}
\end{enumerate}
\end{proposition}

In Remark~\ref{rmk_iota_preserve}, we considered the nature of the slope stability properties under a non-representable cover. In the following, we note the properties of the slope stability under representable covers --- for this, we no longer assume Notation~\ref{not_geometric_notation}.

\begin{proposition}\label{prop_properties_orbi_cover}
Let $f \colon (Y,Q) \longrightarrow (X,P)$ be a cover of connected orbifold curves. This induces a cover $f_0 \colon Y \longrightarrow X$. Let $E \in \text{\rm Vect}(X,P)$. Then the following hold.\label{s:b}
\begin{enumerate}
\item If $f_0$ is a Galois cover, and $Q = f_0^*P$, then the Harder-Narasimhan filtration of $f^*E$ is the one obtained by applying $f^*$ to that of the bundle $E$. In particular,
\[
\mu_{Q, \text{\rm max}}( f^*E) = \text{\rm deg}(f_0) \cdot \mu_{P, \text{\rm max}}(E)\label{s:b1}
\]
\item $E$ is $P$-semistable if and only if $f^*E \in \Vect(Y,Q)$ is $Q$-semistable.\label{s:b2}
\item If $f^*E \in \Vect(Y,Q)$ is $Q$-stable, then $E$ is $P$-stable.\label{s:b3}
\item If $f_0$ is Galois, and $E$ is either $P$-stable or $P$-polystable, then $f^*E \in \Vect(Y,Q)$ is $Q$-polystable.\label{s:b4}
\item If $f_0$ is Galois, $f$ is an \'{e}tale cover, and $f^*E \in \Vect(Y,f_0^*P)$ is $f_0^*P$-polystable, then $E$ is $P$-polystable.\label{s:b5}
\item If $f_0$ is Galois, and $f$ is an \'{e}tale cover, then for any $f_0^*P$-polystable bundle $F \in \Vect(Y,f_0^*P)$, the pushforward bundle $f_*F$ is $P$-polystable.\label{s:b6}
\end{enumerate}
\end{proposition}

\begin{proof}
\eqref{s:b1} is a consequence of Lemma~\ref{lem_push_exactness_rep_cover}. For the rest of the statements, without loss of generality (using Remark~\ref{rmk_iota_preserve}), we assume that $Q = f_0^*P$.

Statement~\eqref{s:b3} and the reverse direction of~\eqref{s:b2} follow directly from the definition of slope stability. Now we prove the forward direction of~\eqref{s:b2}. First, consider the Galois closure
\[
\bar{f}_0 \colon \bar{Y} \xlongrightarrow{\hat{f}_0} Y \xlongrightarrow{f} X
\]
of the cover $f$. This produces the induced representable cover
\[
\bar{f} \colon (\bar{Y}, \bar{f}_0^*P) = (\bar{Y}, \hat{f}_0^*f_0^*P) \longrightarrow (X,P).
\]
By Lemma~\ref{lem_push_exactness_rep_cover}, the maximal destabilizing sub-bundle $\text{\rm HN}(\bar{f}^*E)_1$ of $\bar{f}^*E$ is the pullback of a $P$-semistable sub-bundle $E_1$ of $E$ with $\mu_P(E_1) \geq \mu_P(E)$. Since $E$ is assumed to be $P$-semistable, we have $E_1 = E$, implying that $\bar{f}^*E$ is $\hat{f}_0^*f_0^*P$-semistable. Now by the reverse direction of~\eqref{s:b2} applied to the cover $\bar{f}$, the required statement for the cover $f$ follows.

Now suppose that $f_0$ is $\Gamma$-Galois, and that $E$ is $P$-stable. Then the $\Gamma$-invariant bundle $f^*E$ is $f_0^*P$-semistable by~\eqref{s:b2}. Using Proposition~\ref{prop_properties_stacky}~\eqref{H:5}, we conclude that the socle $\mathcal{S}(f^*E)$ is $\Gamma$-invariant. By Lemma~\ref{lem_push_exactness_rep_cover}~\eqref{ds}, we obtain a $P$-semistable sub-bundle $F \subseteq E$ such that $f^*F \cong \mathcal{S}(f^*E)$. Consequently, $\mu_P(F) = \mu_P(E)$. As $E$ was $P$-stable, we have $F = E$. So $f^*E \cong \mathcal{S}(F^*E) = f^*E$, showing that $f^*E$ is $f_0^*P$-polystable. Furthermore, when $E$ is $P$-polystable, $E$ is direct sum of $P$-stable sub-bundles of the same $P$-slope. As $f^*$ preserves a direct sum, we see that $f^*E$ is a direct sum of $f_0^*P$-polystable sub-bundles of the same $f_0^*P$-slope. Thus, $f^*E$ is also $f_0^*P$-polystable in this case. This proves~\eqref{s:b4}.

By Remark~\ref{rmk_iota_preserve} and Remark~\ref{rmk_natural_atlas_on_pullback}~\eqref{rmk:2}, it is enough to prove the statement~\ref{s:b5} when $(X,P) = [Z/G]$ is a geometric orbifold curve. Let $u \colon Z \longrightarrow (X,P)$ be the natural atlas. Taking a dominant connected component $W$ in the normalization of $Y \times_X Z$, we obtain the following commutative diagram:
\begin{equation*}
    \begin{tikzcd}
        W \arrow[r, "g"] \arrow[d, "v"] \arrow[dr, phantom, "\#"] & Z \arrow[d, "u"] \\
        (Y,f_0^*P) \arrow[r, "f"] & (X,P).
    \end{tikzcd}
\end{equation*}
Here $g$ is an \'{e}tale Galois cover of smooth projective connected $k$-curves with Galois group $N$, say, and $v$ is an atlas. The induced map $W \longrightarrow X$ is Galois, with group $\Gamma$, say. In particular, $N$ is a normal sub-group of $\Gamma$ with quotient $G$. Since $f^*E$ is $f_0^*P$-polystable, the $\Gamma$-equivariant bundle $v^*f^*E \cong g^*u^*E$ is polystable by Remark~\ref{rmk_G-stability_vs_stability}. Again by the same result, $g^*u^*E$ is $N$-polystable. As $g$ is an $N$-Galois \'{e}tale cover, the $G$-equivariant bundle $u^*E$ is $G$-polystable, or equivalently, $E$ is $P$-polystable. This completes the proof of~\eqref{s:b5}.

Now we prove~\ref{s:b6}. Since $f_*$ is left exact, it is an additive functor; hence, it is enough to prove the statement when $F$ is a $f_0^*P$-stable bundle. We have the following Cartesian square:
\begin{equation*}
\begin{tikzcd}
(Y,f_0^*P) \times_{(X,P)} (Y,f_0^*P) \arrow[r, swap, "\text{\rm pr}_2"] \arrow[dr, phantom, "\square"] \arrow[d, "\text{\rm pr}_1"] & (Y,f_0^*P) \arrow[d, "f"] \\
(Y,f_0^*P) \arrow[r, swap, "f"] & (X,P)
\end{tikzcd}    
\end{equation*}
Since $f$ is Galois \'{e}tale with Galois group $\Gamma$, say, the fiber product stack $(Y,f_0^*P) \times_{(X,P)} (Y,f_0^*P)$ is a disjoint union of the orbifold curve $(Y, f_0^*P)$, parameterized by $\gamma \in \Gamma$. Using the base change theorem--Proposition~\ref{prop_proj_bc}, we have
\[
f^*f_* F \cong \text{\rm pr}_1^* \text{\rm pr}_{2,*} F \cong \oplus_{\substack{\gamma \in \Gamma}} \, \gamma^*F.
\]
Since $F$ is $f_0^*P$-stable, the bundle $f^*f_* F$ is $f_0^*P$-polystable. From~\eqref{s:b5} we conclude that $f_*F$ is $P$-polystable.
\end{proof}

We end this section with the following observation that over $k = \mathbb{C}$, our notion of slope stability coincides with the parabolic slope stability.

\begin{remark}\label{rmk_parabolic_slope_same_as_P_slope}
When $k = \mathbb{C}$, an orbifold curve is determined by a finite set $B \subset X$ of closed points and a positive integers $n_x$ for each $x \in B$. Let $D = \sum_{\substack{x \in B}} x \in \Div(X)$. By \cite[Proposition~5.15]{KM}, there is an equivalence of categories
\begin{equation}\label{eq_13}
\Vect(X,P) \overset{\sim} \longrightarrow \Vect_{\text{\rm par, rat}}(X,D)  
\end{equation}
where $\Vect_{\text{\rm par, rat}}(X,D)$ is the category of parabolic vector bundles on $X$ with respect to the divisor $D$, and over each $x \in B$, the weights are of the form $a/n_x$, \, $0 \leq a < n_x$.

There exists a connected $G$-Galois cover $g_0 \colon Z \longrightarrow X$ of smooth projective connected $k$-curves that is branched over the set $B$, and for each point $x \in B$, the integer $n_x$ divides the ramification index at any point $z \in g_0^{-1}(x)$. By \cite[Section~3, Equation~3.12]{B}, \cite[Lemma~4.4]{P}, for each parabolic vector bundle $V_* \in \Vect_{\text{\rm par, rat}}(X,D)$, there is a unique $G$-equivariant bundle $\hat{V} \in \Vect^G(Z)$, and
$$\mu(\hat{V}) = |G|\mu_{\text{\rm para}}(V_*)$$
where $\mu$ is the usual slope for vector bundles on $Z$ and $\mu_{\text{\rm para}}$ is the parabolic slope. Moreover, the association $V_* \mapsto \hat{V}$ preserves the respective slope stability.

Using Remark~\ref{rmk_iota_preserve} and Proposition~\ref{prop_slopes_coincide}, we see that under the equivalence~\eqref{eq_13}, the parabolic slope is the same as $P$-slope, and parabolic slope stability conditions are the same as $P$-stability conditions.
\end{remark}

\subsection{Pushforward of the Structure Sheaf}\label{sec_pushforward}
The purpose of this section is to relate the maximal destabilizing sub-bundle of the pushforward of the structure sheaf with the maximal \'{e}tale sub-cover of a cover of orbifold curves. More precisely, let $f \colon (Y,Q) \longrightarrow (X,P)$ be a cover of connected orbifold curves. We will show that the maximal destabilizing sub-bundle $\text{\rm HN}(f_* \mathcal{O}_{(Y,Q)})_1$ of $f_* \mathcal{O}_{(Y,Q)}$ is a bundle on $(X,P)$ that is an $\mathcal{O}_{(X,P)}$-algebra, it is $P$-semistable of $P$-degree $0$, and the maximal \'{e}tale sub-cover of $f$ is the cover associated to $\text{\rm HN}(f_* \mathcal{O}_{(Y,Q)})_1$. This is shown in \cite{BP} for the cover of curves. We start with the following observations.

\begin{lemma}\label{lem_push_bundle}
Let $f_0 \colon Y \longrightarrow X$ be a cover of smooth projective connected $k$-curves. Let $P$ be a branch data on $X$. Consider the induced cover $f \colon (Y, f_0^*P) \longrightarrow (X,P)$. For any $E \in \text{\rm Vect}(Y, f_0^*P)$, we have
$$\mu_{P, \text{\rm max}}\left( f_*E \right) \leq \frac{\mu_{f_0^*P, \text{\rm max}}\left( E \right)}{\text{\rm deg}\left( f_0 \right)}.$$
\end{lemma}

\begin{proof}
First, suppose that $E$ is $f_0^*P$-semistable. Any sub-bundle $F$ of $f_*E$ uniquely corresponds to a non-zero morphism $f^*F \longrightarrow E$ via the $(f^*,f_*)$-adjunction in \cite[Proposition~9.3.6, page 205]{Olsson}. So for any $P$-semistable sub-bundle $F$ of $f_*E$, by Lemma~\ref{lem_pullback_div_line_bundles} and Proposition~\ref{prop_properties_stacky}\eqref{H:5}, we have
$$\text{\rm deg}(f_0) \cdot \mu_P(F) = \mu_{f_0^*P}(f^*F) \leq \mu_{f_0^*P}(E).$$
In particular,
$$\mu_{P, \text{\rm max}}(f_*E) = \mu_P(\text{\rm HN}(f_*E)_1) \leq \frac{\mu_{f_0^*P}(E)}{\text{\rm deg}(f_0)}.$$

For an arbitrary $E \in \text{\rm Vect}(X,P)$, consider the Harder-Narasimhan filtration in Proposition~\ref{prop_properties_stacky}\eqref{H:2}
$$0 = E_0 \subset E_1 \subset \cdots \subset E_l = E;$$
so each $E_{i+1}/E_i$ is $f_0^*P$-semistable, and $\mu_{f_0^*P}(E_{i+1}/E_i) \leq \mu_{f_0^*P, \text{\rm max}}(E) = \mu_{f_0^*P}(E_1)$.
Since $f_*$ is an exact functor by Lemma~\ref{lem_push_exactness_rep_cover}, the above produces a filtration
$$0 \subset f_*E_1 \subset \cdots \subset f_*E,$$
and $f_*(E_{i+1}/E_i) \cong f_*E_{i+1}/f_*E_i$. Applying the statement for each $f_0^*P$-semistable bundle $E_{i+1}/E_i$, we have
$$\mu_{P, \text{\rm max}}\left( f_*(E_{i+1}/E_i) \right) \leq \frac{\mu_{f_0^*P}\left( E_{i+1}/E_i \right)}{\text{\rm deg}\left( f_0 \right)}.$$
As each $\mu_{f_0^*P}\left( E_{i+1}/E_i \right) \leq \mu_{f_0^*P, \text{\rm max}}(E)$, we have
$$\mu_{P, \text{\rm max}}(f_*E) \leq \text{\rm max}_{1 \leq i \leq l} \left\lbrace \mu_{P, \text{\rm max}}\left( f_*E_{i+1}/f_*E_i \right) \right\rbrace \leq \mu_{f_0^*P, \text{\rm max}}(E)/\text{\rm deg}\left( f_0 \right).$$
\end{proof}

\begin{lemma}\label{lem_etale}
Under the hypotheses of Lemma~\ref{lem_push_bundle}, $\mu_{P, \text{\rm max}}(f_* \mathcal{O}_{(Y,f_0^*P)}) = 0$. Moreover, the following are equivalent.
\begin{enumerate}
\item The cover $f$ is an \'{e}tale cover.\label{et:1}
\item $f_* \mathcal{O}_{(Y,f_0^*P)}$ has $P$-degree $0$.\label{et:2}
\item $f_* \mathcal{O}_{(Y,f_0^*P)}$ is $P$-semistable.\label{et:3}
\end{enumerate}
\end{lemma}

\begin{proof}
We have a canonical inclusion $\mathcal{O}_{(X,P)} \hookrightarrow f_* \mathcal{O}_{(Y,f_0^*P)}$ of bundles. Then  Proposition~\ref{prop_properties_stacky}\eqref{H:5} implies that
$$\mu_{P, \text{\rm max}}(f_* \mathcal{O}_{(Y,f_0^*P)}) \geq 0.$$
This is an equality by Lemma~\ref{lem_push_bundle}.

Let $R$ be the ramification divisor for the cover $f$, defined in~\eqref{ramification_divisor}. Set $B \coloneqq f_*R$. By Lemma~\ref{lem_push_div_line_rep}\eqref{fin:4}, we have $\text{\rm det}(f_* \mathcal{O}_{(Y,f_0^*P)})^{\otimes 2} \cong \mathcal{O}_{(X,P)}(-B)$. Since $\mu_{P, \text{\rm max}}(f_* \mathcal{O}_{(Y,f_0^*P)}) = 0$, the equivalence of the statements follow.
\end{proof}

Recall from \cite[Section~10.2, page 210]{Olsson} that to an algebra $\mathcal{A}$ of $\mathcal{O}_{(X,P)}$-modules, there is an associated algebraic stack $\underline{\text{\rm Spec}}(\mathcal{A})$ together with a representable affine morphism $\pi \colon \underline{\text{\rm Spec}}(\mathcal{A}) \longrightarrow (X,P)$. For any $k$-scheme $T$, the objects of the category $\underline{\text{\rm Spec}}(\mathcal{A})(T)$ are pairs $(x,\rho)$ where $x \in (X,P)(T)$, and $\rho \colon x^*\mathcal{A} \longrightarrow \mathcal{O}_T$ is a morphism of sheaves of algebras on $T$. Also, for $x \in (X,P)(T)$, we have $\underline{\text{\rm Spec}}(\mathcal{A}) \times_{(X,P)} T \cong \underline{\text{\rm Spec}}(x^*\mathcal{A})$. Thus, if $\mathcal{A}$ is also a vector bundle on $(X,P)$, the morphism $\pi$ is a cover of orbifold curves. This way, a representable cover $f \colon (Y, f_0^*P) \longrightarrow (X,P)$ of orbifold curves corresponds to the $\mathcal{O}_{(X,P)}$-algebra $f_*\mathcal{O}_{(Y,f_0^*P)}$, and we have $(Y, f_0^*P) = \underline{\text{\rm Spec}}\left( f_*\mathcal{O}_{(Y,f_0^*P)} \right)$.

\begin{lemma}\label{lem_algebra}
Under the hypothesis of Lemma~\ref{lem_push_bundle}, consider the maximal destabilizing sub-bundle $\text{\rm HN}(f_* \mathcal{O}_{(Y,f_0^*P)})_1$ of $f_* \mathcal{O}_{(Y,f_0^*P)}$. The natural inclusion $\mathcal{O}_{(X,P)} \hookrightarrow f_* \mathcal{O}_{(Y,f_0^*P)}$ of bundles factors via the inclusion $\mathcal{O}_{(X,P)} \hookrightarrow \text{\rm HN}(f_* \mathcal{O}_{(Y,f_0^*P)})_1$ that equips the $P$-semistable bundle $\text{\rm HN}(f_* \mathcal{O}_{(Y,f_0^*P)})_1$ of $P$-degree $0$ with a structure of an $\mathcal{O}_{(X,P)}$-algebra.
\end{lemma}

\begin{proof}
By Lemma~\ref{lem_etale}, we know that $V \coloneqq \text{\rm HN}(f_* \mathcal{O}_{(Y,f_0^*P)})_1$ has $P$-degree $0$. We only need to show that $V$ is an $\mathcal{O}_{(X,P)}$-algebra; then by loc. cit. equivalences, $\underline{\text{\rm Spec}}(V) \longrightarrow (X,P)$ is an \'{e}tale cover via which $f$ factors, implying that the inclusion $\mathcal{O}_{(X,P)} \hookrightarrow f_* \mathcal{O}_{(Y,f_0^*P)}$ of bundles factors via the inclusion $\mathcal{O}_{(X,P)} \hookrightarrow V$, and $V$ is $P$-semistable.

Consider the $\mathcal{O}_{(X,P)}$-algebra homomorphism
$$\phi \colon f_* \mathcal{O}_{(Y,f_0^*P)} \otimes_{\mathcal{O}_{(X,P)}} f_* \mathcal{O}_{(Y,f_0^*P)} \longrightarrow f_* \mathcal{O}_{(Y,f_0^*P)}.$$
Set $W \coloneqq V \otimes V$. We want to show that $\phi(W) \subset V$.

We claim that $W$ is $P$-semistable of $P$-degree $0$. Since $\mu_{P, \text{\rm max}}(f_* \mathcal{O}_{(Y,f_0^*P)}/V) < 0$, by Proposition~\ref{prop_properties_stacky}\eqref{H:4}, this will imply that there is no non-zero homomorphism from $W$ to $f_* \mathcal{O}_{(Y,f_0^*P)}/V$, and hence $\phi(W) \subset V$.

Since $\mu_P(V) = 0$, we already have $\mu_P(W) = 2 \mu_p(V) = 0$. So we want to show that $W$ is $P$-semistable. It is enough to show that $W$ does not contain any non-zero sub-bundle of positive $P$-degree. Further, $W \subset f_* \mathcal{O}_{(Y,f_0^*P)} \otimes f_* \mathcal{O}_{(Y,f_0^*P)}$, and hence it is enough to show that $f_* \mathcal{O}_{(Y,f_0^*P)} \otimes f_* \mathcal{O}_{(Y,f_0^*P)}$ does not contain any sub-bundle of positive $P$-degree. We claim that $\mu_{P, \text{\rm max}}(f_* \mathcal{O}_{(Y,f_0^*P)} \otimes f_* \mathcal{O}_{(Y,f_0^*P)}) = 0$. This will conclude the proof.

Since $\phi$ is a non-zero homomorphism,
$$\mu_{P, \text{\rm max}}(f_* \mathcal{O}_{(Y,f_0^*P)} \otimes f_* \mathcal{O}_{(Y,f_0^*P)}) \leq \mu_P(V) = 0$$
by Proposition~\ref{prop_properties_stacky}\eqref{H:4}. By Proposition~\ref{prop_proj_bc}, we have
$$f_* \mathcal{O}_{(Y,f_0^*P)} \otimes f_* \mathcal{O}_{(Y,f_0^*P)} \cong f_*f^*f_* \mathcal{O}_{(Y,f_0^*P)}.$$
So by Lemma~\ref{lem_push_div_line_rep}\eqref{fin:4} and Proposition~\ref{prop_properties_stacky}\eqref{H:4},
$$\mu_{P, \text{\rm max}}(f_*f^*f_* \mathcal{O}_{(Y,f_0^*P)}) \leq \mu_{f_0^*P, \text{\rm max}}(f^*f_* \mathcal{O}_{(Y,f_0^*P)}) = \text{\rm deg}(f_0) \cdot \mu_P(V) = 0.$$ 
\end{proof}

We are ready to state the main result of this section. Any cover $f \colon (Y,Q) \longrightarrow (X,P)$ induces a homomorphism of the respective \'{e}tale fundamental groups, and the image of $\pi_1(Y,Q)$ in $\pi_1(X,P)$ is a sub-group of finite index, defining an intermediate sub-cover $(X',P') \longrightarrow (X,P)$ that is the maximal \'{e}tale sub-cover of $f$; see \cite[Proposition~2.42]{KP}. We show that $(X',P')$ coincides with $\underline{\text{\rm Spec}}\left( \text{\rm HN}( f_* \mathcal{O}_{(Y,Q)} )_1 \right)$.

\begin{proposition}\label{prop_maximal_etal_subcover}
Let $f \colon (Y,Q) \longrightarrow (X,P)$ be a cover of connected orbifold curves. Denote the induced cover on the Coarse moduli curves by $f_0 \colon Y \longrightarrow X$. Then there is the maximal sub-cover $g_0 \colon \hat{X} \longrightarrow X$ of $f_0$ that is an is essentially \'{e}tale cover of $(X,P)$, and $(\hat{X}, g_0^*P) \coloneqq \underline{\text{\rm Spec}}\left( \text{\rm HN}( f_* \mathcal{O}_{(Y,Q)} )_1 \right)$. Further, $f$ factors as a composition
$$f \colon (Y,Q) \longrightarrow (Y, f_0^*P) \overset{\hat{g}} \longrightarrow (\hat{X}, g_0^*P) = \underline{\text{\rm Spec}}\left( \text{\rm HN}( f_* \mathcal{O}_{(Y,Q)} )_1 \right) \overset{g} \longrightarrow (X,P)$$
where $g$ is the maximal \'{e}tale sub-cover of both $f$ and $(Y,f_0^*P) \longrightarrow (X,P)$.
\end{proposition}

\begin{proof}
By \cite[Proposition~2.42]{KP}, there is the maximal sub-cover $(\hat{X}, g_0^*P) \longrightarrow (X,P)$ of $f$:
$$f \colon (Y,Q) \longrightarrow (Y, f_0^*P) \overset{\hat{g}} \longrightarrow (\hat{X}, g_0^*P) \overset{g} \longrightarrow (X,P).$$
We need to prove that $\underline{\text{\rm Spec}}\left( \text{\rm HN}\left( f_* \mathcal{O}_{(Y,Q)} \right)_1 \right) = (\hat{X}, g_0^*P)$. By Lemma~\ref{lem_etale}, $g_* \mathcal{O}_{(\hat{X},g_0^*P)}$ is a $P$-semistable bundle of $P$-degree zero. Also since $g_* \mathcal{O}_{(\hat{X},g_0^*P)} \subset f_* \mathcal{O}_{(Y,Q)} = (g \circ \hat{g})_* \mathcal{O}_{(Y,f_0^*P)}$, by Proposition~\ref{prop_properties_stacky}\eqref{H:1}, we have $g_* \mathcal{O}_{(\hat{X},g_0^*P)} \subset \text{\rm HN}\left( f_* \mathcal{O}_{(Y,Q)} \right)_1$. By \cite[Theorem~10.2.4, pg. 212]{Olsson}, we obtain a composition of covers
$$\underline{\text{\rm Spec}}\left( \text{\rm HN}\left( f_* \mathcal{O}_{(Y,Q)} \right)_1 \right) \longrightarrow \underline{\text{\rm Spec}}\left( g_* \mathcal{O}_{(\hat{X},g_0^*P)} \right) = (\hat{X}, g_0^*P) \longrightarrow (X,P)$$
of orbifold curves, and the composition is an \'{e}tale cover by Lemma~\ref{lem_etale}. Since $g$ is the maximal cover with this property, we have $(\hat{X},g_0^*P) = \underline{\text{\rm Spec}}\left( \text{\rm HN}\left( f_* \mathcal{O}_{(Y,Q)} \right)_1 \right)$.
\end{proof}

\begin{remark}\label{rmk_factorization_discussion}
Following \cite[Proposition~2.42]{KP}, we also see that the homomorphism
$$f_* \colon \pi_1(Y,Q) \longrightarrow \pi_1(X,P)$$
of the \'{e}tale fundamental groups, induced by the cover $f$, factors as a composition of a surjection $\pi_1(Y,Q) \twoheadrightarrow \pi_1(\hat{X}, g_0^*P)$, induced by the cover $(Y,Q) \longrightarrow (\hat{X}, g_0^*P)$, followed by the injection $g_* \colon \pi_1(\hat{X}, g_0^*P) \hookrightarrow \pi_1(X,P)$, induced by the \'{e}tale cover $g$.
\end{remark}

\begin{remark}
(Higher dimensional aspect) We have shown that a cover of orbifold curves has a maximal representable sub-cover, and the later cover admits a further maximal \'{e}tale sub-cover. In stead, one can consider smooth proper DM stacks admitting smooth proper Coarse moduli schemes. Then defining a cover as a finite flat surjective morphism of stacks, one can argue that every cover has a maximal representable sub-cover which further admits a maximal \'{e}tale sub-cover. Since we do not have a well equipped notion of slope stability for bundles on DM stacks in general, we are unable to relate this to the theory of bundles in this general set up; although, for the projective DM stacks, it seems feasible to construct such a correspondence using the generating sheaves and the slope stability condition defined by them as in \cite{Nironi}.
\end{remark}

We conclude this section with the following useful result.

\begin{lemma}\label{lem_cohomology_vanishing}
Under the notation of Proposition~\ref{prop_maximal_etal_subcover}, set $V \coloneqq \text{\rm HN}\left( f_*\mathcal{O}_{(Y,f_0^*P)} \right)_1$, and $W \coloneqq f_*\mathcal{O}_{(Y,f_0^*P)}/V$. Then the following hold.
\begin{enumerate}
    \item $\mu_{P, \text{\rm max}}(W) < 0$;
    \item $\mathrm{H}^0((Y,f_0^*P), f^*f_*\mathcal{O}_{(Y,f_0^*P)}) \cong \mathrm{H}^0((Y, f_0^*P), f^*V)$.
\end{enumerate}
\end{lemma}

\begin{proof}
We note that
\[
\mu_{P,\text{\rm max}}(W) = \mu_P(\text{\rm HN}(W)_1) = \mu_P(\text{\rm HN}(f_* \mathcal{O}_{(Y,f_0^*P)})_2/V) < \mu_P(V) = 0,
\]
proving the first result. By Proposition~\ref{prop_properties_orbi_cover}~\eqref{s:b1},
\[
\mu_{f_0^*P, \text{\rm max}} \left( f^* W \right) = \text{\rm deg}(f_0) \cdot \mu_{P, \text{\rm max}}(W) < 0.
\]
Now, we have a short exact sequence
\[
0 \longrightarrow V \longrightarrow f_* \mathcal{O}_{(Y,f_0^*P)} \longrightarrow W \longrightarrow 0
\]
in $\text{\rm Vect}(X,P)$. Applying the exact functor $f^*$ produces a short exact sequence
\begin{equation}\label{eq_ses_f^*}
0 \longrightarrow f^* V \longrightarrow f^*f_* \mathcal{O}_{(Y,f_0^*P)} \longrightarrow  f^*f_* \mathcal{O}_{(Y,f_0^*P)}/f^* V \cong f^*W \longrightarrow 0
\end{equation}
in $\text{\rm Vect}(Y,f_0^*P)$. Suppose that there is a non-zero section in $\mathrm{H}^0((Y,f_0^*P), f^*W)$. This corresponds to a non-zero homomorphism $\mathcal{O}_{(Y,f_0^*P)} \longrightarrow f^*W$, which is not possible by Proposition~\ref{prop_properties_stacky}~\eqref{H:5}. Then the long exact sequence of cohomologies associated to the exact sequence~\eqref{eq_ses_f^*} shows that
\[
\mathrm{H}^0((Y,f_0^*P), f^* f_* \mathcal{O}_{(Y,f_0^*P)}) \cong \mathrm{H}^0((Y,f_0^*P), f^* V).
\]
\end{proof}

\section{Genuinely Ramified Covers}\label{sec_gen_ram}
\subsection{Equivalent condition}\label{sec_equivalence_gen_ram}
Historically, the notion of a genuinely ramified morphism arises in the study of covers of the projective line and in the construction of the Hurwitz spaces. Recent study in \cite{BP} on the slope stability of a bundle on smooth curves under finite covers shows another importance of such morphisms. Our objective in this section is to extend the definition of a genuinely ramified morphism from smooth curves to orbifold curves. Let us start by recalling the definition in the case of curves, following \cite{BP}.

Consider any non-trivial cover $f \colon Y \longrightarrow X$ of smooth projective connected $k$-curves. The pushforward sheaf $f_* \mathcal{O}_Y$ is a vector bundle on $X$; it is semistable if and only if $f$ is an \'{e}tale cover by \cite[Lemma~2.3]{BP}. The maximal destabilizing sub-bundle $\text{\rm HN}(f_*\mathcal{O}_Y)_1 \subset f_* \mathcal{O}_Y$ is of degree $0$, and the natural inclusion $\mathcal{O}_X \subset \text{\rm HN}\left( f_* \mathcal{O}_Y \right)_1$ of bundles on $X$ equips $\text{\rm HN}\left( f_* \mathcal{O}_Y \right)_1$ with a structure of $\mathcal{O}_X$-algebras. Moreover, the cover $f$ factors as a composition
\begin{equation}\label{eq_factor_curves_case}
f \colon Y \longrightarrow \hat{X} \coloneqq \underline{\Spec}\left( \text{\rm HN}(f_*\mathcal{O}_Y)_1 \right) \longrightarrow X
\end{equation}
where $\hat{X} \longrightarrow X$ is the maximal \'{e}tale sub-cover of $f$; see \cite[Lemma~2.4, Corollary~2.7]{BP}. The cover $f$ is said to be \textit{genuinely ramified} if $\text{\rm HN}\left( f_*\mathcal{O}_Y \right)_1$ is $\mathcal{O}_X$. This is equivalent to: the homomorphism between \'{e}tale fundamental groups $f_* \colon \pi_1(Y) \longrightarrow \pi_1(X)$, induced by $f$, is a surjection. Other equivalent conditions are given in \cite[Proposition~2.6, Lemma~3.1]{BP}. One of the main results of \cite{BP} gives another important equivalent criterion for the cover $f$ to be genuinely ramified in terms of the slope stability of bundles via the pullback under $f$.

\begin{theorem}[{\cite[Theorem~1.1]{BP}}]
Let $f \colon Y \longrightarrow X$ be a non-trivial cover of smooth projective connected $k$-curves. Then $f$ is genuinely ramified if and only if for every stable bundle $E$ on $X$, the pullback bundle $f^*E$ is stable on $Y$.
\end{theorem}

In Proposition~\ref{prop_maximal_etal_subcover}, we saw that a cover of orbifold curves factors in a similar way as in~\ref{eq_factor_curves_case}, where the maximal \'{e}tale sub-cover of a cover is identified with the (necessarily representable) cover associated with the maximal destabilizing sub-bundle of the pushforward of the structure sheaf. In the following, we establish the equivalent conditions on covers of orbifold curves, which do not have a non-trivial \'{e}tale sub-cover. We will prove in the next section (Theorem~\ref{thm_main} and Theorem~\ref{thm_converse}) that this class of covers is precisely the one preserving orbifold slope stability conditions.

\begin{proposition}\label{prop_gen_ram_equivalences}
Let $f \colon (Y,Q) \longrightarrow (X,P)$ be a cover of connected orbifold curves. The following are equivalent.
\begin{enumerate}[leftmargin=*]
\item $\text{\rm HN}(f_*\mathcal{O}_{(Y,Q)})_1 = \mathcal{O}_{(X,P)}$.\label{equiv:1}
\item The cover $f$ does not factor through any non-trivial \'{e}tale sub-cover.\label{equiv:2}
\item The homomorphism between \'{e}tale fundamental groups $f_* \colon \pi_1(Y,Q) \longrightarrow \pi_1(X,P)$ induced by $f$ (cf. \cite[Proposition~2.26]{KP}) is a surjection.\label{equiv:3}
\item For any \'{e}tale cover $(Z,R) \longrightarrow (X,P)$ of connected orbifold curves, the fiber product orbifold curve $(Y,Q) \times_{(X,P)} (Z,R)$ is connected.\label{equiv:3'}
\item The fiber product stacky curve $(Y,Q) \times_{(X,P)} (Y,Q)$ is connected.\label{equiv:4}
\item $H^0( (Y,Q) , f^* f_* \mathcal{O}_{(Y,Q)}) \cong k$.\label{equiv:5}
\end{enumerate}
Finally, the above conditions imply that the cover $f_0 \colon Y \longrightarrow X$ induced on the Coarse moduli curves is genuinely ramified.
\end{proposition}

\begin{proof}
It is not hard to see that each of the statements \eqref{equiv:1}--\eqref{equiv:5} for the branch data $Q$ is in fact equivalent to the corresponding statement with $Q$ replaced by $f_0^*P$ and $f$ replaced by the induced cover $(Y,f_0^*P) \longrightarrow (X,P)$; here, the crucial points to notice are the following (whose detail we reserve for the reader).
\begin{enumerate}[leftmargin=*, label=(\alph*)]
    \item $f_*\mathcal{O}_{(Y,Q)} \cong \mathcal{O}_{(X,P)} \cong f_*\mathcal{O}_{(Y,f_0^*P)}$ by Lemma~\ref{lem_push_div_line_nonrep};
    \item the maximal \'{e}tale sub-covers for $(Y,Q) \longrightarrow (X,P)$ and for $(Y,f_0^*P) \longrightarrow (X,P)$ coincide;
    \item by the property of the Coarse moduli space (cf. Definition~\ref{def_KM}), a stacky curve $\mathfrak{Y}$ is connected if and only if the Coarse moduli curve $Y$ is connected if and only if
    \[
    \mathrm{H}^0(\mathfrak{Y}, \mathcal{O}_{\mathfrak{Y}}) = \mathrm{H}^0(Y, \mathcal{O}_Y) \cong k
    \]
\end{enumerate}
In the following, we assume $Q = f_0^*P$, and $f \colon (Y,f_0^*P) \longrightarrow (X,P)$.

By Proposition~\ref{prop_maximal_etal_subcover}, the maximal \'{e}tale sub-cover of $f$ is given by
$$(\hat{X}, g_0^*P) = \underline{\text{\rm Spec}}\left( \text{\rm HN}( f_* \mathcal{O}_{(Y,Q)} )_1 \right) \overset{g} \longrightarrow (X,P)$$
where $g_0$ is the maximal sub-cover of $f_0 \colon Y \longrightarrow X$ that is an essentially \'{e}tale cover of $(X,P)$. Since the degree of the later cover $g$ is the rank of the bundle $\text{\rm HN}( f_* \mathcal{O}_{(Y,Q)} )_1$, \eqref{equiv:2} is equivalent to: $\text{\rm HN}( f_* \mathcal{O}_{(Y,Q)} )_1$ is a line bundle. As $\mathcal{O}_{(X,P)} \subset \text{\rm HN}( f_* \mathcal{O}_{(Y,Q)} )_1$, and $\mu_P(\mathcal{O}_{(X,P)}) = \mu_P(\text{\rm HN}( f_* \mathcal{O}_{(Y,Q)} )_1) = 0$ by Lemma~\ref{lem_algebra}, we see that \eqref{equiv:1} and \eqref{equiv:2} are equivalent.

By \cite[Proposition~2.42]{KP}, the image of the homomorphism $f_* \colon \pi_1(Y,Q) \longrightarrow \pi_1(X,P)$, induced by the cover $f$, is a finite index open subgroup of $\pi_1(X,P)$, and this corresponds to the essentially \'{e}tale cover $g_0 \colon \hat{X} \longrightarrow X$. So $f_*$ is a surjective homomorphism if and only if $g_0 = \text{\rm id}_X$, and we have \eqref{equiv:2}$\Leftrightarrow$\eqref{equiv:3}.

\eqref{equiv:3} is equivalent to \eqref{equiv:3'} by \cite[\href{https://stacks.math.columbia.edu/tag/0BN6}{Lemma 0BN6}]{SP} together with the observation in Remark~\ref{rmk_etale_fundamental_group} that a connected object in the category $\text{\rm \'{E}t}_{(X,P)}$ of \'{e}tale covers of $(X,P)$ corresponds to an essentially \'{e}tale cover $Y \longrightarrow X$ of $(X,P)$ where $Y$ is connected.

We show the equivalence \eqref{equiv:4}$\Leftrightarrow$\eqref{equiv:5}. By Proposition~\ref{prop_proj_bc}, we have an isomorphism
$$f^*f_* \mathcal{O}_{(Y,Q)} \cong (p_1)_* p_2^* \mathcal{O}_{(Y,Q)}$$
where $p_1$ and $p_2$ are the projection morphisms $(Y,Q) \times_{(X,P)} (Y,Q) \longrightarrow (Y,Q)$. Since $p_2^* \mathcal{O}_{(Y,Q)} = \mathcal{O}_{(Y,Q) \times_{(X,P)} (Y,Q)}$, we conclude that
\begin{eqnarray*}
\mathrm{H}^0((Y,Q) , f^* f_* \mathcal{O}_{(Y,Q)}) & = & \mathrm{H}^0((Y,Q) , (p_1)_* \mathcal{O}_{(Y,Q) \times_{(X,P)} (Y,Q)})\\
& = & \mathrm{H}^0((Y,Q) \times_{(X,P)} (Y,Q), \mathcal{O}_{(Y,Q) \times_{(X,P)} (Y,Q)}).
\end{eqnarray*}
Note that $(Y,Q) \times_{(X,P)} (Y,Q)$ is a reduced DM stack, each of whose irreducible components is an orbifold curve. Let $(Y,Q) \times_{(X,P)} (Y,Q) \longrightarrow S$ be the Coarse moduli morphism. By the property of the Coarse moduli space (cf. Definition~\ref{def_KM}), 
$$\mathrm{H}^0((Y,Q) \times_{(X,P)} (Y,Q), \mathcal{O}_{(Y,Q) \times_{(X,P)} (Y,Q)}) \cong \mathrm{H}^0(S, \mathcal{O}_S).$$
Since $(Y,Q) \times_{(X,P)} (Y,Q)$ is connected if and only if $S$ is connected, the equivalence follows.

Now we show that \eqref{equiv:4}$\Rightarrow$\eqref{equiv:2}. Suppose that $f$ factors as a composition
$$(Y,Q) \longrightarrow (\hat{X}, g_0^*P) \overset{g} \longrightarrow (X,P)$$
where $g$ is a non-trivial \'{e}tale cover. Then the fiber product $(\hat{X}, g_0^*P) \times_{(X,P)} (\hat{X}, g_0^*P)$ is disconnected as it contains $(\hat{X}, g_0^*P)$ as a connected component; this follows because the diagonal morphism is an open imbedding, and $g_0$ is non-trivial. Since we have the induced cover $(Y,Q) \times_{(X,P)} (Y,Q) \longrightarrow (\hat{X}, g_0^*P) \times_{(X,P)} (\hat{X}, g_0^*P)$, we conclude that the fiber product stacky curve $(Y,Q) \times_{(X,P)} (Y,Q)$ is also not connected.

It remains to show: \eqref{equiv:1}$\Rightarrow$\eqref{equiv:5}. Set $V \coloneqq \text{\rm HN}(f_* \mathcal{O}_{(Y,Q)})_1$, and $W \coloneqq f_* \mathcal{O}_{(Y,Q)}/V$. By Lemma~\ref{lem_cohomology_vanishing}, we have:
\[
\mathrm{H}^0((Y,Q), f^* f_* \mathcal{O}_{(Y,Q)}) \cong \mathrm{H}^0((Y,Q), f^* V)
\]
By our assumption, $V \cong \mathcal{O}_{(X,P)}$. Since $Y$ is connected, we have
\[
\mathrm{H}^0((Y,Q), f^* f_* \mathcal{O}_{(Y,Q)}) \cong \mathrm{H}^0((Y,Q), \mathcal{O}_{(Y,Q)}) \cong k,
\]
and the result follows.

To see the last statement, assume that $f_0 \colon Y \longrightarrow X$ is not genuinely ramified. Then there is a non-trivial \'{e}tale sub-cover $g_0 \colon X' \longrightarrow X$ of $f_0$. In particular, $g_0$ is an essentially \'{e}tale cover of $(X,P)$. Then $(X',g_0^*P) \longrightarrow (X,P)$ is a non-trivial \'{e}tale sub-cover of $f$.
\end{proof}

\begin{definition}\label{def_gen_ram}
A non-trivial cover $f \colon (Y,Q) \longrightarrow (X,P)$ of connected orbifold curves is said to be \textit{genuinely ramified} if $f$ satisfies the equivalent conditions from Proposition~\ref{prop_gen_ram_equivalences}.
\end{definition}

Given a cover $f \colon (Y,Q) \longrightarrow (X,P)$ of connected orbifold curves and a bundle $E \in \text{\rm Vect}(Y,Q)$, it is natural to ask whether there exists a bundle $F \in \text{\rm Vect}(X,P)$ such that $E \cong f^*F$. The following result gives an answer when $f$ is genuinely ramified and $E$ is $Q$-stable (see \cite[Theorem~3.2]{BDP} for the case of normal varieties).

\begin{theorem}\label{thm_descent}
Let $f \colon (Y,Q) \longrightarrow (X,P)$ be a cover of connected orbifold curves. Let $E \in \Vect(Y,Q)$ be $Q$-stable. If there exists a bundle $F \in \Vect(X,P)$ such that $E \cong f^*F$, then the $P$-stable sub-bundle $F$ of $f_*E$ satisfies $\mu_P(F) = \frac{\mu_Q(E)}{\text{\rm deg}(f_0)}$.

The converse is also true when $f$ is genuinely ramified.
\end{theorem}

\begin{proof}
Suppose that $E \cong f^*F$ for some $F \in \Vect(X,P)$. By Proposition~\ref{prop_properties_stacky}\eqref{s:b1}, we have $\mu_P(F) = \frac{\mu_Q(E)}{\text{\rm deg}(f_0)}$. By Remark~\ref{rmk_iota_preserve}, $F$ is $P$-stable. Then by Proposition~\ref{prop_proj_bc}\eqref{pf},
$$f_*E \cong f_*f^*F \cong F \otimes f_* \mathcal{O}_{(Y,Q)}.$$
Since $\mathcal{O}_{(X,P)} \subset f_* \mathcal{O}_{(Y,Q)}$ is a sub-bundle, and tensor product with a bundle is an exact functor, $F$ is realized as a sub-bundle of $f_*E$.

Now suppose that $f$ is genuinely ramified, and $f_*E$ contains a $P$-stable sub-bundle $F$ such that $\mu_P(F) = \frac{\mu_Q(E)}{\text{\rm deg}(f_0)}$. By Theorem~\ref{thm_main}, $f^*F$ is $Q$-stable. The morphism $h \colon f^*F \longrightarrow E$, the adjoint of the inclusion $F \hookrightarrow f_*E$, is a non-zero homomorphism of $Q$-stable bundles of the same $Q$-slopes. So $h$ is an isomorphism (the proof is similar to the curve case).
\end{proof}

\subsection{Pullback under a Genuinely Ramified Map}\label{sec_main-direction}
In this section, we establish that a cover $f \colon (Y,Q) \longrightarrow (X,P)$ of connected orbifold curves is genuinely ramified (see Definition~\ref{def_gen_ram}) if and only if the pullback of any $P$-stable bundle is $Q$-stable; this is the main result of \cite{BP} for the cover of curves.

One of our key observations is that when $(Y,Q)$ is a geometric orbifold curve, $f_0 \colon Y \longrightarrow X$ is a non-trivial Galois cover, and $f$ is genuinely ramified, then for any $P$-semistable bundle $F$ on $(X,P)$ produces a strict inequality:
\begin{equation}\label{eq_re}
\mu_Q(f^*F \otimes f^*f_*\mathcal{O}_{(Y,Q)}/\mathcal{O}_{(Y,Q)}) < \mu_P(f^*F).    
\end{equation}
For an arbitrary representable Galois cover $f$, from Lemma~\ref{lem_push_exactness_rep_cover} and Lemma~\ref{lem_cohomology_vanishing}, we see that the maximal destabilizing sub-bundle of $f^*f_*\mathcal{O}_{(Y,Q)}$ is given by $f^*\text{\rm HN}(f_*\mathcal{O}_{(Y,Q)})_1$ which is of $Q$-slope $0$. The subsequent quotients in the Harder-Narasimhan filtration of the bundle $f^*f_*\mathcal{O}_{(Y,Q)}/f^*\text{\rm HN}(f_*\mathcal{O}_{(Y,Q)})_1$ are $Q$-semistable of negative $Q$-slope. Then $f^*F \otimes \left( f^*f_*\mathcal{O}_{(Y,Q)}/f^*\text{\rm HN}(f_*\mathcal{O}_{(Y,Q)})_1 \right)$ admits a filtration by sub-bundles such that the successive quotients are bundles of negative $Q$-slope, whereas they may not be $Q$-semistable when $\text{\rm char}(k)>0$. The assumption that $f$ is genuinely ramified makes sure that all these quotients are $Q$-semistable, enabling us to prove~\eqref{eq_re}. More precisely, we show that the bundle $f^* \left( f_* \mathcal{O}_{(Y,Q)} / \mathcal{O}_{(X,P)} \right)$ admits a filtration such that the successive quotients are line bundles of negative $Q$-degrees; this is shown in the next two propositions. For the usual curve case, this follows from \cite[Proposition~5.13, pg. 76]{E2}.

\begin{proposition}\label{prop_inclusion}
Let $f \colon (Y,Q) \longrightarrow (X,P)$ be a cover or connected orbifold curves such that the cover $f_0 \colon Y \longrightarrow X$ of the Coarse moduli curves is $G$-Galois for some finite group $G$ with $|G| = d \geq 2$. Then we have an inclusion
$$f^* \left( f_* \mathcal{O}_{(Y,Q)} / \mathcal{O}_{(X,P)} \right) \subset \mathcal{O}^{\oplus (d-1)}_{(Y,Q)}$$
as coherent sheaves on $(Y,Q)$.
\end{proposition}

\begin{proof}
Set $\mathfrak{Y} \coloneqq (Y,Q)$ and $\mathfrak{X} \coloneqq (X,P)$. When $P$ and $Q$ are the trivial branch data, the statement is \cite[Proposition~3.3]{BP}.

We consider the fiber product DM stack $\mathfrak{Y} \times_{\mathfrak{X}} \mathfrak{Y}$, and its normalization orbifold curve $(\widetilde{Y \times_X Y}, Q \times_P Q)$ (see \cite[Proposition~2.14]{KP}) where $\widetilde{Y \times_X Y}$ denote the normalization of the curve $Y \times_X Y$. We have the following commutative diagram.
\begin{center}
\begin{tikzcd}
\tilde{\mathfrak{Y}} \coloneqq (\widetilde{Y \times_X Y}, Q \times_P Q)
\arrow[drr, bend left, "h"]
\arrow[ddr, bend right, "\theta"]
\arrow[dr,  "\nu"] & & \\
& \mathfrak{Y} \times_{\mathfrak{X}} \mathfrak{Y} \arrow[dr, phantom, "\square"] \arrow[r, "p_{2}"] \arrow[d, "p_{1}"]
& \mathfrak{Y} = (Y,Q) \arrow[d, "f"] \\
& \mathfrak{Y} = (Y,Q) \arrow[r, "f"]
& \mathfrak{X} = (X,P)
\end{tikzcd}
\end{center}
We have $(p_1)_* \mathcal{O}_{\mathfrak{Y} \times_{\mathfrak{X}} \mathfrak{Y}} \subset \theta_* \mathcal{O}_{\tilde{\mathfrak{Y}}} \cong \mathcal{O}_{\mathfrak{Y}} \otimes k[G]$. Under this inclusion, $\mathcal{O}_{\mathfrak{Y}} \subset (p_1)_* \mathcal{O}_{\mathfrak{Y} \times_{\mathfrak{X}} \mathfrak{Y}}$ is mapped isomorphically onto $\mathcal{O}_{\mathfrak{Y}} \subset \theta_* \mathcal{O}_{\tilde{\mathfrak{Y}}}$, which in turn defines a line sub-bundle $L \cong \mathcal{O}_{\mathfrak{Y}}$ of $\mathcal{O}_{\mathfrak{Y}} \otimes k[G]$. Any subspace of $H^0(\mathfrak{Y}, \mathcal{O}_{\mathfrak{Y}} \otimes k[G]) = k[G]$ is a direct summand. So there exists a trivial sub-bundle $W \subset \mathcal{O}_{\mathfrak{Y}} \otimes k[G]$ such that $\mathcal{O}_{\mathfrak{Y}} \otimes k[G] = W \oplus L$. Thus we obtain an inclusion
$$(p_1)_* \mathcal{O}_{\mathfrak{Y} \times_{\mathfrak{X}} \mathfrak{Y}}/ \mathcal{O}_{\mathfrak{Y}} \subset W \cong \mathcal{O}_{\mathfrak{Y}}^{\oplus (d-1)}.$$
This produces the desired inclusion
$$f^* \left( f_* \mathcal{O}_{(Y,Q)} / \mathcal{O}_{(X,P)} \right) \cong f^*f_* \mathcal{O}_{(Y,Q)}/f^* \mathcal{O}_{(X,P)} \cong (p_1)_* \mathcal{O}_{\mathfrak{Y} \times_{\mathfrak{X}} \mathfrak{Y}}/ \mathcal{O}_{\mathfrak{Y}} \subset \mathcal{O}_{\mathfrak{Y}}^{\oplus (d-1)}$$
where the first isomorphism holds because $f^*$ is an exact functor, and the second isomorphism is from Proposition~\ref{prop_proj_bc}.
\end{proof}

\begin{proposition}\label{prop_negative}
Under the hypothesis of Proposition~\ref{prop_inclusion}, additionally assume that $f$ is a genuinely ramified cover and $(Y,Q)$ is geometric. Set $E \coloneqq f^* \left( f_* \mathcal{O}_{(Y,Q)} / \mathcal{O}_{(X,P)} \right) \in \Vect(Y,Q)$. Then $E$ admits a filtration
$$E = E_1 \supset E_2 \supset \ldots \supset E_{d}=0$$
of sub-bundles such that each $E_i/E_{i+1}$ is a line bundle of negative $Q$-degree.
\end{proposition}

\begin{proof}
If $d = 2$, then $E$ is a line bundle of negative $Q$-degree since the cover $f$ is genuinely ramified. Assume that $d \geq 3$. We again set $\mathfrak{Y} \coloneqq (Y,Q)$ and $\mathfrak{X} \coloneqq (X,P)$. When $P$ and $Q$ are the trivial branch data, the statement is \cite[Proposition~5.13, pg. 76]{E2}.

By Proposition~\ref{prop_inclusion}, we have an inclusion $\alpha \colon E \hookrightarrow \mathcal{O}_{\mathfrak{Y}}^{\oplus (d-1)}$. By Lemma~\ref{lem_cohomology_vanishing}, $H^0(\mathfrak{Y}, E) = 0$. Thus $\alpha$ cannot be an isomorphism. So $\text{\rm coker}(\alpha)$ is a torsion sheaf. Let $g \colon Z \longrightarrow \mathfrak{Y}$ be a $\Gamma$-Galois \'{e}tale cover where $Z$ is a smooth projective connected $k$-curve. Then we have a $\Gamma$-equivariant inclusion in $\Vect^{\Gamma}(Z)$:
$$g^* \alpha \colon \mathcal{E} \coloneqq g^* E \hookrightarrow \mathcal{O}_Z^{\oplus (d-1)}.$$
So $\text{\rm coker}(g^* \alpha)$ is a $\Gamma$-equivariant torsion sheaf and has a finite $\Gamma$-invariant support. Let $z \in Z$ be in the support. Set $\Gamma z$ for the $\Gamma$-orbit of the point $z$. We can choose a $\Gamma$-equivariant epimorphism $\mathcal{O}_Z^{\oplus (d-1)}/\mathcal{E} \twoheadrightarrow \oplus_{z' \in \Gamma z} \mathcal{O}_{z'}$, where $\mathcal{O}_{z'}$ is the skyscraper sheaf at $z'$. Consider the coherent torsion sheaf $F \coloneqq g_*^G(\oplus_{z' \in \Gamma z} \mathcal{O}_{z'})$ on $\mathfrak{Y}$; it is a skyscraper sheaf on the gerbe $y$ for which $g^*y = \Gamma z$. Then we obtain an epimorphism $\mathcal{O}_{\mathfrak{Y}}^{\oplus (d-1)}/E \twoheadrightarrow F$. Since $\mathcal{O}_{\mathfrak{Y}}^{\oplus (d-1)}$ is generated by its global sections, the images of the global sections generate $F$. So the map $k^{d-1} = H^0(\mathfrak{Y}, \mathcal{O}_{\mathfrak{Y}}^{\oplus (d-1)}) \longrightarrow H^0(\mathfrak{Y}, F) = k$ is surjective, whose kernel has dimension $d-2$. Since any subspace of $H^0(\mathfrak{Y}, \mathcal{O}_{\mathfrak{Y}}^{\oplus (d-1)})$ generate a direct summand, we obtain a summand $\mathcal{O}_{\mathfrak{Y}}^{\oplus (d-2)}$ of $\mathcal{O}_{\mathfrak{Y}}^{\oplus (d-1)}$ together with a map $\mathcal{O}_{\mathfrak{Y}}^{\oplus (d-1)} \longrightarrow \mathcal{O}_{\mathfrak{Y}}^{\oplus (d-1)}/E$ whose image is the skyscraper sheaf supported on the gerbe $y$. The map $\mathcal{O}_{\mathfrak{Y}}^{\oplus (d-1)} \longrightarrow F$ factors through the quotient $\mathcal{O}_{\mathfrak{Y}}^{\oplus (d-1)}/ \mathcal{O}_{\mathfrak{Y}}^{\oplus (d-2)} = \mathcal{O}_{\mathfrak{Y}}$ as in the following diagram.
\begin{center}
\begin{tikzcd}
 & \mathcal{O}_{\mathfrak{Y}}^{\oplus (d-2)} \arrow[d, hookrightarrow] \arrow[dr] & \\
 E \arrow[r, hookrightarrow, "\alpha"] \arrow[rd, "\beta"] & \mathcal{O}_{\mathfrak{Y}}^{\oplus (d-1)} \arrow[r, ->>] \arrow[d, ->>] & \mathcal{O}_{\mathfrak{Y}}^{\oplus (d-1)}/E \arrow[d, ->>] \\
 & \mathcal{O}_{\mathfrak{Y}} \arrow[r, ->>] & F
\end{tikzcd}
\end{center}
As the composite homomorphism $E \longrightarrow F$ is zero, the map $\beta \colon E \overset{\alpha} \hookrightarrow \mathcal{O}_{\mathfrak{Y}}^{\oplus (d-1)} \longrightarrow \mathcal{O}_{\mathfrak{Y}}$ is not a surjection. So the ideal sheaf $L_1 = \beta(E) \subsetneq \mathcal{O}_{\mathfrak{Y}}$ defines a non-empty finite substack $\mathfrak{Y}' \subset \mathfrak{Y}$ corresponding to a $\Gamma$-invariant finite set of points on $Z$, and $\text{\rm deg}_Q(L_1) = - \text{\rm deg}_Q(\mathfrak{Y}') < 0$. Since $(Y,Q)$ is smooth, $L_1$ is a line bundle. Thus we obtain a surjection $E \twoheadrightarrow L_1$ with $\text{\rm deg}_Q(L_1) < 0$. The kernel $E'$ of this surjection is a vector bundle of rank $d - 2$ satisfying: $E' \subset \mathcal{O}_{\mathfrak{Y}}^{\oplus (d-2)}$, and $H^0(\mathfrak{Y}, E') = 0$. Inductively, we obtain the desired filtration.
\end{proof}

We have the following easy consequence.

\begin{corollary}\label{cor_main}
Under the hypothesis of Proposition~\ref{prop_negative}, for any two $P$-semistable vector bundle $E, \, F \in \Vect(X,P)$ with $\mu_P(E) = \mu_P(F)$, we have
$$\text{\rm Hom}_{\Vect(Y,Q)}\left( f^*E, f^*F \right) = \text{\rm Hom}_{\Vect(X,P)}\left(E, F\right).$$
\end{corollary}

\begin{proof}
The adjoint isomorphism~\cite[Proposition~9.3.6, pg. 205]{Olsson} and the projection formula~\cite[Proposition~1.12]{Adjoint}, we have
\begin{equation}\label{eq_containment}
\begin{array}{rcl}
\Hom_{\Vect(Y,Q)}\left( f^* E, f^* F \right) & = & \Hom_{\Vect(X,P)}\left( E, f_* f^* F \right) \\
 & = & \Hom_{\Vect(X,P)}\left( E, F \otimes_{\mathcal{O}_{(X,P)}} f_*\mathcal{O}_{(Y,Q)} \right).
\end{array}
\end{equation}
We claim that
$$\mu_{P,\text{\rm max}} \left( F \otimes_{\mathcal{O}_{(X,P)}} \left( f_*\mathcal{O}_{(Y,Q)}/\mathcal{O}_{(X,P)} \right) \right) < \mu_{P}(F) = \mu_P(E).$$
By Proposition~\ref{prop_negative}, $f^*F \otimes_{\mathcal{O}_{(Y,Q)}} f^* \left( f_*\mathcal{O}_{(Y,Q)}/\mathcal{O}_{(X,P)} \right) \in \text{\rm Vect}(Y,Q)$ admits a filtration $\{V_i\}$ by sub-bundles so that the subsequent quotients $V_i/V_{i+1} = L_i$ are line bundles of negative $Q$-degrees. Then
$$\mu_{Q, \text{\rm max}} \left( f^*F \otimes_{\mathcal{O}_{(Y,Q)}} f^* \left( f_*\mathcal{O}_{(Y,Q)}/\mathcal{O}_{(X,P)} \right) \right) \leq \text{\rm max}_{i} \{ \mu_{Q, \text{\rm max}}\left( f^*F \otimes L_i \right)\} < \mu_{Q, \text{\rm max}}\left( f^*F \right).$$
From this, the claim follows. By Proposition~\ref{prop_properties_stacky}~\eqref{H:5}, we have
$$\text{ \rm Hom}_{\Vect(X,P)}\left( E, F \otimes_{\mathcal{O}_{(X,P)}} \left( f_*\mathcal{O}_{(Y,Q)}/\mathcal{O}_{X,P} \right) \right) = 0.$$
Now the result follows from the exact sequence
\begin{eqnarray*}
0 \longrightarrow \text{\rm Hom}_{\Vect(X,P)}\left( E, F \right) \longrightarrow \text{\rm Hom}_{\Vect(X,P)}\left( E, F \otimes_{\mathcal{O}_{(X,P)}} f_*\mathcal{O}_{(Y,Q)} \right) \longrightarrow \\
\longrightarrow \text{\rm Hom}_{\Vect(X,P)}\left( E, F \otimes_{\mathcal{O}_{(X,P)}} \left( f_* \mathcal{O}_{(Y,Q)}/\mathcal{O}_{(X,P)} \right) \right) \longrightarrow 0.
\end{eqnarray*}
\end{proof}

Using the above results, we first conclude that the orbifold slope stability is preserved under Galois genuinely ramified covers.

\begin{theorem}\label{thm_Galois_main}
Let $f \colon (Y,Q) \longrightarrow (X,P)$ be a cover of connected orbifold curve. Assume that $f_0$ is Galois, and $f$ is a genuinely ramified cover; see Definition~\ref{def_gen_ram}. For any $P$-stable vector bundle $E \in \Vect(X,P)$, the pullback $f^*E \in \Vect(Y,Q)$ is $Q$-stable.    
\end{theorem}

\begin{proof}
Suppose that $E \in \Vect(X,P)$ is $P$-stable. As usual, we denote the cover induced on the Coarse moduli curves by $f_0 \colon Y \longrightarrow X$. By Proposition~\ref{prop_properties_orbi_cover}~\eqref{s:b2}, the vector bundle $f^*E$ is $Q$-semistable. By virtue of Remark~\ref{rmk_iota_preserve} and Remark~\ref{rmk_geometric_results}\eqref{geo:1}, we may assume that $Q = f_0^*P$ and that $(X,P)$ is a geometric orbifold curve.

By Remark~\ref{rmk_natural_atlas_on_pullback}~\eqref{rmk:2}, $(Y,f_0^*P)$ is also a geometric orbifold curve. Then using the equivariant set up and respective equivalences of orbifold bundles with the equivariant bundles, the notion of an extended socle of \cite[Definition~1.5.6]{HL} is also well defined for orbifold bundles. As in the case of curves, the extended socle is invariant under automorphisms of the Coarse moduli curve, and a simple orbifold semistable bundle that is equal to its extended socle is orbifold stable; cf. \cite[Lemma~1.5.9]{HL}. So if $E$ is $P$-stable, the pullback bundle $f^*E$ is $f_0^*P$-polystable by Proposition~\ref{prop_properties_orbi_cover} that is also simple by Corollary~\ref{cor_main} and equals to its extended socle. Thus, $f^*E$ is $f_0^*P$-stable.
\end{proof}

Finally, we are ready to prove that the slope stability conditions are preserved under a genuinely ramified cover.

\begin{theorem}\label{thm_main}
Let $f \colon (Y,Q) \longrightarrow (X,P)$ be a cover of connected orbifold curve. Assume that $f$ is genuinely ramified; see Definition~\ref{def_gen_ram}. For any $P$-stable vector bundle $E \in \Vect(X,P)$, the pullback $f^*E \in \Vect(Y,Q)$ is $Q$-stable.
\end{theorem}

\begin{proof}
Let $E \in \Vect(X,P)$ be $P$-stable. We have the induced cover $f_0 \colon Y \longrightarrow X$. As in the proof of Theorem~\ref{thm_Galois_main}, we may assume that $Q = f_0^*P$ and that $(X,P)$ is a geometric orbifold curve.

Let $\bar{f}_0 \colon \bar{Y} \longrightarrow X$ be its Galois closure, and $\bar{f}_0$ factors as
\[
\bar{f}_0 \colon \bar{Y} \overset{\hat{f}_0} \longrightarrow Y \overset{f_0} \longrightarrow X.
\]
Since $P$ is a geometric branch data on $X$, by Remark~\ref{rmk_natural_atlas_on_pullback}~\eqref{rmk:2}, $\bar{Q} = \bar{f}_0^*P$ is a geometric branch data on $\bar{Y}$. Then $\bar{f}_0$ induces a cover $\bar{f} \colon (\bar{Y}, \bar{Q}) \longrightarrow (X,P)$ that need not be genuinely ramified. By Proposition~\ref{prop_maximal_etal_subcover}, there is a maximal sub-cover $g_0 \colon \hat{X} \longrightarrow X$ of $\bar{f}_0$ such that $g_0$ is an essentially \'{e}tale cover of $(X,P)$, and the Galois cover $\bar{f}$ factors as a composition
$$(\bar{Y},\bar{Q}) \overset{\hat{g}} \longrightarrow (\hat{X}, g_0^*P) = \underline{\text{\rm Spec}}\left( \text{\rm HN}(\bar{f}_* \mathcal{O}_{(\bar{Y}, \bar{Q})})_1 \right) \overset{g}\longrightarrow (X,P)$$
where $g$ is the maximal \'{e}tale sub-cover of $\bar{f}$; moreover, $\hat{g}$ is a genuinely ramified cover by Remark~\ref{rmk_factorization_discussion} and Proposition~\ref{prop_gen_ram_equivalences}. Since $\bar{f}_0$ is Galois, by the maximality of $g_0$, the cover $g_0$ is Galois for some group $G$. Set $\hat{P} \coloneqq g_0^*P$. Further, since $f$ is genuinely ramified and $g$ is \'{e}tale, by Proposition~\ref{prop_gen_ram_equivalences}\eqref{equiv:3'}, the fiber product stacky curve $(Y,f_0^*P) \times_{(X,P)} (\hat{X},\hat{P})$ is connected. Moreover, since the projection morphism $p_1 \colon (Y,f_0^*P) \times_{(X,P)} (\hat{X},\hat{P}) \longrightarrow (Y, f_0^*P)$ is an \'{e}tale cover, $(Y,f_0^*P) \times_{(X,P)} (\hat{X},\hat{P})$ is an orbifold curve $(\hat{Y}, p_{1, 0}^* f_0^*P)$ where $p_{1,0} \colon \hat{Y} \longrightarrow Y$ is the $G$-Galois cover, induced by $p_1$ on the Coarse moduli curves, that is an essentially \'{e}tale cover of $(Y,f_0^*P)$. Setting $q \coloneqq f \circ p_1$ and $q_0 \coloneqq f_0 \circ p_{1,0}$, we have $p_{1,0}^*f_0^*P = q_0^*P$.

Summarizing the above, we have the following commutative diagram. 
\begin{center}
\begin{tikzcd}
(\bar{Y},\bar{Q})
\arrow[drr, bend left, "\hat{g}"]
\arrow[ddr, bend right, "\hat{f}"]
\arrow[dr,  "\nu"] & & \\
& (\hat{Y}, q_0^*P) = (Y,f_0^*P) \times_{(X,P)} (\hat{X},\hat{P}) \arrow[dr, phantom, "\square"] \arrow[r, "p_2"] \arrow[d, "p_1"]
& (\hat{X},g_0^*P) = (\hat{X}, \hat{P}) \arrow[d, "g"] \\
& (Y, f_0^*P) \arrow[r, "f"]
& (X,P)
\end{tikzcd}
\end{center}

Since $E$ is $P$-stable, Proposition~\ref{prop_properties_orbi_cover}~\eqref{s:b4} states that $g^*E$ is $\hat{P}$-polystable. So, $g^*E$ is a finite direct sum $\oplus_i \, F_i$ of $\hat{P}$-stable bundles with $\mu_{\hat{P}}(F_i) = \mu_{\hat{P}}(g^*E)$. By Theorem~\ref{thm_Galois_main}, each $\hat{g}^*F_i$ is $\bar{Q}$-stable. So, for each $i$, we have
\[
\hat{g}^*F_i \cong \nu^* p_2^*F_i.
\]
By Proposition~\ref{prop_properties_orbi_cover}~\eqref{s:b3}, we conclude that each $p_2^*F_i$ is $q_0^*P$-stable. So, $q^*E$ and $\bar{f}^*E$ are $q_0^*P$-polystable and $\bar{Q}$-polystable, respectively. By Proposition~\ref{prop_properties_orbi_cover}~\eqref{s:b2}, $f^*E$ is $f_0^*P$-semistable.

Now, let $0 \neq S \subseteq f^*E$ be a $f_0^*P$-stable sub-bundle of the $f_0^*P$-polystable vector bundle $f^*E$ such that $\mu_{f_0^*P}(S) = \mu_{f_0^*P}(f^*E)$. We will show that $f^*E = S$. For this, we will first construct a sub-bundle $V \subseteq g^*E$ having the same $\mu_{\hat{P}}$ such that $p_2^*V = p_1^*S$. Using the fact that $p_1^*S \subseteq (f \circ p_1)^*E$ is a $G$-invariant inclusion, we will descend the bundle $V$ to a sub-bundle of $E$ on $(X,P)$ having the same $\mu_P$, and this will conclude the proof.

First, taking the pullback under the Galois morphism $\hat{f} \colon (\bar{Y},\bar{Q}) \longrightarrow (Y, f_0^*P)$, we obtain a sub-bundle $\hat{f}^* S \subseteq \hat{f}^* f^* E = \bar{f}^* E$. By Proposition~\ref{prop_properties_stacky}~\eqref{s:b1}, we have $\mu_{\bar{Q}}(\hat{f}^* S) = \mu_{\bar{Q}}(\bar{f}^* E)$. Since $S$ is a $f_0^*P$-stable bundle on $(Y,f_0^*P)$ and $E$ is a $P$-stable bundle on $(X,P)$, both $\hat{f}^* S$ and $\bar{f}^*E$ are $\bar{Q}$-polystable by Proposition~\ref{prop_properties_orbi_cover}~\eqref{s:b4}. Define the right ideal $\bar{\Theta}$ of the associative algebra $\text{\rm End}_{\Vect(\bar{Y},\bar{Q})}(\bar{f}^*E)$ by
$$\bar{\Theta} \coloneqq \{ \gamma \in \text{\rm End}_{\Vect(\bar{Y},\bar{Q})}(\bar{f}^*E) \hspace{.2cm} | \hspace{.2cm} \gamma(\bar{f}^* E) \subset \hat{f}^*S \}.$$
Since $\hat{f}^*S$ is a direct summand of $\bar{f}^*E$, the bundle $\hat{f}^* S$ coincides with the vector sub-bundle of $\bar{f}^*E$ generated by the images of the endomorphisms in $\bar{\Theta}$.

Similarly, via the pullback under the $G$-Galois \'{e}tale cover $p_1 \colon (\hat{Y},q_0^*P) \longrightarrow (Y,f_0^*P)$, we obtain a sub-bundle $p_1^* S \subseteq q^* E$ on $(\hat{Y},q_0^*P)$ such that $\mu_{q_0^*P}(p_1^*S) = \mu_{q_0^*P}(q^*E)$. We already saw that $q^*E$ is $q_0^*P$-polystable, and $p_1^* S$ is $q_0^*P$-polystable by Proposition~\ref{prop_properties_orbi_cover}~\eqref{s:b5}. So, $p_1^*S$ is generated by the images of endomorphisms in the right ideal $\hat{\Theta} \subset \text{\rm End}_{\Vect(\hat{Y},q_0^*P)}(q^*E)$ defined as
$$\hat{\Theta} \coloneqq \{ \gamma' \in \text{\rm End}_{\Vect(\hat{Y},q_0^*P)}(q^*E) \hspace{.2cm} | \hspace{.2cm} \gamma'(q^*E) \subset p_1^*S \}.$$

Applying Corollary~\ref{cor_main} to the genuinely ramified Galois map $\hat{g} \colon (\bar{Y},\bar{Q}) \longrightarrow (\hat{X},\hat{P})$, we obtain
\begin{equation}\label{eq_7}
\text{\rm End}_{\Vect(\bar{Y},\bar{Q})}(\bar{f}^*E) = \text{\rm End}_{\Vect(\hat{X},\hat{P})}(g^* E).
\end{equation}
As an element $\gamma \in \text{\rm End}_{\Vect(\hat{X},\hat{P})}(g^* E)$ is mapped to $\hat{g}^* \gamma \in \text{\rm End}_{\Vect(\bar{Y},\bar{Q})}(\bar{f}^*E)$ under $\hat{g}^*$, the associative algebra structures are preserved. Let the right ideal $\hat{\Theta}' \subset \text{\rm End}_{\Vect(\hat{X},\hat{P})}(g^* E)$ be the image of $\bar{\Theta}$. Since $g^*E$ is a $\hat{P}$-polystable bundle on $(\hat{X}, \hat{P})$, the image of any endomorphism of it is a sub-bundle. Let $V$ be the sub-bundle of $g^* E$ generated by the images $\gamma(g^* E)$ for $\gamma \in \hat{\Theta}'$. Then we have $\hat{g}^* V = \hat{f}^*S$. Moreover, $V$ is again a $\hat{P}$-polystable bundle on $(\hat{X},\hat{P})$, and by Lemma~\ref{lem_pullback_div_line_bundles}, we have
$$\mu_{\hat{P}} \left( V \right) = \mu_{\hat{P}} \left( g^* E \right).$$

We have the following inclusions (note that Equation~\eqref{eq_containment} holds for any cover)
$$\text{\rm End}_{\Vect(\hat{X},\hat{P})}(g^* E) \subseteq \text{\rm End}_{\Vect(\hat{Y},q_0^*P)}(q^*E) \subseteq \text{\rm End}_{\Vect(\bar{Y},\bar{Q})}(\bar{f}^*E).$$
By Equation~\eqref{eq_7}, each of the above containment is an equality, and the associative algebra structures are preserved. Thus the ideal $\hat{\Theta}'$ maps onto the ideal $\hat{\Theta}$ which maps onto the ideal $\bar{\Theta}$. Since $V$ is generated by the images of the endomorphisms in $\hat{\Theta}'$, and $p_1^*S$ is generated by the images of the endomorphisms in $\hat{\Theta}$, we have
\begin{equation}\label{eq_8}
p_1^*S = p_2^* V.
\end{equation}

Since $p_1$ is a $G$-Galois \'{e}tale cover, the injective morphism $p_1^*S \subseteq q^* E$ is $G$-equivariant. By our construction, the injective morphism $V \subseteq g^*E$ is also $G$-equivariant. As $g$ is a representable $G$-Galois cover, Lemma~\ref{lem_push_exactness_rep_cover} show that there is a sub-bundle $W \subseteq E$ on $(X,P)$ such that $g^*W = V$. Thus $f^*W = S$. By Proposition~\ref{prop_properties_stacky}\eqref{s:b1}, we have
$$\mu_P(W) = \mu_P(E).$$
Since $E$ was assumed to be $P$-stable, we obtain $W = E$, and consequently, $S = f^* E$.
\end{proof}

Now let us see that the converse of the above theorem is also true.

\begin{theorem}\label{thm_converse}
Let $f \colon (Y,Q) \longrightarrow (X,P)$ be a non-trivial \'{e}tale cover of connected orbifold curves. There exists a $P$-stable vector bundle $E \in \text{\rm Vect}(X,P)$ such that $f^*E$ is not $Q$-stable.
\end{theorem}

\begin{proof}
We first observe that if $u_0 \colon Z \longrightarrow X$ is a $G$-Galois cover of smooth projective connected $k$-curves for some non-trivial finite group $G$, then there is a $G$-equivariant vector bundle $\mathcal{E}$ on $Z$ that is $G$-stable, but not stable. We saw such an example in Remark~\ref{rmk_G-stability_vs_stability} when $G$ is non-abelian, but our following construction will be used to address the general case as in the statement. Fix a closed point $x \in X$ in the \'{e}tale locus of $u_0$, and fix a point $z_1 \in u_0^{-1}(x)$. By \cite[Equation~(3.2)]{Das2}, the $G$-set $u_0^{-1}(x)$ equipped with the simply transitive action of $G$ is in a bijective correspondence with $G$; we have $u_0^{-1}(x) = \{z_g\}_{g \in G}$ where we write $1$ for the identity element of $G$. Set $\mathcal{E} \coloneqq \bigoplus_{g \in G} \mathcal{O}_Z(-z_g)$, equipped with a $G$-action that permutes the line sub-bundles, i.e. each $h \in G$ defines an automorphism of $\mathcal{E}$ by sending $\mathcal{O}_Z(-z_g)$ isomorphically onto $\mathcal{O}_Z(-z_{hg})$. Then $\mathcal{E}$ is a polystable vector bundle of rank $|G|$ and degree $-|G|$.

Let $u \colon Z \longrightarrow [Z/G] = (X,P)$ be the natural atlas. Then by the above construction, we have $\mathcal{E} \cong u^*u_* \mathcal{O}_Z(-z_1)$. By Remark~\ref{rmk_G-stability_vs_stability}, the vector bundle $u_* \mathcal{O}_Z(-z_1)$ is $P$-polystable. Equivalently, $\mathcal{E}$ is $G$-polystable. We claim that $\mathcal{E}$ is $G$-stable, but not $H$-stable for any subgroup $H \lneq G$.

Suppose that $\mathcal{F} \subset \mathcal{E}$ be a non-zero $G$-stable sub-bundle such that $\mu(\mathcal{F}) = \mu(\mathcal{E}) = -1$. Forgetting the $G$-actions for the moment, $\mathcal{F}$ is a non-zero sub-bundle of $\mathcal{E} \cong \mathcal{O}_Z(-z_1) \otimes_k \text{\rm Hom}(\mathcal{O}_Z(-z_1) , \mathcal{E}) \cong \mathcal{O}_Z(-z_1) \otimes_k k[G]$. Twisting by the line bundle $\mathcal{O}_Z(z_1)$, we see that $\mathcal{F} \otimes \mathcal{O}_Z(z_1)$ is a non-zero sub-bundle of $\mathcal{O}_Z \otimes_k k[G]$ of degree $0$. Then there is a rational map from $Z$ to the Grassmannian of sub-spaces of $k^{\oplus |G|}$ of dimension $a \coloneqq \text{\rm rk}\left( \mathcal{F} \otimes \mathcal{O}_Z(z_1) \right) = \text{\rm rk}\left( \mathcal{F} \right)$, and $\mathcal{F} \otimes \mathcal{O}_Z(z_1)$ corresponds to the pullback of the universal bundle on the Grassmannian under the rational map. So $\mathcal{F} \otimes \mathcal{O}_Z(z_1)$ is of the form $\mathcal{O}_Z \otimes_k V$ for a $k$-sub-space $V \subset k[G]$ of dimension $a$. Consequently, $\mathcal{F} \cong \mathcal{O}_Z(-z_1) \otimes_k V$. Since $\mathcal{F}$ is also a $G$-equivariant sub-bundle of $\mathcal{E}$, the bundle $\mathcal{F}$ is of the form
$$\mathcal{F} \cong \bigoplus \mathcal{O}_Z(-z_g)$$
where $g$ varies in a subset of $G$ of size $a$. Since $\mu(\mathcal{F}) = -1$, and the $G$-orbit of a point $z_g$ has size $|G|$, we conclude that $\text{\rm rk}(\mathcal{F}) = \text{\rm rk}(\mathcal{E})$, and hence $\mathcal{F} = \mathcal{E}$. This proves that $\mathcal{E}$ is $G$-stable.

For any sub-group $H \lneq G$, we have an $H$-equivariant sub-bundle $\mathcal{E}_H \coloneqq \bigoplus_{h \in H} \mathcal{O}_Z(-z_{h})$ of $\mathcal{E} \in \text{\rm Vect}^H(Z)$. By our above discussion, $\mathcal{E}_H$ is $H$-stable, and $\mu(\mathcal{E}_H) = -1 = \mu(\mathcal{E})$. This shows that $\mathcal{E}$ cannot be $H$-stable for any proper subgroup $H$ of $G$. We also conclude that there is a $P$-stable vector bundle $E \cong u_*^G \left( \mathcal{E} \right)$ under the equivalence $\text{\rm Vect}(X,P) \cong \text{\rm Vect}^G(Z)$, such that $u^*E \cong \mathcal{E}$ is not stable. Now we move to the general case as in the statement of the theorem. With our above observation, we may assume that $P$ is a non-trivial branch data on $X$.

Since $f$ is an \'{e}tale cover, by Lemma~\ref{lem_rep_etale_cover}, the cover $f_0 \colon Y \longrightarrow X$ of the Coarse moduli curves is an essentially \'{e}tale cover of $(X,P)$, and $Q = f_0^*P$. By Remark~\ref{rmk_geometric_results}, there is a maximal geometric branch data $P'$ on $X$ with $P \geq P'$, and every essential \'{e}tale cover of $(X,P)$ is also an essentially \'{e}tale cover of $(X,P')$. Thus the cover $f' \colon (Y,f_0^*P') \longrightarrow (X,P')$ is also an \'{e}tale cover. Suppose that there is a $P'$-stable vector bundle $E' \in \text{\rm Vect}(X,P')$ such that $f'^*E'$ is not $f_0^*P'$-stable. Writing $\iota \colon (X,P) \longrightarrow (X,P')$ and $j \colon (Y,f_0^*P) \longrightarrow (Y, f_0^*P')$ for the induced covers, $f^*\iota^*E' \cong j^* f'^*E'$ is not $f_0^*P$-stable by Proposition~\ref{prop_properties_stacky}\eqref{s:b3} and Remark~\ref{rmk_iota_preserve}. So it is enough to assume that $P$ is a geometric branch data on $X$. Then there is a $G$-Galois cover $g_0 \colon Z \longrightarrow X$ of smooth projective connected curves for some non-trivial finite group $G$, with $[Z/G] = (X,P)$, and $g_0$ factors as
$$g_0 \colon Z \overset{h_0}\longrightarrow Y \overset{f_0}\longrightarrow X.$$
The cover $h_0$ is necessarily Galois, for some subgroup $H \lneq G$ since $f_0$ is a non-trivial cover. We also see that $h \colon Z \longrightarrow (Y,f_0^*P)$ is an $H$-Galois \'{e}tale cover, and we also have the canonical map $u \colon Z \longrightarrow (X,P)$ that is a $G$-Galois \'{e}tale cover. From the first part of the proof, there is a $G$-stable vector bundle $\mathcal{E}$ on $Z$ that is not $H$-stable under the induced $H$-action. Under the equivalences
$$\text{\rm Vect}(X,P) \overset{u^*}\longrightarrow \text{\rm Vect}^G(Z),$$
$$\text{\rm and } \, \text{\rm Vect}(Y,f_0^*P) \overset{h^*}\longrightarrow \text{\rm Vect}^H(Z),$$
we obtain a vector bundle $E \in \text{\rm Vect}(X,P)$, with $u^*E \cong \mathcal{E}$, such that $f^*E$ is not $f_0^*P$-stable since $h^*f_0^*E \cong u^*E$ as $H$-equivariant bundles. This finishes the proof.
\end{proof}

\begin{remark}
Note that the above also gives an alternate proof of \cite[Proposition~5.2]{BP}.
\end{remark}

We have the following easier proofs in certain situations.

\begin{example}\label{eg_conv}
Let $f \colon (Y,Q) \longrightarrow (X,P)$ be a cover of connected orbifold curves. If the cover $f_0 \colon Y \longrightarrow X$ of underlying moduli curves is not genuinely ramified, there exists a $P$-stable vector bundle $E$ such that $f^*E$ is not $Q$-stable. To see this, note that by \cite[Proposition~5.1]{BP}, there is a stable vector bundle $E$ on $X$ such that $f_0^*E$ is not stable. Let $\iota \colon (X,P) \longrightarrow X$ and $j \colon (Y,Q) \longrightarrow Y$ be the respective Coarse moduli morphisms. Then by Remark~\ref{rmk_iota_preserve}, $\iota^*E \in \Vect(X,P)$ is $P$-stable. If $f^*\iota^*E \cong j^* f_0^*E$ is $Q$-stable, by loc. cit., $f_0^*E$ is a stable vector bundle on $Y$, a contradiction.
\end{example}

\begin{remark}[Comparison to other literature]\label{rmk_condition}
In \cite{BKP}, the authors have put necessary conditions on the orbifold curve $(X,P)$ such that the fiber product $Y \times_X (X,P)$ is again an orbifold curve and concluded Theorem~\ref{thm_main} in this case.

The authors of the above article generalize their result in \cite{BKP2} where they work over an algebraically closed field of characteristic $0$, and consider the parabolic bundles. Given a cover $f_0 \colon Y \longrightarrow X$ and an effective divisor $D = \sum n_x x$ on $X$, they define a sub-bundle $W \subset f_{0,*}Y$ containing $\mathcal{O}_X$. If $W$ is a line bundle, they show that for every parabolic stable bundle on $X$ whose parabolic weights at $x \in \text{\rm Supp}(D)$ are integral multiples of $1/n_x$, the pullback bundle on $Y$ is parabolic stable where the parabolic structure is inherited from $X$. It is not clear to us whether this bundle $W$ coincides with the vector bundle $\iota_*\left( \text{\rm HN}(f_*\mathcal{O}_{(Y,Q)}) \right)$ on $X$, when we appropriately identify the parabolic bundles with orbifold bundles as in Remark~\ref{rmk_parabolic_slope_same_as_P_slope}.
\end{remark}

\appendix

\section{Formal orbifold curves: a stacky view}\label{sec_equiv}
We lay out the outline of the folklore result, known to the experts (mentioned in \cite[Appendix~B, Theorem~B.1]{Mitsui}) that a formal orbifold curve (see Definition~\ref{def_f_o_c}) is the same as an orbifold curve (see Definition~\ref{def_orbi_curve}). Moreover, the categories of vector bundles on them coincide. This view lets us compare covers on the respective categories, define the orbifold slope stability conditions in an intrinsic way, and analyze them under the covers of orbifold curves.

First, we see that there is functorial assignment $\alpha$ of a formal orbifold curve to an orbifold curve. Let $X$ be a smooth projective connected $k$-curve and $P$ be a branch data on $X$. When $P$ is the trivial branch data, we set $\alpha(X) \coloneqq X$. Now suppose that $P$ is a non-trivial branch data. Then $\text{\rm BL}(P) = \{x_1, \ldots, x_r\}$, $r \geq 1$. Set $U_0 \coloneqq X - \text{\rm BL}(P)$. By \cite[Proposition~4.13]{K}, there exists a Zariski open covering $\{U_i\}_{0 \leq i \leq r}$ of $X$ such that the following hold: (1) $x_i \in U_i$, and $x_i \not\in U_j$ for any $1 \leq i \leq r$; (2) there is a smooth irreducible affine curve $V_i$, equipped with an action of a finite group $G_i$, such that $V_i \longrightarrow U_i$ is a $G_1 = \text{\rm Gal}\left( k(V_i)/ k(U_i) \right)$-Galois cover of smooth irreducible $k$-curves that is \'{e}tale away from $x_i \in U_i$. Set $\alpha(X,P)$ to be the DM stack with Coarse moduli curve $X$, obtained by gluing the quotient stacks $[Z_i/G_i]$ over the covering $\{U_i\}$ of $X$. By the construction, $\alpha(X,P)$ is an orbifold curve whose stacky points are $x_i, \, 1 \leq i \leq r$. It is also clear from the construction that a cover $f \colon (Y,Q) \longrightarrow (X,P)$ of orbifold curves necessarily defines a cover $\alpha(f) \colon \alpha(Y,Q) \longrightarrow \alpha(X,P)$; cf. Definition~\ref{def_stacky_curve}, Definition~\ref{def_foc_maps} and Remark~\ref{rmk_cover_curves}.

Conversely, given an orbifold curve $\mathfrak{X}$, let $U$ denote the maximal open sub-curve. If $\iota \colon \mathfrak{X} \longrightarrow X$ is the Coarse moduli morphism, $\mathfrak{X} \times_X U \cong U$ by definition. We may assume that $U$ is a proper subset of $\mathfrak{X}$ as otherwise, we associate the curve $\lambda(\mathfrak{X}) = X$ with the trivial branch data on it to this stack. Then $\mathfrak{X} - U$ is the finite set of stacky points on $\mathfrak{X}$. Let $x$ be a stacky point (which corresponds to a residual gerbe in $\mathfrak{X}$ over a closed point $x \in X(k)$, also denoted by $x$; see discussion in the beginning of Section~\ref{sec_degree}). Let $G_x$ be the stabilizer group at $x$. By \cite[Theorem~11.13.1, page 230]{Olsson}, there is an \'{e}tale neighborhood $U_x$ of $x$ in $X$ and a finite $U_x$-scheme $V_x$, equipped with an action of the group $G_x$, such that $\mathfrak{X} \times_X U_x \cong [V_x/G_x]$. For any point $u_x \in U_x$ over $x \in X$, and a point $v_x \in V_x$ lying over $u_x$, the field extension $\text{\rm QF}\left( \widehat{\mathcal{O}}_{V_x,v_x} \right)/\text{\rm QF}\left( \widehat{\mathcal{O}}_{U_x,u_x} \right)$ is a $G_x$-Galois field extension. Since $U_x \longrightarrow X$ is an \'{e}tale neighborhood of $x$, we have an isomorphism $K_{X,x} = \text{\rm QF}\left( \widehat{\mathcal{O}}_{X,x} \right) \cong \text{\rm QF}\left( \widehat{\mathcal{O}}_{U_x,u_x} \right)$. This produces a $G_x$-Galois extension $P(x)/K_{X,x}$. Note that by shrinking $U_x$, we may assume that the morphism $V_x \longrightarrow U_x$ is a finite separable surjection of smooth affine curves, and $G_x$ is the Galois group of the corresponding extension of function fields. This shows that the $G_x$-Galois extension $P(x)/K_{X,x}$ does not depend on the choice of the points $u_x$ and $v_x$, upto field isomorphisms in a chosen separable closure of $K_{X,x}$. We set $\lambda(\mathfrak{X})$ to be the formal orbifold curve $(X,P)$ where $P$ is the branch data described above. Further, given a cover $f \colon \mathfrak{Y} \longrightarrow \mathfrak{X}$ of orbifold curves, we obtain a cover $f_0 \colon Y \longrightarrow X$ of the Coarse moduli curves. Consider a closed point $x \in X(k)$, and let $y \in f_0^{-1}(x) \in Y(k)$. Let $G_x$ and $G'_y$ be the respective stabilizer groups at the corresponding residual gerbes $x$ and $y$. Consider the pair $(U_x, V_x)$ as before where $U_x \longrightarrow X$ is an \'{e}tale neighborhood of $x$, $V_x \longrightarrow U_x$ a finite separable surjection of smooth affine curves, and $G_x$ is the Galois group of the corresponding extension of function fields. Similarly, we have a pair $(V'_y, U'_y)$ for $y \in Y$. Further shrinking the neighborhoods, we may assume that they are compatible with the given morphism $f$ of orbifold curves. More precisely, $f$ induces finite surjective separable morphisms $f_y \colon U'_y \longrightarrow U_x$ and $h_y \colon V'_y \longrightarrow V_x$ making the following diagram commute.
\begin{equation}\label{diag_3}
\begin{tikzcd}
V'_y \arrow[r, "h_y"] \arrow[d, swap, "g'_y"] & V_x \arrow[d, "g_x"]\\
U'_y \arrow[r, "f_y"] & U_x
\end{tikzcd}
\end{equation}
We conclude that $Q(y)$ contains the compositum $P(x) \cdot K_{Y,y}$. So $f \colon \mathfrak{Y} \longrightarrow \mathfrak{X}$ defines a morphism $\lambda(f) \colon (Y,Q) \longrightarrow (X,P)$ of connected formal orbifold curves (up to isomorphism).

From the above functorial constructions, we immediately conclude the following (also see \cite[Appendix~B, Theorem~B.1]{Mitsui} and \cite[Example~4.3]{Noohi}).

\begin{theorem}\label{thm_equiv}
The functor $\lambda$ defines an equivalence of categories
$$\begin{pmatrix}
\text{Isomorphism classes of } \\ \text{orbifold curves and their covers}
\end{pmatrix} \overset{\sim} \longrightarrow \begin{pmatrix}
\text{Isomorphism classes of }\\ \text{formal orbifold curves and covers} \end{pmatrix},$$
with a quasi-inverse given by $\alpha$. Under this equivalence, an \'{e}tale cover corresponds to an \'{e}tale cover, and $\lambda$ preserves the \'{e}tale fundamental groups.
\end{theorem}

In \cite[Definition~4.5]{KM}, the authors define a vector bundle on a formal orbifold curve $(X,P)$ as a vector bundle on $X$ satisfying natural local actions and compatibility. Using the \'{e}tale local description of a stacky curve as before, it is not hard to see (which we reserve for the reader) that the above equivalence of categories also produces the following equivalence.

\begin{theorem}\label{thm_equiv_bundles}
Let $\mathfrak{X}$ be an orbifold curve. Consider the formal orbifold curve $(X,P) = \lambda(\mathfrak{X})$, where $\lambda$ is the categorical equivalence defined in Theorem~\ref{thm_equiv}. We have an equivalence of categories
$$\Vect(\mathfrak{X}) \overset{\sim} \longrightarrow \Vect(X,P)$$
between the category $\Vect(\mathfrak{X})$ of vector bundles on the orbifold curve $\mathfrak{X}$ in the sense of \cite[Definition 7.18]{V} and the category $\Vect(X,P)$ of algebraic parabolic bundles on $X$ with branch data $P$ in the sense of \cite[Definition~4.6]{KM}.
\end{theorem}

We end this section by mentioning two useful results which have been used extensively throughout our article --- the projection formula and the base change theorem.

\begin{proposition}\label{prop_proj_bc}
\begin{enumerate}[leftmargin=*]
\item (Base Change Theorem; {\cite[Proposition~13.1.9, pg. 122]{LMB} or \cite[Proposition~A.1.7.4]{Brochard}}) Given a cartesian square
\begin{equation*}
\begin{tikzcd}
(Y,Q) \times_{(X,P)} (Z,R) \arrow[r, "\text{\rm pr}_2"] \arrow[d, "\text{\rm pr}_1"] \arrow[dr, "\square", phantom] & (Z,R) \arrow[d, "f"] \\
(Y,Q) \arrow[r, "g"] & (X,P)
\end{tikzcd}
\end{equation*}
where $f$ and $g$ are covers of orbifold curves, and a bundle $E$ on $(Z,R)$, for any $i \geq 0$, we have the following isomorphism by the flat base change theorem.
\[
g^*\mathrm{R}^if_*E \cong \mathrm{R}^i\text{\rm pr}_{1, *}\text{\rm pr}_2^*E.\label{bc}
\]
\item (Projection Formula; \cite[Proposition~1.12]{Adjoint}) For any cover $f \colon (Y,Q) \longrightarrow (X,P)$ of orbifold curves, and bundles $E \in \Vect(X,P), \, F \in \Vect(Y,Q)$, there is a canonical isomorphism
\[
f_* \left( f^* E \otimes F \right) \cong E \otimes f_*F.\label{pf}
\]
\end{enumerate}
\end{proposition}

\bibliographystyle{amsplain}

\begin{thebibliography}{10}

\bibitem{B} Biswas I., \text{Parabolic bundles as orbifold bundles}, Duke Math. J., vol. 88, no. 2, 1997, 305--325.

\bibitem{BP} Biswas I., Parameswaran A. J., \textit{Ramified Covering Maps and Stability of Pulled-back Bundles}, {International Mathematics Research Notices}, 05, doi: {10.1093/imrn/rnab062}, {2021}, arxiv preprint: arXiv:2102.08744 [math.AG].

\bibitem{BDP} Biswas I., Das S., Parameswaran A. J., \textit{Genuinely ramified maps and stable vector bundles}, {Internat. J. Math.}, vol. {33}, no. {5}, {Paper No. 2250039, 14}, doi: {10.1142/S0129167X22500392}, 2022.

\bibitem{BKP} Biswas I., Kumar M., Parameswaran A. J., \textit{Genuinely ramified maps and stability of pulled-back parabolic bundles}, {Indagationes Mathematicae}, {https://doi.org/10.1016/j.indag.2022.04.003}, 2022.

\bibitem{BKP2} Biswas I., Kumar M., Parameswaran A. J., \textit{On the stability of pulled back parabolic vector bundles}, J. Math. Sci. Univ. Tokyo 29 (2022), no. 3, 359--382.

\bibitem{Brochard} Brochard S., \textit{Champs alg\'{e}briques et foncteur de Picard}, arxiv preprint, arXiv:0808.3253 [math.AG], {
https://doi.org/10.48550/arXiv.0808.3253}, 2008.

\bibitem{Das} Das S., \textit{On the Inertia Conjecture and its Generalizations}, Isr. J. Math. 253, 157–-204, {\href{https://doi.org/10.1007/s11856-022-2359-6}{10.1007/s11856-022-2359-6}}, , 2023.

\bibitem{Das2} Das S., \textit{Galois Covers of Singular Curves in Positive Characteristics}, Accepted for publication in Isr. J. Math. on 09 December 2022, arXiv preprint, {\href{https://arxiv.org/abs/2203.11870}{arXiv:2203.11870}}, 2022.

\bibitem{DM} {Deligne, P. and Mumford, D.}, \textit{The irreducibility of the space of curves of given genus}, {Inst. Hautes \'{E}tudes Sci. Publ. Math.}, no. {36}, {1969}, {75--109}.

\bibitem{E} Eisenbud D., \textit{Commutative algebra with a view toward algebraic geometry}, Graduate Texts in Mathematics, 150, Springer-Verlag, New York, 1995, xvi+785.

\bibitem{E2} Eisenbud D., \textit{The geometry of syzygies}, {Graduate Texts in Mathematics}, vol. {229}, {A second course in commutative algebra and algebraic geometry}, {Springer-Verlag, New York}, {2005}, {xvi+243}.

\bibitem{Adjoint} {Fausk, H. and Hu, P. and May, J. P.}, \textit{Isomorphisms between left and right adjoints}, {Theory Appl. Categ.}, vol. {11}, {No. 4, 107--131}, {2003}.

\bibitem{Ha} R. Hartshorne, \textit{Algebraic geometry}, Graduate Texts in Mathematics, No. 52. Springer-Verlag, New York-Heidelberg, 1977.

\bibitem{HL} Huybrechts D., Lehn M., \textit{The Geometry of Moduli Spaces of Sheaves}, Second Ed., Cam. Uni. Press, 2010.

\bibitem{Katz} Katz Nicholas M., \textit{Local-to-global extensions of representations of fundamental groups}, Ann. Inst. Fourier (Grenoble), 36, 4, 69--106, 1986.

\bibitem{KP} Kumar M., Parameswaran A. J., \textit{Formal orbifolds and orbifold bundles in positive characteristic}, Internat. J. Math. 30, no. 12, 1950067, 20 pp., 2019.

\bibitem{KM} Kumar M., Majumder S., \textit{Parabolic bundles in positive characteristic}, J. Ramanujan Math. Soc. 33, no. 1, 1--36, 2018.

\bibitem{K} Kumar M., \textit{Ramification theory and formal orbifolds in arbitrary dimension}, Proc. Indian Acad. Sci. Math. Sci. 129, no. 3, Art. 38, 34 pp., 2019.

\bibitem{Kobin} Kobin A., \textit{Artin-Schreier root stacks}, Journal of Algebra, vol. 586, 1014--1052, 2021.

\bibitem{LMB} {Laumon, G\'{e}rard and Moret-Bailly, Laurent}, \textit{Champs alg\'{e}briques}, {Ergebnisse der Mathematik und ihrer Grenzgebiete. 3. Folge. A Series of Modern Surveys in Mathematics [Results in Mathematics and Related Areas. 3rd Series. A Series of Modern Surveys in Mathematics]}, vol. {39}, {Springer-Verlag, Berlin}, {2000}, {xii+208}.

\bibitem{Mitsui} Mitsui K., \textit{Homotopy exact sequences and orbifolds}, Algebra Number Theory, 9, no. 5, {2015}, 1089 - 1136.

\bibitem{Mumford} {Mumford D.}, \textit{Projective invariants of projective structures and applications}, {Proc. {I}nternat. {C}ongr. {M}athematicians ({S}tockholm, 1962)}, {Inst. Mittag-Leffler, Djursholm}, {1963}, {526--530}.

\bibitem{Nironi} Nironi F., \textit{Grothendieck Duality for Deligne-Mumford Stacks}, ArXiv preprint arXiv:0811.1955v2 [math.AG],  2009.

\bibitem{Noohi} Noohi B., \textit{Fundamental Groups and Algebraic Stack}, Journal of the Institute of Mathematics of Jussieu, vol. 3 ,1, pp. 69--103, 2004.

\bibitem{Olsson} Olsson M., \textit{Algebraic spaces and stacks}, American Mathematical Society Colloquium Publications, 62, American Mathematical Society, Providence, RI, 2016, xi+298.

\bibitem{P} Parmeswaran A. J., \textit{Parabolic coverings {I}: the case of curves}, J. Ramanujan Math. Soc., vol. 25, 3, 2010, 233--251.

\bibitem{SP} The {Stacks project authors}, \textit{Stack Project}, {\url{https://stacks.math.columbia.edu}}, 2023.

\bibitem{V} Vistoli A., \textit{Intersection theory on algebraic stacks and on their moduli spaces}, Invent. Math., 97, no. 3, 1989, 613--670.

\bibitem{VZB} Voight J., Zureick-Brown D., \textit{The canonical ring of a stacky curve}, {Mem. Amer. Math. Soc.}, vol. {277}, no. {1362}, doi: {10.1090/memo/1362}, {2022}, {v+144}.

\end{thebibliography}

\end{document}